\documentclass{amsart}
\usepackage{amsfonts,amssymb,stmaryrd,amscd,amsmath,latexsym,amsbsy}

\numberwithin{equation}{section}
\theoremstyle{plain}
\newtheorem{thm}{Theorem}[section]

\newtheorem{prop}[thm]{Proposition}

\newtheorem{defi}[thm]{Definition}
\newtheorem{lem}[thm]{Lemma}
\newtheorem{cor}[thm]{Corollary}
\newtheorem{eg}[thm]{{Example}}

\theoremstyle{remark}
\newtheorem{rema}[thm]{Remark}
\newtheorem{remarks}[thm]{Remarks}

\newcommand{\abv}[2]{\genfrac{}{}{0pt}{}{#1}{#2}}
\newcommand{\ad}{{\mbox{\upshape{ad}}}}
\newcommand{\Ad}{{\mbox{\upshape{Ad}}}}

\newcommand{\afrak}{{\mathfrak a}}

\newcommand{\Aut}{\mathrm{Aut}}
\newcommand{\Auth}{\mathrm{Aut}_{\mathrm{Hopf}}}
\newcommand{\bA}{{\mathbf{A}}}

\newcommand{\Btil}{{\tilde{B}}}
\newcommand{\bfrak}{{\mathfrak b}}

\newcommand{\C}{{\mathbb C}}

\newcommand{\N}{{\mathbb N}}
\newcommand{\cBcs}{\check{B}_{\uc,\us}}

\newcommand{\cB}{{\mathcal B}}
\newcommand{\cC}{{\mathcal C}}

\newcommand{\cD}{{\mathcal D}}
\newcommand{\cF}{{\mathcal F}}

\newcommand{\cG}{{\mathcal G}}

\newcommand{\cJ}{{\mathcal J}}

\newcommand{\cL}{{\mathcal L}}
\newcommand{\cM}{{\mathcal M}}

\newcommand{\cS}{{\mathcal S}}
\newcommand{\cU}{{\mathcal U}}
\newcommand{\cW}{{\mathcal W}}
\newcommand{\cZ}{{\mathcal Z}}

\newcommand{\Ctil}{\tilde{C}}

\newcommand{\End}{\mbox{End}}

\newcommand{\field}{{\mathbb K}}

\newcommand{\gfrak}{{\mathfrak g}}
\newcommand{\glfrak}{{\mathfrak{gl}}}

\newcommand{\GL}{\mathrm{GL}}

\newcommand{\hfrak}{{\mathfrak h}}
\newcommand{\hI}{\hat{I}}
\newcommand{\Htil}{\tilde{H}}
\newcommand{\Hom}{{\mathrm{Hom}}}

\newcommand{\id}{{\mbox{id}}}

\newcommand{\kfrak}{{\mathfrak k}}
\newcommand{\kow}{{\varDelta}}

\newcommand{\Mfrak}{{\mathfrak M}}

\newcommand{\nfrak}{{\mathfrak n}}

\newcommand{\ocC}{{\overline{\mathcal{C}}}}
\newcommand{\ocD}{{\overline{\mathcal{D}}}}
\newcommand{\ocS}{{\overline{\mathcal{S}}}}

\newcommand{\ovG}{{\overline{G}}}

\newcommand{\oI}{{\overline{I}}}
\newcommand{\oi}{{\overline{i}}}
\newcommand{\oj}{{\overline{j}}}
\newcommand{\oL}{{\overline{L}}}

\newcommand{\oKq}{{\overline{\field(q)}}}

\newcommand{\ot}{\otimes}
\newcommand{\Out}{\mathrm{Out}}

\newcommand{\qfield}{k}
\newcommand{\pfrak}{{\mathfrak p}}

\newcommand{\rank}{\mathrm{rank}}

\newcommand{\simHt}{\stackrel{\Htil}{\sim}}
\newcommand{\simS}{\stackrel{\cS}{\sim}}

\newcommand{\slfrak}{{\mathfrak{sl}}}
\newcommand{\sofrak}{{\mathfrak{so}}}
\newcommand{\spfrak}{{\mathfrak{sp}}}
\newcommand{\tfrak}{{\mathfrak t}}

\newcommand{\uc}{\mathbf{c}}
\newcommand{\ud}{\mathbf{d}}
\newcommand{\us}{\mathbf{s}}
\newcommand{\Uq}{U}

\newcommand{\uqg}{{U_q(\mathfrak{g})}}
\newcommand{\uqgh}{{\check{U}_q(\mathfrak{g})}}
\newcommand{\uqgp}{{U_q(\mathfrak{g'})}}

\newcommand{\uqLg}{{U_q(L(\mathfrak{g}))}}
\newcommand{\uqRglN}{{U^R_q(\hat{\mathfrak{gl}}_N)}}
\newcommand{\thetah}{\hat{\theta}}

\newcommand{\vep}{\varepsilon}
\newcommand{\wght}{\mathrm{wt}}

\newcommand{\Yt}{Y^{\mathrm{tw}}_q}
\newcommand{\Yto}{Y^{\mathrm{tw}}_q(\mathfrak{o}_N)}
\newcommand{\Ytsp}{Y^{\mathrm{tw}}_q(\mathfrak{sp}_{2m})}
\newcommand{\Ytspfour}{Y^{\mathrm{tw}}_q(\mathfrak{sp}_{4})}
\newcommand{\Z}{{\mathbb Z}}

\title{Quantum symmetric Kac-Moody pairs}
\author{Stefan Kolb}
\address{School of Mathematics and Statistics, Newcastle University,
   Herschel Building, Newcastle upon Tyne NE1 7RU, UK} 
\email{stefan.kolb@newcastle.ac.uk}


\thanks{The early stages of this work were supported by the Maxwell Institute for Mathematical Sciences, Edinburgh, and by the Netherlands Organization for Scientific Research (NWO) within the VIDI-project ``Symmetry and modularity in exactly solvable models". The final corrections to this paper were supported by EPSRC grant EP/K025384/1.}

\begin{document}

\begin{abstract}
  The present paper develops a general theory of quantum group analogs of symmetric pairs for involutive automorphism of the second kind of symmetrizable Kac-Moody algebras. The resulting quantum symmetric pairs are right coideal subalgebras of quantized enveloping algebras. They give rise to triangular decompositions, including a quantum analog of the Iwasawa decomposition, and they can be written explicitly in terms of generators and relations. Moreover, their centers and their specializations are determined. The constructions follow G.~Letzter's theory of quantum symmetric pairs for semisimple Lie algebras. The main additional ingredient is the classification of involutive automorphisms of the second kind of symmetrizable Kac-Moody algebras due to Kac and Wang. The resulting theory comprises various classes of examples which have previously appeared in the literature, such as $q$-Onsager algebras and the twisted $q$-Yangians introduced by Molev, Ragoucy, and Sorba. 
\end{abstract}

\maketitle

\tableofcontents

\section{Introduction}
Let $\gfrak$ be a symmetrizable Kac-Moody algebra defined over an algebraically closed field $\field$ of characteristic $0$. The quantized enveloping algebra $\uqg$ of $\gfrak$, discovered by Drinfeld \cite{inp-Drinfeld1} and Jimbo \cite{a-Jimbo1} nearly thirty years ago, is an integral part of representation theory with many deep applications. Let $\theta:\gfrak\rightarrow \gfrak$ be an involutive Lie algebra automorphism and let $\gfrak=\kfrak\oplus \pfrak$ be the decomposition of $\gfrak$ into the $(+1)$- and the $(-1)$-eigenspace of $\theta$. The present paper is concerned with the construction and the structure theory of quantum group analogs of $U(\kfrak)$ as right coideal subalgebras $B=B(\theta)$ of $\uqg$. We call such analogs quantum symmetric pair coideal subalgebras and refer to $(\uqg, B)$ as a quantum symmetric pair.

If $\gfrak$ is finite dimensional then there exist two approaches to the construction of quantum symmetric pairs. In the early nineties, Noumi, Sugitani, and Dijkhuizen constructed quantum group analogs of $U(\kfrak)$ as coideal subalgebras of $\uqg$ for all $\gfrak$ of classical type \cite{a-Noumi96}, \cite{a-NS95}, \cite{a-NDS97}. Their approach is based on explicit solutions of the reflection equation \cite{a-Cher84}, \cite{a-KuSkl92} and hence follows in spirit the methods developed by the (then) Leningrad school of mathematical physics \cite{a-FadResTak2}. Independently, G.~Letzter developed a comprehensive theory of quantum symmetric pairs for all semisimple symmetric Lie algebras \cite{a-Letzter97}, \cite{a-Letzter99a}, \cite{a-Letzter00}, \cite{MSRI-Letzter}. \cite{a-Letzter03}, \cite{a-Letzter04}, \cite{a-Letzter-memoirs}. Her theory is based on the Drinfeld-Jimbo presentation of quantized enveloping algebras. It is well understood that the two approaches to quantum symmetric pairs provide essentially the same coideal subalgebras of $\uqg$, see \cite[Section 6]{a-Letzter99a}.

Over the last decade, numerous examples of quantum symmetric pairs for infinite dimensional symmetrizable Kac-Moody algebras have appeared in the literature. Here we group these examples in three classes.

\medskip

\noindent{\bf (1) $q$-Onsager algebras:} The $q$-Onsager algebra is a quantum symmetric pair coideal subalgebra for the Chevalley involution of the affine Lie algebra $\hat{\slfrak}_2(\field)$. It derives its name from the fact that the Lie subalgebra of $\hat{\slfrak}_2(\C)$ fixed under the Chevalley involution appeared in Onsager's investigation of the planar Ising model \cite{a-Onsager44}, see also \cite[Remark 9.1]{a-Terwilliger06} for historical comments. The $q$-Onsager algebra appeared, as an algebra, in Terwilliger's investigation of tridiagonal pairs and polynomial association schemes \cite[Lemma 5.4]{a-Terwilliger93}. The name $q$-Onsager algebra, however, goes back to Baseilhac and Koizumi \cite{a-BasKoz05} who observed its role as a symmetry algebra for a class of quantum integrable models. An embedding of the $q$-Onsager algebra as a coideal subalgebra of $U_q(\hat{\slfrak}_2(\field))$ was established in \cite[Proposition 1.13]{a-ItoTerwilliger10}, see also \cite[Section 2]{a-Baseilhac05}. Recently, it was proposed to study quantum symmetric pair coideal subalgebras corresponding the Chevalley involution for arbitrary affine Kac-Moody algebras \cite{a-BasBe10} under the name generalized $q$-Onsager algebras. These algebras previously appeared in \cite[Section 3]{a-DeliusGeorge02} and \cite[3.4]{a-DeliusMacKay03}.

\medskip

\noindent{\bf (2) Twisted quantum loop algebras of the second kind:} Assume that $\gfrak$ is finite dimensional and let $L(\gfrak)=\gfrak\ot \field[t,t^{-1}]$ denote the corresponding loop algebra. The involutive automorphism $\theta$ of $\gfrak$ lifts to an involutive automorphism $\hat{\theta}_L$ of $L(\gfrak)$ defined by $\hat{\theta}_L(x\ot t^n)=\theta(x)\ot t^{-n}$ for all $x\in \gfrak$, $n\in \Z$. The same formula also extends $\theta$ to an involutive automorphism $\hat{\theta}$ of the untwisted affine Lie algebra $\hat{\gfrak}$ of $\gfrak$. The fixed Lie subalgebras of $L(\gfrak)$ and $\hat{\gfrak}$ are isomorphic and we call them twisted loop algebras of the second kind. Correspondingly, any quantum symmetric pair coideal subalgebra for the involution $\hat{\theta}$ will be called a twisted quantum loop algebra of the second kind. The twisted quantum loop algebras (of the second kind) corresponding to the Chevalley involution coincide with the generalized $q$-Onsager algebras discussed in (1). The most prominent examples of twisted quantum loop algebras, however, were introduced by Molev, Ragoucy, and Sorba \cite{a-MolRagSor03} for the involutions corresponding to the symmetric pairs $(\slfrak_N(\field),\sofrak_N(\field))$ and $(\slfrak_{2m}(\field),\spfrak_{2m}(\field))$. Molev, Ragoucy, and Sorba call their algebras twisted $q$-Yangians. A twisted quantum loop algebra corresponding to the symmetric pair $(\slfrak_N(\field),\slfrak_N(\field)\cap(\glfrak_r(\field)\oplus \glfrak_{N-r}(\field)))$ was constructed in \cite{a-CGM14}. Both \cite{a-MolRagSor03} and \cite{a-CGM14} use an FRT approach for their constructions and work with $\glfrak_N(\field)$ rather than $\slfrak_N(\field)$. Examples of quantum symmetric pair coideal subalgebras of $U_q(\hat{\slfrak}_2(\field))$ which are not twisted quantum loop algebras were recently constructed in \cite{a-Regelskis12p}.

\medskip

\noindent{\bf (3) Quantized GIM Lie algebras:} Generalized intersection matrix (GIM) Lie algebras, originally introduced by Slodowy \cite{habil-Slodowy}, \cite{a-Slodowy86}, form a class of Lie algebras which generalize Kac-Moody algebras by allowing Cartan matrices with positive off-diagonal entries. It was realized by Berman \cite{a-Berman89} that GIM Lie algebras are subalgebras fixed under an involution of certain symmetrizable Kac-Moody algebras. Recently, quantum group analogs of GIM Lie algebras associated to two-fold affinizations of Cartan matrices of type ADE were constructed by Tan \cite{a-Tan05} in an attempt to better understand the quantum toroidal algebras defined in \cite{a-GKV95}. In \cite{a-LvTan12} the resulting quantized GIM Lie algebras were realized as coideal subalgebras of quantized enveloping algebras. Quantized GIM Lie algebras provide examples of quantum symmetric pairs for non-affine symmetrizable Kac-Moody algebras.

\medskip

In the present paper a comprehensive theory of quantum symmetric pairs for symmetrizable Kac-Moody algebras is developed. This theory comprises all of the above example classes and reduces to Letzter's theory if $\gfrak$ is finite dimensional. The paper develops algebraic properties of the resulting quantum symmetric Kac-Moody pairs. This includes triangular decompositions, in particular an analog of the Iwasawa decomposition. It also includes explicit presentations of quantum symmetric pair coideal subalgebras in terms of generators and relations and a description of their centers. Moreover, we investigate the transformation behavior of quantum symmetric pairs under Hopf algebra automorphisms of $\uqg$ and the behavior of quantum symmetric pairs in the limit $q\rightarrow 1$, commonly referred to a specialization. It has to be stressed that for finite dimensional $\gfrak$ most of these results were obtained by Letzter and that the proofs in the Kac-Moody case are for the most part straightforward translations of her arguments. Nevertheless, we feel that a detailed presentation of the construction of quantum symmetric pairs and of the proofs of their properties is appropriate. Firstly, this guarantees that everything does indeed translate to the Kac-Moody setting. This fact was not used by any of the papers referred to in (1), (2), and (3) above (except \cite{a-Regelskis12p}), although it significantly conceptualizes and generalizes those constructions. This will become apparent in the final two sections of the present paper where the theory is applied to define quantum loop algebras and quantized GIM algebras in full generality. Delius and MacKay asked in \cite[p. 189]{a-DeliusMacKay03} for an affine generalization of Noumi's and Letzter's theories. In \cite[Section 4]{a-BasBe10} Baseilhac and Belliard asked for a Drinfeld-Jimbo realization of the twisted $q$-Yangians in \cite{a-MolRagSor03}. The present paper answers both these questions. 
Secondly, a detailed presentation allows for some minor changes to Letzter's original approach. This will hopefully help to convince the reader that quantum symmetric pairs appear very naturally and are technically not more involved than quantized enveloping algebras themselves.  

An important ingredient in Letzter's construction is the classification of involutive automorphisms of simple complex Lie algebras up to conjugation. Being equivalent to the classification of real simple Lie algebras, this classical problem was essentially solved by E.~Cartan \cite{a-Cartan14}. It was revisited by Araki \cite{a-Araki62} using Cartan subalgebras $\hfrak$ of $\gfrak$, invariant under $\theta$, such that $\hfrak\cap\pfrak$ is maximally abelian in $\pfrak$. Dijkhuizen observed in \cite{a-Dijk96} that this condition is crucial for the general construction of quantum symmetric pairs. In \cite{a-Letzter99a} Letzter specifies the choice of $\theta$ even further. Let $\gfrak=\nfrak^+\oplus \hfrak \oplus \nfrak^-$ be the triangular decomposition of $\gfrak$ with respect to a fixed choice $\Pi=\{\alpha_i\,|\,i\in I\}$ of simple roots and let $e_i,f_i,h_i$ for $i\in I$ denote the Chevalley generators of $\gfrak$. Moreover, for any subset $X\subset I$ let $\gfrak_X$ denote the Lie subalgebra of $\gfrak$ generated by all $e_i,f_i,h_i$ for $i\in X$. To construct quantum symmetric pairs for finite dimensional $\gfrak$, Letzter assumes that 
\begin{align}\label{eq:Lcond1}
  \theta(\hfrak)=\hfrak
\end{align} 
and that there exists a subset $X\subset I$ such that
\begin{align}
  \theta|_{\gfrak_X} = \id_{\gfrak_X}
\end{align}\label{eq:Lcond2}
and
\begin{align}\label{eq:Lcond3}
  \theta(e_i) \in \nfrak^- \quad \mbox{and} \quad \theta(f_i) \in \nfrak^+ \qquad \mbox{if $i \in I\setminus X$.}
\end{align}
For finite dimensional $\gfrak$ the above conditions can be achieved for any involutive automorphism $\theta$ via conjugation. It turns out that Letzter's theory of quantum symmetric pairs can be extended to symmetrizable Kac-Moody algebras and involutions $\theta$ which satisfy \eqref{eq:Lcond1}-\eqref{eq:Lcond3} for a subset $X\subset I$ of \textit{finite} type. More precisely, recall that a Lie algebra automorphism $\sigma: \gfrak\rightarrow \gfrak$ is said to be of the second kind if the standard Borel subalgebras $\bfrak^+$ and $\bfrak^-$ of $\gfrak$ satisfy $[\bfrak^+:\sigma(\bfrak^-)\cap \bfrak^+]<\infty$, \cite[III.1]{a-Levstein88}. Any automorphism $\theta$ satisfying \eqref{eq:Lcond1}-\eqref{eq:Lcond3} for a subset $X\subset I$ of finite type is of the second kind. It is one of the main observations of the present paper that Letzter's theory extends to involutive automorphisms of the second kind of symmetrizable Kac-Moody algebras. 

Involutive automorphisms of the second kind were essentially classified by Kac and Wang \cite{a-KW92}. 
It follows from their results that any involution of the second kind is conjugate to an involution satisfying \eqref{eq:Lcond1}-\eqref{eq:Lcond3}. This involution can be written explicitly in terms of certain pairs $(X,\tau)$ where $X\subset I$ as above, and $\tau$ is a diagram automorphism of $I$. The pairs $(X,\tau)$ corresponding to involutive automorphisms of the second kind can be described solely in terms of the root system of $\gfrak$ and will be called admissible pairs. They are natural generalizations of Satake diagrams \cite[pp.~32/33]{a-Araki62} to the Kac-Moody setting. The explicit form of the involution $\theta(X,\tau)$ corresponding to the admissible pair $(X,\tau)$ immediately suggests the definition of a quantum group analog $\theta_q(X,\tau)$. Similarly, a set of generators of the Lie subalgebra $\kfrak$ fixed under $\theta(X,\tau)$ suggests a set of generators and hence the definition of the corresponding quantum symmetric pair coideal subalgebra.

In the following the content of the present paper is described in some more detail. The classification of involutions of the second kind of symmetrizable Kac-Moody algebras in terms of admissible pairs is recalled in Section \ref{sec:inv2nd}.
This section contains all facts about Kac-Moody algebras used subsequently for the construction of quantum symmetric pairs. Only the proof of the main result, Theorem \ref{classThm}, is moved to Appendix \ref{app:classThm}. Section \ref{sec:LusztigBraid} briefly recalls the definition and properties of quantized enveloping algebras defined over the field $\field(q)$ of rational functions in $q$. Here the focus is on the relation between Lusztig's braid group action on $\uqg$ and the adjoint action of $\uqg$ on itself. The aim is to find $q$-analogs of all the constituents of the involution $\theta(X,\tau)$ of $\gfrak$. With the preparations provided by Sections \ref{sec:inv2nd} and \ref{sec:LusztigBraid} at hand, it is straightforward to define the quantum group analog $\theta_q(X,\tau)$ of $\theta(X,\tau)$ as an $\field(q)$-algebra automorphism of $\uqg$ in Section \ref{sec:qInvolution}. This definition is one of the more essential differences between the present paper and Letzter's constructions for finite dimensional $\gfrak$ where the quantum group analog of $\theta$ maps $q$ to $q^{-1}$, see Subsection \ref{sec:ref-Letzter-qtheta}.
In Section \ref{sec:QSP} quantum symmetric pair coideal subalgebras $B_{\uc,\us}$ are defined depending on two families of additional parameters $\uc,\us\in \field(q)^{I\setminus X}$. The next aim is to describe $B_{\uc,\us}$ explicitly in terms of generators and relations. To this end, three crucial triangular decompositions are presented in Section \ref{sec:triang}. The first two decompositions provide bases of $\uqg$ and $B_{\uc,\us}$ in terms of certain monomials in generators of $B_{\uc,\us}$. The third decomposition is a quantum analog of the Iwasawa decomposition. These tools make it possible to give the desired presentation of $B_{\uc,\us}$ in Section \ref{sec:GensRels}. The relations are given explicitly for all generalized Cartan matrices with entries $\ge -2$, but the method works in full generality. Here the $q$-Onsager algebra appears in Example \ref{eg:qOnsager}. In Section \ref{sec:center} the center of $B_{\uc,\us}$ is determined. It is shown that if $B_{\uc,\us}$ corresponds to an indecomposable Cartan matrix of infinite type then the center $Z(B_{\uc,\us})$ is trivial. The center of $B_{\uc,\us}$ in the finite case was previously determined in \cite{a-KL08} and is here briefly recalled for completeness. Section \ref{sec:equivalence} analyzes the behavior of quantum symmetric pairs under the action of the group $\Aut_{\mathrm{Hopf}}(\uqg)$ of Hopf algebra automorphisms of $\uqg$. The aim is in particular to find a class of representatives of the $\Aut_{\mathrm{Hopf}}(\uqg)$-orbits of quantum symmetric pairs. In Section \ref{sec:specialize} we discuss non-restricted specialization, that is, the behavior of $\theta_q(X,\tau)$ and $B_{\uc,\us}$ in the limit $q\rightarrow 1$. Using specialization, it is in particular shown in Theorem \ref{thm:maxcond} that $B_{\uc,\us}$ satisfies a maximality condition. In these specialization arguments it is crucial that the triangular decompositions involving $B_{\uc,\us}$ remain true over the localization $\bA=\field[q,q^{-1}]_{(q-1)}$, see Theorem \ref{thm:A-triang}. The maximality property of $B_{\uc,\us}$ is used to identify the twisted $q$-Yangians from \cite{a-MolRagSor03} with quantum symmetric pair coideal subalgebras in Section \ref{sec:TwistedQLoop}. It should be noted that twisted $q$-Yangians contain an additional infinite dimensional central subalgebra due to the fact that they are constructed from $\glfrak_n(\field)$ and not from $\slfrak_n(\field)$. The final Section \ref{sec:QGIM} provides the construction of quantized GIM Lie algebras in full generality, going beyond the two-fold affinizations considered by Lv and Tan in \cite{a-Tan05}, \cite{a-LvTan12}. 

At the end of each section that emulates Letzter's constructions for finite dimensional $\gfrak$, we provide  detailed references to her work and point out any differences between her constructions and those presented here in the Kac-Moody case. This applies to Sections \ref{sec:qInvolution}, \ref{sec:QSP}, \ref{sec:triang}, \ref{sec:GensRels}, \ref{sec:equivalence}, \ref{sec:specialize}, and \ref{sec:TwistedQLoop}. 

\medskip

\noindent{\bf Acknowledgments.} The author is grateful to N.~Guay for pointing out the paper \cite{a-CGM14}, to Y.~Tan for providing the paper \cite{a-LvTan12} ahead of print, and to J.~Stokman for valuable comments. The initial work on this project was done in 2009 at the Universities of Edinburgh and Amsterdam. The results were announced at a miniworkshop on `Generalizations of Symmetric Spaces' in Oberwolfach in April 2012. The author is grateful to the organizers, Loek Helminck and Ralf Koehl, for this occasion which significantly furthered the completion of the paper. 

\section{Involutive automorphisms of the second kind}\label{sec:inv2nd}
Automorphisms of infinite-dimensional, symmetrizable Kac-Moody algebras belong to either of two disjoint classes, automorphisms of the first and automorphisms of the second kind. In principle, involutive automorphisms of the second kind were classified in \cite{a-KW92}. In the very final remark of their paper, Kac and Wang announce that a combinatorial description of involutions of the second kind in terms of Dynkin diagrams is possible. This section is the author's take on what Kac and Wang might have had in mind. This will, moreover, allow us to fix notations for Kac-Moody algebras and their automorphisms. The main result of this section, Theorem \ref{classThm}, should be attributed to \cite{a-KW92}, however, the author was unable to locate the given explicit formulation in the literature. In Appendix \ref{app:classThm} we will reduce Theorem \ref{classThm} to results stated in \cite{a-KW92}. A similar result is contained in \cite{a-BBBR95} where almost split real forms of symmetrizable Kac-Moody algebras are classified. 

All through this paper $\field$ denotes an algebraically closed field of characteristic $0$.
\subsection{Kac-Moody algebras and groups}\label{sec:KM}
Let $I$ be a finite set and let $A=(a_{ij})_{i,j \in I}$ be a generalized Cartan matrix. Recall that $A$ is a square matrix with integer entries such that $a_{ii}=2$ for all $i\in I$, $a_{ij}\le 0$ if $i\neq j$, and $(a_{ij}=0 \Leftrightarrow a_{ji}=0)$. We always assume $A$ to be symmetrizable, that is, there exists a diagonal matrix $D=\mbox{diag}(\epsilon_i|i\in I)$ with coprime entries $\epsilon_i\in \N$ such that $DA$ is symmetric. Moreover, we assume $A$ to be indecomposable. 
Let $P^\vee$ be the free abelian group of rank $2|I|-\mbox{rank}(A)$ with $\Z$-basis $\{h_i\,|\,i\in I\}\cup \{d_s\,|\, s=1, \dots,|I|-\mbox{rank}(A)\}$, set $\hfrak=\field\ot_\Z P^\vee$, and define 
$P=\{\lambda\in \hfrak^\ast\,|\,\lambda(P^\vee)\subseteq \Z\}$. 
As usual, we call $P$ the weight lattice associated to $A$. With the notation $P^\vee$ we follow \cite[2.1]{b-HongKang02} but note that $P^\vee$ is an extension of the dual root lattice. Moreover, define $\Pi^\vee=\{h_i\,|\,i\in I\}$ and choose a linearly independent subset $\Pi=\{\alpha_i\,|\,i\in I\}$ of $\hfrak^\ast$ such that 
\begin{align}\label{eq:alpha(h)}
  \alpha_j(h_i)=a_{ij} \qquad \mbox{and} \qquad \alpha_j(d_s)\in \{0,1\}
\end{align}
for all $i,j\in I$ and all $s=1, \dots,|I|-\mbox{rank}(A)$. In this case $(\hfrak,\Pi, \Pi^\vee)$ is a minimal realization of $A$. Let $Q=\Z \Pi$ and $Q^\vee=\Z \Pi^\vee$ denote the root lattice and the dual root lattice, respectively, and  set $Q^+=\sum_{i\in I}\N_0 \alpha_i$. For any $\mu,\nu\in \hfrak^\ast$ we write $\mu\ge \nu$ if $\mu-\nu\in Q^+$.

The Kac-Moody algebra $\gfrak=\gfrak(A)$ is the Lie algebra over $\field$ generated by $\hfrak$ and elements $e_i, f_i$ for $i\in I$ and the relations given in \cite[1.3]{b-Kac1}. The derived algebra $\gfrak'=[\gfrak,\gfrak]$ is the Lie subalgebra of $\gfrak$ generated by the elements $e_i, f_i$ for $i\in I$. Let $\nfrak^+$ and $\nfrak^-$ denote the Lie subalgebras of $\gfrak$ generated by the elements of the sets $\{e_i\,|\,i\in I\}$ and $\{f_i\,|\,i\in I\}$, respectively. One has triangular decompositions
\begin{align*}
  \gfrak=\nfrak^+ \oplus \hfrak \oplus \nfrak^- \qquad \gfrak'=\nfrak^+ \oplus \hfrak' \oplus \nfrak^-
\end{align*}
where $\hfrak'=\sum_{i\in I}\field h_i$. As a $\hfrak$-module $\gfrak$ decomposes into root spaces
\begin{align*}
  \gfrak=\bigoplus_{\beta\in \hfrak^\ast} \gfrak_\beta
\end{align*}
where $\gfrak_\beta=\{x\in \gfrak\,|\,[h,x]=\beta(h)x,\, \mbox{for all } h\in \hfrak\}$. Let $\Phi=\{\beta\in \hfrak^\ast\,|\,\gfrak_\beta\neq 0\}$ denote the set of roots and set $\Phi^+=\Phi\cap Q^+$.

For any $i\in I$ the fundamental reflection $r_i \in \GL(\hfrak)$ is defined by 
\begin{align*}
  r_i(h)=h-\alpha_i(h) h_i \qquad \mbox{for all } h\in \hfrak.
\end{align*}
The Weyl group $W$ is the subgroup of $\GL(\hfrak)$ generated by the fundamental reflections $r_i$.
It is a Coxeter group given by defining relations
\begin{align}\label{Coxeter}
  r_i^2=1, \qquad (r_ir_j)^{m_{ij}}=1
\end{align}
for all $i,j\in I$ where $m_{ij}=2,3,4,6,$ and $\infty$ if
$a_{ij}a_{ji}=0,1,2,3$, and $\ge 4$, respectively, \cite[Proposition 3.13]{b-Kac1}. Via duality $W$ also acts on $\hfrak^\ast$ by
\begin{align*}
  r_i(\alpha)=\alpha -\alpha(h_i)\alpha_i \qquad \mbox{for all } \alpha\in \hfrak^\ast
\end{align*}
and the root system $\Phi$ is stable under $W$. Let $\Phi^{re}=W\Phi$ and $\Phi^{im}=\Phi\setminus \Phi^{re}$ denote the sets of real and imaginary roots, respectively.

By \cite[2.1]{b-Kac1} there exists a nondegenerate, symmetric, bilinear form $(\cdot,\cdot)$ on $\hfrak$ such that
\begin{align*}
  (h_i,h)=\alpha_i(h)/\epsilon_i \quad \forall h\in \hfrak, i \in I,\qquad (d_m,d_n)=0 \quad\forall n,m\in \{1,\dots,s\}.
\end{align*}
We denote the induced bilinear form on $\hfrak^\ast$ by the same symbol and observe that
\begin{align*}
  (\alpha_i,\alpha_j)=\epsilon_i a_{ij} \qquad \mbox{for all } i,j\in I.
\end{align*}
The action of $W$ on $\hfrak$ can be interpreted in terms of the Kac-Moody group $G$ associated to $\gfrak'$. We briefly recall the construction of $G$ along the lines of \cite[1.3]{a-KW92}. Let $G^\ast$
denote the free product of the additive groups $\gfrak_\alpha$ for $\alpha\in \Phi^{re}$ and let $i_\alpha:\gfrak_\alpha\rightarrow G^\ast$ be the canonical inclusion. Recall that a $\gfrak'$-module $(V,\pi)$ is called integrable if $\pi(e_i)$ and $\pi(f_i)$ act locally nilpotently on $V$ for all $i\in I$. For any integrable $\gfrak'$-module $(V,\pi)$ define a homomorphism $\pi^\ast:G^\ast\rightarrow \Aut(V)$ by $\pi^\ast(i_\alpha(x))=\exp(\pi(x))$ for $x\in \gfrak_\alpha$,
$\alpha\in \Phi^{re}$. Let $N^\ast$ denote the intersection of all $\ker(\pi^\ast)$. By definition $G:=G^\ast/N^\ast$ is the Kac-Moody group associated to $\gfrak'$. We write $\pi^\ast:G^\ast\rightarrow G$ to denote the canonical projection. To shorten notation we write
$\exp(x):=\pi^\ast(i_\alpha(x))$ for $x\in \gfrak_\alpha$, $\alpha\in \Phi^{re}$. Let $\Aut(\gfrak)$ denote the group of Lie algebra automorphisms of $\gfrak$. There is a group homomorphism
$\Ad:G\rightarrow \Aut(\gfrak)$ corresponding to the adjoint action of $\gfrak'$ on $\gfrak$.
More explicitly, we have $\Ad(\exp(x)))=\exp(\ad(x))$ for all $x\in \gfrak_\alpha$,
$\alpha\in \Phi^{re}$.  Consult \cite[1.3]{a-KW92} for more details. 

We are now ready to lift the action of $r_i$ on $\hfrak$ to an element in $\Aut(\gfrak)$.
For any $i\in I$ define an element $m_i\in G$ by
\begin{align*}
  m_i=\exp(e_i) \exp(-f_i) \exp(e_i).
\end{align*}
By \cite[Lemma 3.8]{b-Kac1} one has $\Ad(m_i)(\gfrak_\alpha)=\gfrak_{r_i(\alpha)}$ for all $\alpha\in
\hfrak^\ast$ and $\Ad(m_i)|_{\hfrak}=r_i$ for all $i\in I$. Moreover, by \cite[Remark 3.8]{b-Kac1} the elements $m_i$ satisfy the relations
\begin{align}\label{eq:coxeter-rel}
  \underbrace{m_i m_j m_i\, .\,.\,.}_{m_{ij} \mathrm{factors}} = \underbrace{m_j m_i m_j\, .\,.\,.}_{m_{ij} \mathrm{ factors}}
\end{align}
with $m_{ij}$ as in \eqref{Coxeter}. 
\subsection{Automorphisms of the first and second kind}\label{sec:Auto12}
Let $\bfrak^+=\nfrak^+\oplus \hfrak$ and $\bfrak^-=\hfrak\oplus \nfrak^-$ denote the standard Borel subalgebras of $\gfrak$. One says that $\sigma\in \Aut(\gfrak)$ is of the first kind if $\sigma(\bfrak^+)=\Ad(g)(\bfrak^+)$ for some $g\in G$. By \cite[4.6]{a-KW92} this is equivalent to 
\begin{align*}
  \dim(\sigma(\bfrak^+)\cap \bfrak^-)<\infty.
\end{align*}
Similarly, $\sigma$ is of the second kind if $\sigma(\bfrak^+)=\Ad(g)(\bfrak^-)$ for some $g$ in $G$, or equivalently,
\begin{align*}
  \dim(\sigma(\bfrak^+)\cap \bfrak^+)<\infty.
\end{align*}
If $\gfrak$ not of finite type then any automorphism of $\gfrak$ is either of the first or of the second kind but never both. In this case the group $\Aut(\gfrak)$ is $\Z_2$-graded with automorphisms of the first kind in degree $0$ and automorphisms of the second kind in degree $1$.

Apart from $\Ad(G)$, four other subgroups of $\Aut(\gfrak)$ are relevant in the following.
The set $\Htil=\Hom(Q,\field^\times)$ of group homomorphism from $Q$ to the multiplicative group $\field^\times$ is a group under multiplication. For any $x\in \Htil$ define $\Ad(x)\in \Aut(\gfrak)$ by $\Ad(x)|_\hfrak=\id_\hfrak$ and $\Ad(x)(v)=x(\alpha)v$ for all $\alpha\in \Phi$, $v\in \gfrak_\alpha$. Then $\Ad(\Htil)$ is a subgroup of $\Aut(\gfrak)$. The automorphisms in $\Ad(\Htil)$ are of the first kind.

Let $\Aut(A)$ denote the group of all permutations $\sigma$ of the set $I$ such that $a_{i,j}=a_{\sigma(i),\sigma(j)}$. View $\Aut(A)$ as a subgroup of $\Aut(\gfrak')$ by requiring
$\sigma(e_i)=e_{\sigma(i)}$,  $\sigma(f_i)=f_{\sigma(i)}$. The action of $\Aut(A)$ on $\gfrak'$ can be extended to an action on $\gfrak$ following \cite[4.19]{a-KW92}. Then the induced map on $\hfrak^\ast$ satisfies $\sigma(\alpha_i)=\alpha_{\sigma(i)}$. Viewed as a subgroup of $\Aut(\gfrak)$, the group $\Aut(A)$ consists of automorphisms of the first kind.

Recall the Chevalley involution $\omega$ defined by 
  \begin{align}\label{eq:Chevalley}
    \omega(e_i)&=-f_i & \omega(f_i)&=-e_i & \omega(h)=-h 
  \end{align}
for all $i\in I$ and  $h\in \hfrak$. Define $\Out(A)=\Aut(A)$ if $A$ is of finite type and $\Out(A)=\Aut(A)\cup\omega\Aut(A)$ else \cite[1.32]{a-KW92}. As $\omega$ commutes with all elements in $\Aut(A)$, one obtains that $\Out(A)$ is a subgroup of $\Aut(\gfrak)$. 

Let $\Aut(\gfrak,\gfrak')$ denote the group of automorphisms of $\gfrak$ which restrict to the identity on $\gfrak'$. The following decomposition of $\Aut(\gfrak)$ is established in \cite[4.23]{a-KW92}.
\begin{prop}
  $\Aut(\gfrak)=\Out(A)\ltimes \big(\Aut(\gfrak,\gfrak')\times (\Ad(\Htil)\ltimes \Ad(G))\big).$
\end{prop}
The above proposition implies that any $\sigma\in \Aut(\gfrak)$ of the second kind can be written in the form
\begin{align}\label{eq:Aut-general-form}
  \sigma= \Ad(s)\circ f\circ \tau\circ\omega \circ \Ad(g) 
\end{align}
for some $s\in \Htil$, $f\in \Aut(\gfrak,\gfrak')$, $\tau\in \Aut(A)$, and $g\in G$.
\subsection{Longest elements in parabolic subgroups of $W$}\label{sec:Weyl}
Special elements in $G$ play a major role in the classification of involutive automorphism of the second kind in Theorem \ref{classThm}. Let $X\subset I$ be a subset of finite type, let $W_X\subset
W$ be the corresponding parabolic subgroup with longest element $w_X$,
and let $\Phi_X$ be the corresponding root system considered as a subset of $\Phi$. Fix a reduced
decomposition $w_X=r_{i_1}\dots r_{i_k}$ and define $m_X=m_{i_1}\dots
m_{i_k}$. By the word property for Coxeter groups \cite{b-BjornerBrenti06} and relation \eqref{eq:coxeter-rel}, the element $m_X\in G$ is independent of the chosen reduced decomposition of $w_X$. 

Let $\gfrak_X\subseteq\gfrak$ denote the semisimple Lie algebra generated by
$\{e_i,f_i\,|\,i\in X\}$ with Chevalley automorphism $\omega_X$
mapping $e_i$ to $-f_i$. Let moreover $\tau_X$ denote the diagram
automorphism of $\gfrak_X$ corresponding to the longest element
$w_X$. The following proposition can be found in \cite[4.9, 4.10]{a-BBBR95}. We use the notation $\Aut(A,X)=\{\sigma\in \Aut(A)\,|\,\sigma(X)=X\}$.
\begin{prop}\label{Rousseau-Prop}
  \begin{enumerate}
     \item The automorphism $\Ad(m_X)$ of $\gfrak$ leaves $\gfrak_X$
       invariant and 
       satisfies the relation
       \begin{align*}
          \Ad(m_X)|_{\gfrak_X}=\tau_X \omega_X.
       \end{align*}
     \item In $\Aut(\gfrak)$ the relation
        \begin{align*}
          \Ad(m_X^2)=\Ad(\exp(i\pi 2 \rho^\vee_X))
        \end{align*}
      holds, where $2\rho^\vee_X$ denotes the sum of the positive
      coroots of $\Phi_X$.
     \item The automorphism $\Ad(m_X)$ of $\gfrak$ commutes with all
     elements in $\Aut(A,X)$ and with the Chevalley involution $\omega$.
  \end{enumerate}
\end{prop}
\begin{proof}
   Parts (1) and (2) are \cite[Lemme 4.9]{a-BBBR95} and \cite[Corollaire 4.10.3]{a-BBBR95}, respectively. 
   To prove (3) note that by definition the relation
      \begin{align} \label{Ad-commute}
        \sigma(\Ad(m_i)x)=\Ad(m_{\sigma(i)})\sigma(x) 
      \end{align}
      holds for any $\sigma\in\Aut(A)$ and any $x\in \gfrak$. If
      $\sigma\in \Aut(A,X)$ and $w_X=r_{i_1}\dots r_{i_k}$ is a
      reduced decomposition then $w_X=r_{\sigma(i_1)}\dots
      r_{\sigma(i_k)}$. Hence
      \begin{align*}
        m_X=m_{\sigma(i_1)}\dots m_{\sigma(i_k)}
      \end{align*}
      because $m_X$ is independent of the choice of reduced
      decomposition for $w_X$. Now \eqref{Ad-commute} implies
      \begin{align*}
        \sigma(\Ad(m_X)(x))=\Ad(m_X)(\sigma(x))
      \end{align*}
      and thus $\Ad(m_X)$ and $\sigma$ commute.

      An $\slfrak_2$-argument shows the relation $m_i=\exp(-f_i) \exp(e_i) \exp(-f_i)$ in $G$.
      Hence 
      \begin{align*}
        \omega(\Ad(m_i) x)= \Ad(\exp(-f_i) \exp(e_i) \exp(-f_i))(\omega(x))=\Ad(m_i)(\omega(x))
      \end{align*}  
      holds for any $x\in \gfrak$. Thus $\Ad(m_X)$ commutes with $\omega$.
\end{proof}
\subsection{Admissible pairs}
We introduce the notion of an admissible pair which generalizes Satake diagrams as given in \cite{a-Araki62} in the finite case to symmetrizable Kac-Moody algebras. It is slightly more explicit than a similar notion used in \cite[Def. 4.10 b), Cor. 4.10.4]{a-BBBR95} in so far as we also
include Property (\ref{adm3}), below, in the definition. 
\begin{defi}\label{admissible}
  A pair $(X,\tau)$ consisting of a subset $X\subseteq I$ of finite type and an
  element $\tau\in \Aut(A,X)$ is called admissible if the following
  conditions are satisfied:
  \begin{enumerate}
    \item $\tau^2=\id_I$.
    \item \label{adm2} The action of $\tau$ on $X$ coincides with the action of $-w_X$.
    \item \label{adm3} If $j\in I\setminus X$ and $\tau(j)=j$ then
         $\alpha_j(\rho_X^\vee)\in \Z$. 
   \end{enumerate}
\end{defi}
  If $\gfrak$ is of finite type then the pair $(X=I,\tau=-w_X)$ is always admissible.
  It is an instructive exercise to verify that the Satake diagrams in \cite[pp.~32/33]{a-Araki62} describe all other admissible pairs for finite dimensional simple $\gfrak$.
\begin{eg}
Consider $\gfrak=\slfrak_4(\field)$ with $I=\{1,2,3\}$ and the standard
choice of simple roots and coroots. Then the pair $(\{1,3\},\id_I)$ is
admissible because $-w_X$ acts as the identity on $X$ and
$\rho_X^\vee=(h_1+h_3)/2$ satisfies
\begin{align*}
  \alpha_1(\rho_X^\vee)=\alpha_3(\rho_X^\vee)=1,\qquad \alpha_2(\rho_X^\vee)=-1.
\end{align*}
The pair $(\{1\},\id_I)$, however, is not admissible because in this
case $\rho_X^\vee=h_1/2$ satisfies
\begin{align*}
  \alpha_1(\rho^\vee_X)&=1,& \alpha_2(\rho_X^\vee)&=-1/2,&
  \alpha_3(\rho_X^\vee)=0
\end{align*}
and hence condition (\ref{adm3}) is not satisfied.
The pair $(\{2\},\id_I)$ is not admissible for the same reason while the pair $(\{1,2\},\id_I)$ violates (\ref{adm2}). Finally, the pair  $(\{2\},(13))$ is admissible because now condition
(\ref{adm3}) is empty. 
\end{eg}
We now associate an involutive automorphism of the second kind to any admissible pair.
Let $>$ be a fixed total order on the set $I$. Let $(X,\tau)$ be an
admissible pair. We define an element $s(X,\tau)\in \Htil$ by
\begin{align}\label{sDef}
  s(X,\tau)(\alpha_j)=\begin{cases}
      1& \mbox{if $j\in X$ or $\tau(j)=j$,}\\
     i^{\alpha_j(2\rho^\vee_X)}& \mbox{if $j\notin X$ and $\tau(j)>
      j$,}\\
     (-i)^{\alpha_j(2\rho^\vee_X)}& \mbox{if $j\notin X$ and $\tau(j)<
      j$,}\\
    \end{cases}
\end{align}
where $i\in \field$ denotes a square-root of $-1$.
\begin{thm}
  Let $(X,\tau)$ be an admissible pair. Then
  \begin{align}\label{tauDef}
    \theta(X,\tau)=\Ad(s(X,\tau))\circ \tau\circ \omega\circ \Ad(m_X)
  \end{align}
  defines an involutive automorphism of $\gfrak$ of the second kind which commutes with
  $\Ad(s(X,\tau))$.
\end{thm}
\begin{proof}
  Note first that $\Ad(s(X,\tau))$ as defined by \eqref{sDef} commutes 
  with $\Ad(m_X)$ and with $\tau\circ \omega$ and hence it commutes
  also with $\theta(X,\tau)$ as defined by \eqref{tauDef}. Note also
  that $\theta(X,\tau)$ is an automorphism of second kind because
  $\Ad(s(X,\tau))$, $\tau$, and $\Ad(m_X)$ are automorphisms of the
  first kind, and $\omega$ is an automorphism of the second kind. By
  Proposition \ref{Rousseau-Prop}.(3) the map $\Ad(m_X)$ commutes with
  $\tau$ and $\omega$ and hence one obtains 
  \begin{align*}
    \theta(X,\tau)^2=\Ad(s(X,\tau))^2\circ \Ad(m_X)^2.
  \end{align*}
  Both $\Ad(s(X,\tau))^2$ and $\Ad(m_X)^2$ act trivially on $\hfrak$ and
  hence it remains to show that 
  \begin{align}\label{need}
     \Ad(s(X,\tau))^2\circ \Ad(m_X)^2(x)=x
  \end{align}
  for all $x\in \gfrak_{\alpha_j}$, $j\in I$. By Proposition
  \ref{Rousseau-Prop}.(2) on has
  $\Ad(m_X)^2(x)=(-1)^{\alpha_j(2\rho^\vee_X)}x$ for all $x\in
  \gfrak_{\alpha_j}$. Relation \eqref{need} now follows immediately
  from definition \eqref{sDef} and condition (\ref{adm3}) in
  Definition \ref{admissible}.
\end{proof}
\begin{rema}
  The definition of $s(X,\tau)$ in \eqref{sDef} depends on the chosen total order on the set $I$ and hence so does $\theta_q(X,\tau)$. More explicitly, for a different total order on $I$ the map $s(X,\tau)$ may change by a factor $-1$ on both $\alpha_j$ and $\alpha_{\tau(j)}$ if $j\neq \tau(j)$. The corresponding map $\theta(X,\tau)$ for the new total order is conjugate to the original $\theta(X,\tau)$ by an element in $\Ad(\Htil)$.
\end{rema}
The following theorem provides the main result of this section, namely the
classification of involutive automorphisms of the second kind of
$\gfrak$ in terms of admissible pairs. Note that the group $\Aut(A)$ acts on the set of all admissible pairs by
\begin{align*}
  \sigma((X,\tau))=(\sigma(X), \sigma\circ\tau\circ \sigma^{-1})
\end{align*}
for $\sigma\in \Aut(A)$.
\begin{thm}\label{classThm}
   The map $(X,\tau)\mapsto \theta(X,\tau)$ gives a
  bijection between the set of $\Aut(A)$-orbits of admissible pairs for
  $\gfrak$ and the set of $\Aut(\gfrak)$-conjugacy classes of involutive automorphisms
  of the second kind. 
\end{thm}
  We will prove Theorem \ref{classThm} in Appendix \ref{app:classThm} by reducing it to results in \cite{a-KW92}. 
  
By construction one has $\theta(X,\tau)(h)=-w_X\tau(h)$ for all $h\in \hfrak$. Hence, by duality, $\theta=\theta(X,\tau)$ induces the map
\begin{align}\label{eq:Theta}
  \Theta:\hfrak^\ast\rightarrow \hfrak^\ast, \qquad \Theta(\alpha)=-w_X \tau(\alpha). 
\end{align}  
As the bilinear form $(\cdot,\cdot)$ on $\hfrak^\ast$ is $W$-invariant one obtains
\begin{align}\label{eq:invTheta}
  (\alpha,\beta)=(\Theta(\alpha),\Theta(\beta))
\end{align}
for all $\alpha, \beta\in Q$.
\subsection{The invariant Lie subalgebras $\kfrak$ and $\kfrak'$}
Let $(X,\tau)$ be an admissible pair and let $\theta=\theta(X,\tau)$
be the corresponding involution of the second kind defined by
\eqref{tauDef}. Let $\kfrak=\{x\in \gfrak\,|\,\theta(x)=x\}$ and $\kfrak'=\{x\in \gfrak'\,|\,\theta(x)=x\}$ denote the invariant Lie subalgebras of $\gfrak$ and $\gfrak'$, respectively. To define quantum group analogs of $\kfrak$ and $\kfrak'$ we determine a set of generators.
\begin{lem}\label{generators}
  The Lie algebra $\kfrak$ is generated by the elements
  \begin{align}
    & e_i,f_i &&\mbox{ for $i\in X$,} \label{m-gens}\\
    & h\in \hfrak &&\mbox{ with $\theta(h)=h$}, \label{Ut0-gens}\\
    & f_i+\theta(f_i) &&\mbox{ for $i\in
      I\setminus X$.} \label{Bis} 
  \end{align}
  Similarly, the Lie algebra $\kfrak'$ is generated by the elements \eqref{m-gens}, \eqref{Bis}, and all $h\in \hfrak'$ with $\theta(h)=h$.
\end{lem}
\begin{proof}
  Let $\tilde{\kfrak}$ denote the Lie subalgebra of $\gfrak$ generated by the elements \eqref{m-gens}, \eqref{Ut0-gens}, and \eqref{Bis}. Observe that $\tilde{\kfrak}\subseteq\kfrak$ because all generators of $\tilde{\kfrak}$ are invariant under $\theta$. Conversely, assume that $x\in \kfrak$. Write $x=x^+ + x^0 + x^-$ with $x^+\in \nfrak^+$, $x^0\in \hfrak$, and $x^-\in \nfrak^-$. As $\theta(f_i)\in \nfrak^+$ for $i\in I\setminus X$ there exists $u\in \tilde{\kfrak}$ such that $u-x^-\in\nfrak^+ + \hfrak$. Hence we may assume that $x^-=0$. Similarly, $x^0\in \tilde{\kfrak}$ and hence we may assume that $x=x^+\in \nfrak^+$. Write $x$ as a sum of weight vectors $x=\sum_{\beta\in Q^+} x_\beta$. As $\theta(\gfrak_\beta)=\gfrak_{-w_X\tau(\beta)}$ the relation $x_\beta\neq 0$ implies $\beta\in \sum_{i\in X} \N_0 \alpha_i$. Hence $x\in \tilde{\kfrak}$.    
\end{proof}
The corresponding generators of the universal enveloping algebra can be modified by constant terms. The nonstandard quantum symmetric pair coideal subalgebras which will be introduced in Definition \ref{def:qsp} contain quantum analogs of such modified generators.
\begin{cor}\label{cor:U(g)generators}
  Let $\us=(s_i)_{i\in I\setminus X}\in \field^{I\setminus X}$. The universal enveloping algebra $U(\kfrak)$ is generated by the elements \eqref{m-gens}, \eqref{Ut0-gens}, and 
  \begin{align}
    &f_i+\theta(f_i)+s_i&&\mbox{ for $i\in
      I\setminus X$.} \label{eq:UBis} 
  \end{align}
  as a unital algebra.
\end{cor}
\begin{rema}
  A complete set of defining relations of $U(\kfrak)$ or $U(\kfrak')$ in terms of the generators \eqref{m-gens}--\eqref{Bis} can be written down. For the GIM Lie algebras described in the introduction and in Section \ref{sec:QGIM} this was achieved by Berman in \cite{a-Berman89}. The author is not aware of any publication describing defining relations of $U(\kfrak)$ or $U(\kfrak')$ for general admissible pairs. A rather indirect way to obtain such relations is provided by specialization of the defining relations of their quantum analogs using the results of Sections \ref{sec:GensRels} and \ref{sec:specialize} of the present paper. 
\end{rema}
\subsection{Compatible minimal realizations}\label{sec:compatible}
  For $\tau\in \Aut(A)$ the corresponding Lie algebra automorphism $\tau:\gfrak\rightarrow \gfrak$ defined in \cite[4.19]{a-KW92} does not necessarily leave $P^\vee$ invariant. If, however, there exists a permutation $\sigma$ of the set $\{1,\dots, s\}$ such that
  \begin{align}\label{eq:compatible}
    &\alpha_{\tau(i)}(d_{\sigma(j)})=\alpha_i(d_j) & &\mbox{ for all $i\in I$, $j\in \{1,\dots, s\}$}
  \end{align}
then one has $\tau(P^\vee)=P^\vee$ with $\tau(d_j)=d_{\sigma(j)}$. In this case we say that the  minimal realization is compatible with $\tau$. In the next subsection we will need to choose a compatible minimal realization in order to lift $\tau$ to an automorphism of the quantized enveloping algebra of $\gfrak$. 
\begin{eg}
  Let $A=\left(\begin{matrix} 2 & -2 \\ -2 & 2 \end{matrix}\right)$ be the generalized Cartan matrix of affine $\slfrak_2$ with $I=\{0,1\}$ and $P^\vee=\Z h_0\oplus \Z h_1 \oplus \Z d$. Then the minimal realization defined by $\alpha_0(d)=1$, $\alpha_1(d)=0$ is not compatible with the transposition $(01)\in \Aut(A)$. However, the minimal realization defined by $\alpha_0(d)=1=\alpha_1(d)$ is compatible with $(01)$.
\end{eg}
  \begin{prop}
    Let $A$ be of affine type and $\tau\in \Aut(A)$. Then there exists a minimal realization for $A$ which is compatible with $\tau$.
  \end{prop}
\begin{proof}
  Assume that $A=(a_{ij})_{i,j\in I}$ is given by one of the Dynkin diagrams in \cite[p. 54/55]{b-Kac1} and $I=\{0,1,\dots,n\}$, where $\alpha_0$ denotes the additional affine root. There exist positive integers $b_j$, uniquely determined by $b_0=1$, such that $\sum_{j=0}^n b_j a_{ij}=0$. If $\tau(0)=0$ then define $\alpha_i\in \hfrak^\ast$ by \eqref{eq:alpha(h)} and $\alpha_i(d_1)=\delta_{i,0}$. If $\tau(0)\neq 0$ then define $\alpha_i$ by \eqref{eq:alpha(h)} and $\alpha_i(d_1)=\delta_{i,0}+\delta_{i,\tau(0)}$. In both cases it follows from the positivity of the coefficients $b_j$ that $\{\alpha_i\,|\,i\in I\}$ is a linearly independent subset of $\hfrak^\ast$. Relation \eqref{eq:compatible} holds by construction with $\sigma=\id_{\{1\}}$.
\end{proof}
\begin{rema}
  The above proof does not work in the general symmetrizable Kac-Moody case. It would be good to know if any automorphism of a symmetrizable generalized Cartan matrix has a compatible minimal realization. 
\end{rema}  
\section{Lusztig automorphisms}\label{sec:LusztigBraid}
This section provides notation and background on quantum groups. It is well known that quantized enveloping algebras have only very few Hopf algebra automorphisms, see Theorem \ref{thm:AutUqg}. A quantum analog of the involutive automorphism $\theta(X,\tau)$ corresponding to an admissible pair can therefore be defined only as an algebra automorphism of the quantized enveloping algebra of $\gfrak$. It is straightforward to define quantum group analogs of the Chevalley involution and of elements of $\Aut(A)$. To define a quantum group analog of the automorphism $\Ad(m_X)$ we briefly recall the definition and properties of the Lusztig automorphisms. In particular, in Sections \ref{sec:triang-decomp}, \ref{sec:Lusztig} we use results by K\'eb\'e \cite{a-Kebe99} to describe the action of certain Lusztig automorphisms on Chevalley generators of quantized enveloping algebras.  
\subsection{Quantum groups}\label{sec:quagroup}
  Let $\field(q)$ be the field of rational functions in an indeterminate $q$. Recall that $D=\mathrm{diag}(\epsilon_i|i\in I)$ and define $q_i=q^{\epsilon_i}$ for any $i\in I$. Up to minor notational changes we follow the presentation in \cite{b-Lusztig94} and define for $\gfrak=\gfrak(A)$ the quantized enveloping algebra $\uqg$ to be the associative $\field(q)$-algebra with generators $E_i, F_i,$ and $K_h$ for all $i\in I$, $h\in P^\vee$ and relations  
\begin{enumerate}
  \item $K_0=1$ and $K_h K_{h'}=K_{h+h'}$ for all $h,h'\in P^\vee$. 
  \item $K_h E_i= q^{\alpha_i(h)}E_i K_h$ for all $i\in I$, $h\in P^\vee$.
  \item $K_h F_i= q^{-\alpha_i(h)}F_i K_h$ for all $i\in I$, $h\in P^\vee$.
  \item $\displaystyle E_i F_i - F_i E_i=\delta_{ij} \frac{K_i-K_i^{-1}}{q_i-q_i^{-1}}$ for all $i\in I$ where $K_i=K_{h_i}^{\epsilon_i}$.
  \item \label{q-Serre} The quantum Serre relations given in \cite[3.1.1.(e)]{b-Lusztig94}.
\end{enumerate}
For later use we make the quantum Serre relations more explicit. For any $i,j\in I$ let $F_{ij}(x,y)$ denote the noncommutative polynomial in two variables defined by
\begin{align}\label{eq:Fij-def}
  F_{ij}(x,y)=\sum_{n=0}^{1-a_{ij}}(-1)^n
  \left[\begin{matrix}1-a_{ij}\\n\end{matrix}\right]_{q_i}
  x^{1-a_{ij}-n}yx^n.  
\end{align}
where $\left[ \abv{1-a_{ij}}{n} \right]_{q_i}$ denotes the $q_i$-binomial coefficient defined for instance in \cite[1.3.3]{b-Lusztig94}.  
By \cite[Corollary 33.1.5]{b-Lusztig94}, the quantum Serre relations in (\ref{q-Serre}) can be written in the form
\begin{align}\label{eq:q-Serre}
  F_{ij}(E_i,E_j)=F_{ij}(F_i,F_j)=0\qquad \mbox{for all $i,j\in I$}.
\end{align}
The algebra $\uqg$ is a Hopf algebra with coproduct $\kow$, counit $\vep$, and antipode $S$ given by
\begin{align}
  \kow(E_i)&=E_i \ot 1 + K_i\ot E_i,& \vep(E_i)&=0, & S(E_i)&=-K_i^{-1} E_i,\label{eq:E-copr}\\
  \kow(F_i)&=F_i\ot K_i^{-1} + 1 \ot F_i,&\vep(F_i)&=0, & S(F_i)&=-F_i K_i,\label{eq:F-copr}\\
  \kow(K_h)&=K_h \ot K_h,& \vep(K_h)&=1, & S(K_h)&=K_{-h}\label{eq:K-copr}
\end{align}
for all $i\in I$, $h\in P^\vee$.
We denote by $\uqgp$ the Hopf subalgebra of $\uqg$ generated by the elements $E_i, F_i$, and $K_i^{\pm 1}$ for all $i\in I$. 
\begin{remarks}
  1) The Hopf algebra $\uqgp$ coincides with the Hopf algebra $U_q(\gfrak_C)$ for $C=DA$ defined in \cite[3.2.9]{b-Joseph}.\\
  2) In later sections we will sometimes need to work with $\uqg$ defined over a field extension of $\field(q)$.  In particular, in Sections \ref{sec:omega-commute} and \ref{sec:equivalence} we will work over $\field(q^{1/2})$ and $\oKq$, respectively. In this case $\uqg$ is defined by the same relations (1)--(4) and \eqref{eq:q-Serre}.
\end{remarks}
As usual let $U^+$, $U^-$, and $U^0$ denote the subalgebra of $\uqg$ generated by the elements $\{E_i\,|\,i\in I\}$, $\{F_i\,|\,i\in I\}$, and $\{K_h\,|\,h\in P^\vee\}$ respectively. Moreover define $U^0{}'$ to be the subalgebra of $U^0$ generated by the elements $\{K_i^{\pm 1}\,|\,i\in I\}$. By \cite[3.2]{b-Lusztig94} the multiplication maps give isomorphisms of vector spaces
\begin{align}\label{eq:triang-decomp}
  U^+\ot U^0\ot U^- \cong \uqg, \qquad U^+\ot U^0{}'\ot U^- \cong \uqgp.
\end{align} 
Let $\field(q)[Q]$ be the group algebra of the root lattice. There is an algebra isomorphism $\field(q)[Q]\rightarrow U^0{}'$ such that $\alpha_i\mapsto K_i$. For any $\beta\in Q$ we hence write
\begin{align}\label{eq:Kbeta}
  K_\beta=\prod_{i\in I} K_i^{n_i} \qquad \mbox{if $\beta=\sum_{i\in I} n_i \alpha_i$.}
\end{align}
The commutation relations (2), (3) then take the form
\begin{align*}
  K_\beta E_i = q^{(\beta,\alpha_i)} E_i K_\beta,\qquad K_\beta F_i = q^{-(\beta,\alpha_i)} F_i K_\beta
  \qquad \mbox{for all } \beta\in Q,\, i\in I.
\end{align*}
For any $\uqgp$-module $M$ and any $\mu\in P$ let 
\begin{align*}
  M_\mu=\{m\in M\,|\, K_i m=q_i^{\mu(h_i)}m \mbox{ for all } i\in I\}
\end{align*}
denote the corresponding weight space. In particular, with respect to the left adjoint action of $\uqgp$ on itself or on $\uqg$ one obtains a $Q$-grading of both $\uqgp$ and $\uqg$. More precisely, we use Sweedler notation to define the left adjoint action of $\uqg$ on itself by 
\begin{align*}
  \ad(x)(u)= x_{(1)} u S(x_{(2)}) \qquad \mbox{ for all $x,u\in \uqg$.}
\end{align*}
Then one has $\ad(K_i)(u)=K_i u K_i^{-1}$ and hence
\begin{align*}
  \uqg_\beta= \{u\in \uqg\,|\, K_i u K_i^{-1}= q^{(\alpha_i,\beta)}u\}
\end{align*}
for all $\beta\in Q$, and $\uqgp_\beta$ is defined analogously.
\subsection{Automorphisms of $\uqg$ and $U_q(\gfrak')$}\label{sec:q-Aut}
Define $\omega$ to be the algebra automorphism of $\uqg$ determined by
\begin{align}\label{eq:q-Chevalley}
  \omega(E_i)=-F_i, \qquad \omega(F_i)=-E_i, \qquad \omega(K_h)=K_{-h}.
\end{align}
Observe that our definition of $\omega$ differs by a sign from the definition given in \cite[3.1.3]{b-Lusztig94}. We choose to add the sign in order to make $\omega$ a quantum analog of the Chevalley involution given by \eqref{eq:Chevalley}. The map $\omega$ is a coalgebra antiautomorphism of $\uqg$. 

Using the notation introduced in Subsection \ref{sec:Auto12} also in the quantum case, we set $\tilde{H}=\Hom(Q,\field(q)^\times)$ where $\field(q)^\times$ denotes the multiplicative group of $\field(q)$. For any $x$ in $\tilde{H}$ define a Hopf algebra automorphism $\Ad(x)$ of $\uqg$ by 
\begin{align*}
  \Ad(x)(u)=x(\beta)u \qquad \mbox{for all } u\in \uqg_\beta \mbox{ and } \beta\in Q. 
\end{align*}
Note that $\Ad(x)$ leaves $\uqgp$ invariant. Moreover, one has
\begin{align}\label{eq:Htil-om}
  \Ad(x)\circ \omega = \omega\circ\Ad(x^{-1}) \qquad \mbox{ for all $x\in \Htil$.}
\end{align}
By \cite[4.19]{a-KW92} the group $\Aut(A)$ acts on $\gfrak(A)$ by Lie algebra automorphisms. It seems not straightforward to find an analogue action of $\Aut(A)$ on $\uqg$ because it is unclear how to define the action on the generators $K_{d_s}$. However, for any $\tau\in\Aut(A)$ there exists an algebra automorphism of $\uqgp$, denoted by the same symbol $\tau$, such that
\begin{align}\label{eq:tau-def}
  \tau(E_i)=E_{\tau(i)},\qquad  \tau(F_i)=F_{\tau(i)},\qquad  \tau(K_i^{\pm1})=K_{\tau(i)}^{\pm 1}.
\end{align}  
Moreover, if the minimal realization of $A$ is compatible with $\tau$ in the sense of Subsection \ref{sec:compatible} with corresponding permutation $\sigma$ of $\{1,\dots,s\}$, then one can extend $\tau$ to an algebra automorphism of $\uqg$ given by \eqref{eq:tau-def} and $\tau(d_j)=d_{\sigma(j)}$. 
The action of $\Aut(A)$ on $Q$ allows us to form the semidirect product $\tilde{H} \rtimes \Aut(A)$ which acts faithfully on $\uqgp$ by Hopf algebra automorphisms. Let $\Auth(\uqgp)$ denote the group of Hopf algebra automorphisms of $\uqgp$. The following theorem is given in \cite[Theorem 2.1]{a-Twietmeyer92}.
\begin{thm}\label{thm:AutUqg}
  The map 
  \begin{align*}
    \tilde{H} \rtimes \Aut(A)\rightarrow \Auth(\uqgp), \qquad (x,\tau)\mapsto \Ad(x)\circ \tau
  \end{align*}
  is an isomorphism of groups. 
\end{thm}

\subsection{Parabolic decompositions}\label{sec:triang-decomp}
We recall here some results of \cite{a-Kebe99} which will be used
to make the quantum group analogue of the involution $\theta(X,\tau)$ explicit in Theorem \ref{thm:theta-props}. Let $X\subset
I$ be a subset of finite type. Let $\cM_X$ denote the subalgebra of $\uqg$ generated by
$\{E_i, F_i, K_i^{\pm 1}\,|\,i \in X\}$. Write $\cM_X^-$, $\cM_X^+$, and $\cM^0_X$ for the
subalgebras generated by  $\{F_i,\,|\,i \in X\}$, $\{E_i\,|\,i \in
X\}$, and $\{K_i^{\pm 1}\,|\,i\in X\}$, respectively. For any $i\in I\setminus X$ the subspace
$\ad(\cM_X)(E_i)$ of $\Uq$ is finite-dimensional and contained in
$\Uq^+$. This follows from the triangular decomposition of
$\cM_X$, the fact that $\ad(F_j)(E_i)=0$ for $j\in X$, the quantum Serre relations, and
the fact that $\ad (E_j)u\in \Uq^+$ for all $u \in \Uq^+$. Moreover,
the $\ad(\cM_X)$-module  $\ad(\cM_X)(E_i)$ is irreducible. Now define
$V_X^+$ to be the subalgebra generated by the elements of all the finite
dimensional subspaces $\ad (\cM_X)(E_i)$ for $i \notin X$. It is proved
in \cite{a-Kebe99} that multiplication gives an isomorphisms of vector spaces 
\begin{align} 
  \Uq^+ &\cong V_X^+\ot \cM^+_X\label{Keb1}.
\end{align}
To write explicitly the action of Lusztig's isomorphisms corresponding
to the longest element in $W_X$ we provide the following lemma. Recall from \eqref{eq:E-copr}--\eqref{eq:K-copr} that $S$ denotes the antipode of $\uqg$.
\begin{lem}\label{EiLem}
  Let $i\in I\setminus X$.\\
  (1) If $v\in \Uq^+$ is a highest weight vector
  for the adjoint action of $\cM_X\Uq^0$ of weight $w_X\alpha_i$ then
  $v\in \ad(\cM_X)(E_i)$. \\
  (2) If $v\in S(\Uq^-)$ is a lowest weight vector
  for the adjoint action of $\cM_X\Uq^0$ of weight $-w_X\alpha_i$ then
  $v\in \ad(\cM_X)(F_iK_i)$.
\end{lem}
\begin{proof}
  To prove (1) note that the assumption on the weight of $v$ and \eqref{Keb1}
  imply that 
  \begin{align*}
    v\in \ad (\cM_X^+ )(E_i)\ot \cM_X^+.
  \end{align*}
  With respect to this decomposition write $v=\sum_l u_l\ot m_l $ with 
  linearly independent weight vectors $u_l\in \ad(\cM_X^+)(E_i)$. The
  $\ad(\cM^+_X)$-invariance of $v$ implies
  \begin{align}\label{eq:adE}
    \sum_l \ad(E_i)(u_l)\ot m_l= -\sum_l \ad(K_i)(u_l)\ot \ad(E_i)(m_l)
  \end{align}
  for all $i\in X$ and hence the space $\cM_v$ spanned by
  the elements $m_l$ is $\ad(\cM_X^+\cM^0_X)$-invariant. It follows that
  $\cM_v$ is contained in the locally finite part $F_l(\cM_X)=\{m\in\cM_X\,|\,\dim(\ad(\cM_X)(m))<\infty\}$, see for example \cite[Proof of Lemma 3.1.1]{a-Fauqu98}. Choose $m_l$ of maximal weight (w.r.t.~$\ad(\cM_X)$). Then \eqref{eq:adE} implies that $m_l\in F_l(\cM_X)\cap \cM^+_X$ is a highest
  weight vector (w.r.t.~$\ad(\cM_X)$) and hence $m_l\in \field(q)
  1$. Hence $v\in \ad(\cM_X^+)(E_i)\ot \field(q) 1$ which proves (1).
  
  Property (2) is proved analogously using the decomposition
  $S(\Uq^-)\cong V_X^-\ot S(\cM^-_X)$ where $V_X^-$ denotes the subalgebra of
  $S(\Uq^-)$ generated by the elements of all the finite dimensional
  subspaces $\ad (\cM_X)(F_jK_j)$ for $j \notin X$. 
\end{proof}
In the following subsection we will give the highest and
the lowest weight vector from the previous lemma explicitly in terms
of Lusztig's automorphisms.  
\subsection{Definition and properties of Lusztig automorphisms}\label{sec:Lusztig}
  For any $i\in I$ let $T_i$ be the algebra automorphism of $\uqg$ denoted by $T''_{i,1}$ in \cite[37.1]{b-Lusztig94}. In particular, $T_i$ satisfies the relations
  \begin{align}\label{eq:Ti-rels}
       T_i(K_h)=K_{r_i(h)},\qquad
       T_i(E_i)=-F_i K_i,\qquad
       T_i(F_i)=-K_i^{-1}E_i
    \end{align}
  for any $h\in P^\vee$. Moreover, by \cite[37.24]{b-Lusztig94} the inverse of $T_i$ satisfies 
  \begin{align}\label{eq:T-invers}
      T_i^{-1}=\sigma\circ T_i\circ \sigma
    \end{align}
    where $\sigma:\uqg\rightarrow \uqg$ is the unique $\field(q)$-algebra
    antiautomorphism of $\uqg$ determined by $\sigma(E_i)=E_i$, $\sigma(F_i)=F_i$, and $\sigma(K_h)=K_{-h}$ for all $i\in I$, $h\in P^\vee$. In particular one has 
  \begin{align*}
       T^{-1}_i(K_h)=K_{r_i(h)},\qquad
       T^{-1}_i(E_i)=-K_i^{-1}F_i, \qquad
       T^{-1}_i(F_i)=-E_i K_i.
    \end{align*}
  By \cite[39.4.3]{b-Lusztig94} the automorphisms $T_i$ satisfy braid relations. Hence, for any $w\in W$ one may define an algebra automorphism $T_w:\uqg\rightarrow \uqg$ by
  \begin{align*}
      T_w=T_{i_1}\dots T_{i_k}
  \end{align*}
where $w=r_{i_1}\dots r_{i_k}$ is a reduced expression for $w$. One hence has $T_w(K_h)=K_{w(h)}$ for all $h\in P^\vee$ and, with the notation \eqref{eq:Kbeta}, $T_w(K_\beta)=K_{w(\beta)}$ for all $\beta\in Q$.    
   
Let $X$ be a subset of $I$ of finite type. We would like to replace the
  automorphism $\Ad(m_X)$ in \eqref{tauDef} by the Lusztig automorphism
  $T_{w_X}$ where $w_X$ denotes the longest element in the parabolic
  subgroup $W_X$.
  We first determine the action of $T_{w_X}$ on the generators
  corresponding to the subset $X$.
  \begin{lem}\label{Tw-on-X} Let $X$ be a subset of $I$ of finite type. Define a permutation $\tau$ of $X$ by $w_X(\alpha_i)=-\alpha_{\tau(i)}$ for all $i\in X$. For all $i\in X$ one has
  \begin{align*}
    T_{w_X}(E_i)&=-F_{\tau(i)} K_{\tau(i)},&
    T_{w_X}(F_i)&=-K_{\tau(i)}^{-1}E_{\tau(i)},&
    T_{w_X}(K_i)&=K^{-1}_{\tau(i)},\\
    T_{w_X}^{-1}(E_i)&=-K_{\tau(i)}^{-1} F_{\tau(i)},&
    T_{w_X}^{-1}(F_i)&=-E_{\tau(i)} K_{\tau(i)},&
    T_{w_X}^{-1}(K_i)&=K^{-1}_{\tau(i)}.
  \end{align*}
  \end{lem}
  \begin{proof}
    Write $w_X=w' r_i$ for some $w'\in W_X$. Then $w'(\alpha_i)=\alpha_{\tau(i)}$ because
    $w_X(\alpha_i)=-\alpha_{\tau(i)}$. By \cite[Proposition 8.20]{b-Jantzen96} one obtains $T_{w'}(E_i)=E_{\tau(i)}$ and hence
    \begin{align*}
      T_{w_X }(F_i)=T_{w'}(-K_i^{-1}E_i)=-K_{\tau(i)}^{-1}E_{\tau(i)}.
    \end{align*}
    Similarly one obtains the other two expressions. The formulas
    for $T_{w_X}^{-1}$ follow by conjugation with the algebra
    antiautomorphism $\sigma$, see \eqref{eq:T-invers}, using $w_X=w_X^{-1}$.
  \end{proof}
  Now we determine the action of $T_{w_X}$ and $T_{w_X}^{-1}$ on the remaining
  generators. This will allow us to identify the $\ad(\cM_X)$-highest
  weight vector in $\ad(\cM_X)(E_i)$ and the $\ad(\cM_X)$-lowest
  weight vector in $\ad(\cM_X)(F_iK_i)$ for $i\in I\setminus X$. The
  following lemma is a generalization of \cite[Remark 1.6]{a-DCoKa90}.
  \begin{lem}\label{Tw-on-Xc}
    Let $X$ be a subset of $I$ of finite type and $i\in I\setminus X$. 
    \begin{enumerate}
      \item The subspace $\ad(\cM_X)(E_i)$ of $U^+$ is a
       finite dimensional, irreducible $\ad(\cM_X)$-submodule of $\uqg$
       with highest weight vector $T_{w_X}(E_i)$ and lowest weight vector
       $E_i$. 
      \item The subspace $\ad(\cM_X)(F_iK_i)$ of $S(U^-)$ is
       a finite dimensional, irreducible $\ad(\cM_X)$-submodule of $\uqg$
       with highest weight vector $F_iK_i$ and lowest weight vector
       $T^{-1}_{w_X}(F_iK_i)$.
    \end{enumerate}     
  \end{lem}
  \begin{proof}
    For any $i\in I\setminus X$, $j\in X$ one has by Lemma \ref{Tw-on-X} the relation
    \begin{align*}
      0&=T_{w_X}(\ad(F_j)(E_i))=T_{w_X}(F_jE_iK_j-E_iF_jK_j)\\
       &=-K_{\tau(j)}^{-1} \big(E_{\tau(j)}T_{w_X}(E_i) - K_{\tau(j)}
       T_{w_X}(E_i)K_{\tau(j)}^{-1}E_{\tau(j)}  \big) K_{\tau(j)}^{-1}\\
       &=-K_{\tau(j)}^{-1} \big(\ad(E_{\tau(j)})(T_{w_X}(E_i))\big) K_{\tau(j)}^{-1}
    \end{align*}
    and hence $T_{w_X}(E_i)$ is a highest weight vector for the adjoint
    action of $\cM_X$ of weight $w_X\alpha_i$. By Lemma \ref{EiLem}
    one obtains that $T_{w_X}(E_i)$ is a highest weight vector of
    the irreducible $\ad(\cM_X)$-module generated by $E_i$.

    Similarly, for any $i\in I\setminus X$, $j\in X$ one obtains
    \begin{align*}
      0&=T^{-1}_{w_X}(\ad (E_j)(F_iK_i))=-K_{\tau(j)}^{-1}(\ad F_{\tau(j)})(T^{-1}_{w_X}(F_iK_i))K_{\tau(j)}^{-1}.
    \end{align*}
    Moreover, $T_{w_X}^{-1}(F_iK_i)=\sigma\circ T_{w_X}(F_i)K_{w_X\alpha_i}\in S(U^-)$. 
    Hence we may apply Lemma \ref{EiLem} and obtain that $T_{w_X}^{-1}
    (F_iK_i)\in \ad(\cM_X)(F_iK_i)$ is a lowest weight vector.
  \end{proof}

\section{Quantum involutions}\label{sec:qInvolution}
  From now on until the end of Section \ref{sec:specialize} fix an admissible pair $(X,\tau)$. In this section the
  automorphism $\theta(X,\tau)$ defined by \eqref{tauDef} is deformed to an
  automorphism of $\uqg$. We first consider the case
  $X=\emptyset$. Then we use the Lusztig automorphism corresponding to
  the longest word $w_X\in W_X$ to also deform the automorphism $\Ad(m_X)$.
\subsection{The case $X=\emptyset$}
  In this case $\theta(\emptyset, \tau)=\tau\circ \omega$ by
  \eqref{tauDef}. Recall that we use the same symbols $\tau$ and $\omega$ in the
  quantum case. The automorphism $\tau\circ \omega:\uqgp\rightarrow \uqgp$ is a natural quantum analog of $\theta(\emptyset, \tau)$, but to make the definition in this case compatible with the definition for general $X\subset I$ we consider a slightly different map.
  Define an algebra automorphism $\psi:\uqg \rightarrow \uqg$ by
  \begin{align}\label{psi-Def}
     \psi(E_i)&= E_i K_i,& \psi(F_i)&=K_i^{-1} F_i ,& \psi(K_h)&=K_h
  \end{align}
  for all $i\in I$, $h\in P^\vee$. Now define $\theta_q(\emptyset, \tau)=\psi \circ\tau\circ   \omega:\uqgp\rightarrow \uqgp$ as the $q$-deformation of
  $\theta(\emptyset, \tau)$. 
  If the minimal realization of $A$ is compatible with $\tau$ then $\theta_q(\emptyset, \tau)$ also defines an automorphism of $\uqg$. 
\subsection{A $q$-analog of $\mathrm{Ad}(m_X)$}  
  Assume now that $X\neq \emptyset$. Define 
  \begin{align*}
    T_X= T_{w_X}\circ \psi:\uqg\rightarrow \uqg
  \end{align*}
  where $\psi$ is given by \eqref{psi-Def}. The following lemma is a $q$-analogue  of statement (1) and the first part of (3) in Proposition \ref{Rousseau-Prop}. The first statement in the lemma is the reason for introducing the additional isomorphism $\psi$. 
  \begin{lem}\label{q-Rousseau1}
    (1) $ T_X \circ \tau \circ \omega|_{\cM_X}= \mathrm{id}|_{\cM_X}$.\\
    (2) $T_X$ commutes with $\tau$ as an automorphism of $\uqgp$.
  \end{lem}
  \begin{proof}
    The first claim follows immediately from Lemma \ref{Tw-on-X} and
    the definition \eqref{psi-Def} of $\psi$.
    To verify (2), observe that by definition of $T_i$ one has
    \begin{align}\label{tau-comp1}
       \tau(T_i(E_j))=T_{\tau(i)}(E_{\tau(j)})
    \end{align}
    for any $i,j\in I$, see \cite[37.1.3]{b-Lusztig94}. Let $w_X=r_{i_1}\dots r_{i_k}$ be a reduced
    expression. Then the relation $w_X=r_{\tau(i_1)}\dots r_{\tau(i_k)}$ and \eqref{tau-comp1} imply
    $\tau(T_{w_X}(E_j))=T_{w_X}(\tau(E_j))$. Similar relations hold
    for $F_j$ and $K_j$ and hence $T_{w_X}$ commutes with $\tau$. As
    $\psi$ commutes with $\tau$ so does $ T_{w_X}\circ \psi$.
  \end{proof}
  \begin{rema}
    If the minimal realization of $A$ is compatible with $\tau$ then $T_X$ and $\tau$ commute as automorphisms of $\uqg$.
  \end{rema}
\subsection{The quantum involution $\theta_q(X,\tau)$}
We are now in a position to define a quantum analog of the involutive automorphism
$\theta(X,\tau)$ defined by \eqref{tauDef}. Recall the definition of
the element $s(X,\tau)\in \tilde{H}$ from \eqref{sDef} and of the subgroup
$\Ad(\tilde{H})$ of $\Aut(\uqg)$ from Subsection \ref{sec:q-Aut}.  
\begin{defi}\label{defi:q-invol}
 We call the automorphism $\theta_q(X,\tau):\uqgp\rightarrow \uqgp$ defined by
 \begin{align*}
   \theta_q(X,\tau)=\Ad(s(X,\tau))\circ T_X \circ \tau\circ\omega
 \end{align*}
 the quantum involution corresponding to $(X,\tau)$.
\end{defi}
By Lemma \ref{Tw-on-Xc} for any $i\in I\setminus X$ there exist  $r\in \N_0$, monomials 
\begin{align}\label{Z-def}
  Z^-_i(X)&=F_{i_1}\dots F_{i_r}\in \cM^-_X, &
  Z^+_i(X)&=E_{j_1}\dots E_{j_r}\in \cM^+_X, 
\end{align}
and coefficients $a_i^+,$ $a_i^-\in \field(q)$ such that
\begin{align}\label{Z-prop}
  T^{-1}_{w_X}(F_iK_i)&=a_i^-\ad(Z^-_i(X))(F_iK_i),&
  T_{w_X}(E_i)&=a_i^+\ad(Z^+_i(X))(E_i).
\end{align}
By construction the quantum involution $\theta_q(X,\tau)$ is a $\field(q)$-algebra automorphism of $\uqgp$. It is not involutive, but by the following theorem it retains crucial properties of $\theta(X,\tau)$. 
\begin{thm}\label{thm:theta-props}
   The quantum involution $\theta_q(X,\tau)$ has the following properties: 
   
(1) $\theta_q(X,\tau)|_{\cM_X}=\mathrm{id}|_{\cM_X}$.  

(2) For all $\beta\in Q$ one has $\theta_q(X,\tau)(K_\beta)=K_{\Theta(\beta)}$.  

(3) For any $i\in I\setminus X$ there exist $u_i, v_i\in
\field(q)^\times$ such that 
  \begin{align*}
    \theta_q(X,\tau)(E_i)&=-u_i\,\sigma \big(\ad( Z^-_{\tau(i)}(X))(F_{\tau(i)}K_{\tau(i)})\big),\\
    \theta_q(X,\tau)(F_i K_i)&=-v_i\,\ad(Z^+_{\tau(i)}(X))(E_{\tau(i)}).
  \end{align*} 
  
\end{thm}
\begin{proof}
  Claim (1) follows from Lemma \ref{q-Rousseau1} and the fact that
  $\Ad(s(X,\tau))$ restricts to the identity on $\cM_X$. One obtains
  Claim (2) by the following calculation 
  \begin{align*}
    \theta_q(X,\tau)(K_\beta)= T_{w_X}(K^{-1}_{\tau(\beta)})=K^{-1}_{w_X( \tau(\beta))} =
    K_{\Theta(\beta)}. 
  \end{align*}
  To verify Claim (3) observe for $i\in I\setminus X$ the relation 
  \begin{align*}
    \theta_q(X,\tau)(F_i K_i)&\stackrel{\phantom{\eqref{Z-prop}}}{=}
    \Ad(s(X,\tau))\circ T_X\circ \tau\circ\omega( F_i K_i)\\ 
      &\stackrel{\phantom{\eqref{Z-prop}}}{=}\Ad(s(X,\tau))\circ
      T_X(- E_{\tau(i)} K_{\tau(i)}^{-1})\\ 
      &\stackrel{\phantom{\eqref{Z-prop}}}{=}\Ad(s(X,\tau))\circ T_{w_X}(-E_{\tau(i)} )\\ 
      &\stackrel{\eqref{Z-prop}}{=}-a_i^+ \Ad(s(X,\tau))
      \circ\ad(Z^+_{\tau(i)}(X))(E_{\tau(i)}). 
  \end{align*} 
The map $\Ad(s(X,\tau))$ multiplies by a nonzero scalar. This
proves the second formula in (3). The first formula follows
analogously from the first formula in \eqref{Z-prop} and relation 
\eqref{eq:T-invers}.  
\end{proof}
\begin{rema}
  If the minimal realization of $A$ is compatible with $\tau$ then $\theta_q(X,\tau)$ extends to an algebra automorphism of $\uqg$.
\end{rema}
\subsection{Commutation with $\omega$}\label{sec:omega-commute}
By Proposition \ref{Rousseau-Prop}.(3) the automorphism $\Ad(m_X)$ of $\gfrak$ commutes with the Chevalley involution $\omega$ of $\gfrak$. As shown below, however, the automorphism $T_X$ of $\uqg$ does not commute with the Chevalley involution $\omega$ of $\uqg$. Commutation with $\omega$ can be achieved by a slight modification of $T_X$ if on works with $\uqg$ defined over the field $\field(q^{1/2})$. In the following we make this explicit although commutation of $\omega$ and $T_X$ is not necessary for the construction of quantum symmetric pairs. The content of this subsection will not be used in the rest of the paper.  

For any $i \in I$ one has
\begin{align*}
  \omega\circ \psi(E_i)&=q_i^{-2} \psi\circ \omega(E_i),&
  \omega\circ \psi(F_i)&=q_i^{2} \psi\circ \omega(F_i).
\end{align*}
Hence one obtains the relation
\begin{align}\label{nu-commute}
  \psi \circ \omega=\omega\circ \psi \circ \Ad(\nu^2)
\end{align}
where $\nu \in \Htil$ is defined by $\nu(\alpha_i)=q_i$.

By \cite[37.2.4]{b-Lusztig94} one has
\begin{align}\label{eq:Ti-om}
  T_i(\omega(u))= q^{-(\alpha_i,\beta)}\omega (T_i(u)) \qquad \mbox{ for any } i\in I, u\in \uqg_\beta
\end{align}   
where the additional sign appearing in \cite[37.2.4]{b-Lusztig94} is conveniently hidden in our definition of $\omega$. Relation \eqref{eq:Ti-om} implies in particular that
 \begin{align}\label{omega-commute}
      T_w(\omega(u))= \big(\prod_{i\in I}q_i^{m_i}\big)\omega(T_w(u))
    \end{align}
for all $u \in \uqg_\beta$ and $w\beta-\beta=\sum_{i\in I}m_i \alpha_i$, see \cite[8.18(5)]{b-Jantzen96}. 
We abbreviate the above expression in the special case $w=w_X$. For any $i\in I$ define
\begin{align}\label{Q-def}
    Q_{X,i}=\prod_{j\in I} q_{j}^{m_j/2}\in \field(q^{1/2}) 
\end{align}
where $w_X\alpha_i-\alpha_i=\sum_{\alpha_j\in\pi} m_j \alpha_j$. For the rest of this subsection assume that $\uqg$ is defined over the field $\field(q^{1/2})$. Formula \eqref{omega-commute} implies
that
\begin{align}\label{omega-commute-special}
  T_{w_X} \circ \omega=\omega\circ T_{w_X}\circ \Ad(\eta_X^2)
\end{align}
where $\eta_X\in \Htil$ is defined by $\eta_X(\alpha_i)=Q_{X,i}$.
Combining \eqref{nu-commute} and \eqref{omega-commute-special} one obtains
\begin{align}
  \omega\circ\big( T_{w_X} \circ \psi \circ \Ad(\eta_X\nu)\big)
  &\stackrel{\phantom{\eqref{eq:Htil-om}}}{=} T_{w_X} \circ \omega \circ \psi \circ \Ad(\eta_X^{-1}\nu) \nonumber\\
  &\stackrel{\phantom{\eqref{eq:Htil-om}}}{=} T_{w_X}\circ \psi \circ \omega\circ \Ad(\eta_X^{-1}\nu^{-1})\nonumber\\
  &\stackrel{\eqref{eq:Htil-om}}{=}\big( T_{w_X} \circ \psi \circ \Ad(\eta_X\nu)\big)\circ\omega.\label{eq:tXpomega}
\end{align}
Define $T_X'=T_{w_X}\circ\psi\circ \Ad(\eta_X\nu)$. One can now formulate a version of Lemma \ref{q-Rousseau1} which includes commutation with $\omega$.
\begin{lem}
  (1) $ T_X' \circ \tau \circ \omega|_{\cM_X}= \mathrm{id}|_{\cM_X}$.\\
  (2) $T_X'$ commutes with $\tau$ and $\omega$ as an automorphism of $\uqgp$.
\end{lem}
\begin{proof}
  Property (1) follows from the fact that $\Ad(\eta_X \nu)|_{\cM_X}=\id|_{\cM_X}$ and from Lemma \ref{q-Rousseau1}.(1). Commutation of $\tau$ and $T_X'$ follows from $q_i Q_{X,i}=q_{\tau(i)}Q_{X,\tau(i)}$ and from Lemma \ref{q-Rousseau1}.(2). Commutation of $\omega$ and $T_X'$ was verified in \eqref{eq:tXpomega}. 
\end{proof}
\subsection{References to Letzter's constructions}\label{sec:ref-Letzter-qtheta}
For finite dimensional $\gfrak$, an automorphism $\tilde{\theta}_2$ very similar to $\theta_q(X,\tau)$ was constructed in \cite[Theorem 3.1]{a-Letzter99a}. The definition of $\tilde{\theta}_2$ contains an additional Hopf algebra automorphism $\chi$ which produces an additional parameter. In our setting this parameter will be introduced in the next subsection.  More importantly, the automorphism $\tilde{\theta}_2$ constructed in \cite[Theorem 3.1]{a-Letzter99a} is only a $\field$-algebra automorphism which maps $q$ to $q^{-1}$ and $K_\beta$ to $K_{-\Theta(\beta)}$ for all $\beta\in Q$.

Letzter developed the main body of her theory in \cite{a-Letzter99a} in terms of right coideal subalgebras of $\uqg$. In the subsequent papers \cite{a-Letzter00} and \cite{MSRI-Letzter} she changed conventions to left coideal subalgebras. This is convenient for the study of quantum Harish-Chandra modules as defined in \cite[Definition 3.1]{a-Letzter00} but it comes at the price that the definition of the quantum involutions given in \cite[Theorem 7.1]{MSRI-Letzter} involves the right adjoint action. Moreover, the connection with Lusztig's automorphisms, which was evident in \cite{a-Letzter99a}, only appears implicitly in the later papers, for example in the proof of \cite[Theorem 7.1]{MSRI-Letzter}.
\section{Quantum symmetric pairs}\label{sec:QSP}
To shorten notation we just write $\theta_q$ instead
of $\theta_q(X,\tau)$ to denote the quantum involution associated to
the fixed admissible pair $(X,\tau)$ by Definition \ref{defi:q-invol}. 
Recall that Corollary \ref{cor:U(g)generators} provides a set of generators of
the enveloping algebra $U(\kfrak)$ of the invariant Lie algebra $\kfrak$. Replacing these generators
by suitable elements in $U_q(\gfrak')$ we now define a quantum
analog of $U(\kfrak')$ as a right coideal subalgebra of
$U_q(\gfrak')$. For the elements \eqref{m-gens} and
\eqref{Ut0-gens} this is straightforward. Indeed, the generators $e_i,
f_i$ for $i\in X$ are replaced by $\cM_X$ while the generators $h\in \hfrak'$ with $\theta(h)=h$ are replaced by $K_h$ if $h\in Q^\vee$. For the generators $f_i+\theta(f_i)+s_i$ for $i\in I\setminus X$ from  \eqref{eq:UBis}, however, there exist families of possible $q$-analogs which are given by
\begin{align}\label{eq:Bi-def}
  &B_i= F_i + c_i \theta_q(F_i K_i)K_i^{-1} + s_i K_i^{-1}&&\mbox{for all $i\in I\setminus X$}
\end{align}
for suitable $(c_i)_{i\in I\setminus X}\in (\field(q)^\times)^{I\setminus X}$ and $(s_i)_{i\in I\setminus X}\in \field(q)^{I\setminus X}$. Following \cite[Section 7]{MSRI-Letzter} we will
call a $q$-analog of $U(\kfrak')$ \textit{standard} if $s_i=0$ for all $i\in I\setminus X$. If there are also generators of the form \eqref{eq:Bi-def} with $s_i\neq 0$ then the quantum analog of $U(\kfrak)$ will be called \textit{nonstandard}. In the following subsection we give rigorous definitions of quantum symmetric pairs coideal subalgebras. We then collect properties of their generators which will eventually, in Section \ref{sec:GensRels}, lead to a presentation of quantum symmetric pair coideal subalgebras in terms of generators and relations. 
\subsection{Definition of quantum symmetric pair coideal subalgebras $B_{\uc,\us}$}\label{sec:standard}
Define $Q^{\Theta}=\{\beta \in Q \,|\, \Theta(\beta)=\beta\}$
and let $U^0_\Theta{}'$ denote the subalgebra of $U^0{}'$ generated by the elements $K_\beta$ for all $\beta\in Q^{\Theta}$.
\begin{defi}
  For any $\uc=(c_i)_{i\in I\setminus X}\in(\field(q)^\times)^{I\setminus X}$ and $\us=(s_i)_{i\in I\setminus X}\in \field(q)^{I\setminus X}$ define $B_{\uc,\us}=B_{\uc,\us}(X,\tau)$ to be the subalgebra of $\uqgp$ generated by $\cM_X$, ${U^0_\Theta}{}'$, and the elements \eqref{eq:Bi-def}
for all $i\in I\setminus X$.
\end{defi}
  The algebra $B_{\uc,\us}$ is a quantum analog of $U(\kfrak')$ only for suitable parameters $\uc$ and $\us$. In the remainder of this subsection the necessary restrictions on the parameters will be elaborated. Independently of these restrictions, however, $B_{\uc,\us}$ is always a coideal of $\uqgp$.
\begin{prop}\label{prop:coid}
  Let $\uc\in (\field(q)^\times)^{I\setminus X}$ and $\us\in \field(q)^{I\setminus X}$. Then $B_{\uc,\us}$ is a right coideal subalgebra of $\uqgp$.  
\end{prop}
\begin{proof}
  Clearly, $\cM_X$ and $U^0_\Theta{}'$ are Hopf subalgebras of $\uqgp$. It remains to show that 
  \begin{align}\label{eq:kowBigoal}
    \kow(B_i)\in B_{\uc,\us}\ot \uqgp 
  \end{align} 
   for all $i\in I\setminus X$. Theorem \ref{thm:theta-props} implies that there exists $v_i\in \field(q)^\times$ such that
  \begin{align}\label{eq:BiZi}
    B_i = F_i - c_i v_i\, \ad(Z_{\tau(i)}^+(X))(E_{\tau(i)}) K_i^{-1} + s_i K_i^{-1}.
  \end{align}
  If one applies the relation 
  \begin{align*}
    \kow(\ad(x)(u)) = x_{(1)} u_{(1)} S(u_{(3)}) \ot \ad(x_{(2)})(u_{(2)})
  \end{align*}
  to $x=Z^+_{\tau(i)}(X)$ and $u=E_{\tau(i)}$ then one obtains in view of Equation \eqref{eq:E-copr} the relation
  \begin{align}\label{eq:kowZi}
    \kow\big(\ad(Z_{\tau(i)}^+(X))(E_{\tau(i)})\big) - \ad(Z_{\tau(i)}^+(X))(E_{\tau(i)})\ot 1 \in \cM^+_X K_{\tau(i)}\ot \uqgp.
  \end{align}
  Relations  \eqref{eq:BiZi} and \eqref{eq:kowZi} together imply that
  \begin{align}\label{eq:Bi-kow}
    \kow(B_i)-B_i\ot K_i^{-1}\in \cM^+_X {U^0_\Theta}' \ot \uqgp
  \end{align}
  and hence relation \eqref{eq:kowBigoal} holds for all $i\in I\setminus X$.
\end{proof}
A desirable condition for $B_{\uc,\us}$ to qualify as a quantum analog of $U(\kfrak')$ is that $B_{\uc,\us}\cap U^0{}'=U^0_\Theta{}'$. Lemmas \ref{lem:towards-standard} and \ref{lem:towards-nonstandard} below impose restrictions on the parameters $\uc$ and $\us$ for which this condition is satisfied.
\begin{lem}\label{lem:cond2}
  Let $i\in I\setminus X$ such that $\tau(i)\neq i$ and
  $(\alpha_i,\Theta(\alpha_i))=0$. Then
  $(\alpha_i,\alpha_{\tau(i)})=0$ and 
  $\Theta(\alpha_i)=-\alpha_{\tau(i)}$.  
\end{lem}
\begin{proof}
  As $\tau(i)\neq i$ and $(\alpha_i,\alpha_j)\le 0$ for all $j\neq i$
  the relation
  $(\alpha_i,-w_X\alpha_{\tau(i)})=(\alpha_i,\Theta(\alpha_i))=0$
  implies that
  \begin{align}\label{eq:aiati}
    (\alpha_i,\alpha_{\tau(i)})=0.
  \end{align}
  Moreover, $ \alpha_i+\Theta(\alpha_{\tau(i)})=\alpha_i-w_X\alpha_i\in \Z X$
  and hence 
  \begin{align}\label{eq:tau-inv}
     \alpha_i+\Theta(\alpha_{\tau(i)})=\Theta(\alpha_i+\Theta(\alpha_{\tau(i)}))
     =\alpha_{\tau(i)}+\Theta(\alpha_i). 
  \end{align}
  Equations \eqref{eq:aiati}, \eqref{eq:tau-inv}, and the assumption
  $(\alpha_{\tau(i)},\Theta(\alpha_{\tau(i)}))=(\alpha_i,\Theta(\alpha_i))=0$ imply
  \begin{align*}
    (\alpha_{\tau(i)}+\Theta(\alpha_i),\alpha_{\tau(i)}+\Theta(\alpha_i))
      =(\alpha_{\tau(i)}+\Theta(\alpha_i),\alpha_i+\Theta(\alpha_{\tau(i)}))=0
  \end{align*}
  and hence $\Theta(\alpha_i)=-\alpha_{\tau(i)}$. 
\end{proof}
\begin{lem}\label{lem:towards-standard}
  Let $\cB$ be a right coideal subalgebra of $\uqg$. Let $i,j\in I$ such
  that $(\alpha_i,\alpha_j)=0$, $(\alpha_i,\alpha_i)=(\alpha_j,\alpha_j)$, and 
  \begin{align*}
    F_i - c_i E_j K_i^{-1}\in \cB, \quad 
    F_j - c_j E_i K_j^{-1}\in \cB
  \end{align*}
  for some $c_i,c_j\in \field(q)^\times$. If $c_i\neq c_j$ then
  $(K_iK_j)^{-1}\in \cB$.
\end{lem}
\begin{proof}
  The claim follows from the relation
  \begin{align}
      [F_i - c_i E_j K_i^{-1}&, F_j - c_j E_i K_j^{-1}]\nonumber\\
        &= - c_i\frac{K_j-K_j^{-1}}{q_j-q_j^{-1}}K_i^{-1} + c_j
        \frac{K_i-K_i^{-1}}{q_i-q_i^{-1}} K_j^{-1}  \label{eq:BiBj-com}
  \end{align}
  by applying the coproduct to the right hand side.
\end{proof}
By the above lemma and Lemma \ref{lem:cond2} the condition $B_{\uc,\us}\cap U^0{}'=U^0_\Theta{}'$ can only be satisfied if the parameter $\uc$ is contained in the set
\begin{align}\label{eq:C-def}
  \cC=\{\uc\in (\field(q)^\times)^{I\setminus X}\,|\, \mbox{$c_i=c_{\tau(i)}$ if
    $\tau(i)\neq i$ and $(\alpha_i,\Theta(\alpha_i))=0$} \}.
\end{align}
The parameters $\us$ are also subject to restrictions. Define
\begin{align}\label{eq:Ins}
  I_{ns}=\{i\in I\setminus X\,|\,\tau(i)=i \mbox{ and } \alpha_i(h_j)=0\, \forall j\in X\}.
\end{align}
By the following lemma it is only reasonable to allow $s_i\neq 0$ for $i\in I_{ns}$.
\begin{lem}\label{lem:towards-nonstandard}
  Let $\cB\subseteq \uqgp$ be a subalgebra which contains $U^0_\Theta{}'$ and the element $B_{i}$ defined by \eqref{eq:Bi-def} for some $c_i,s_i\in \field(q)^\times$, $i\in I\setminus X$. If $i\notin I_{ns}$ then $K_i^{-1}\in \cB$.  
\end{lem}
\begin{proof}
  If $\tau(i)\neq i$ then $K_i K_{\tau(i)}^{-1}\in U^0_\Theta{}'$. Conjugating $B_{i}$ by this element one obtains $K_i^{-1}\in \cB$. Similarly, if $\alpha_i(h_j)\neq 0$ for some $j\in X$, then one conjugates by $K_j\in U^0_\Theta{}'$ to obtain $K_i^{-1}\in \cB$.
\end{proof}
It will turn out that the condition $s_i=0$ if $i\notin I_{ns}$ is not sufficient to ensure $B_{\uc,\us}\cap U^0{}' = U^0_\Theta{}'$. Consider the set
\begin{align}\label{eq:Sdef}
  \cS=\{\us\in \field(q)^{I\setminus X}\,|\,s_i\neq 0 \Rightarrow \, (i\in  I_{ns} \mbox{ and } a_{ij}\in -2\N_0 \forall j\in I_{ns}\setminus \{i\})\}.
\end{align}
The relevance of the additional condition $a_{ij}\in -2\N_0$ in the definition of $\cS$ will become apparent in the proof of relation \eqref{eq:p00inTheta}, see also Remark \ref{rem:S}.
\begin{defi}\label{def:qsp}
  We call $B_{\uc,\us}$ for $\uc\in \cC$ and $\us \in \cS$ a quantum symmetric pair coideal subalgebra of $\uqgp$. If $\us=\mathbf{0}=(0,0,\dots,0)$ then $B_{\uc}=B_{\uc,{\mathbf{0}}}$ is called {\upshape standard}. If $s_i\neq 0$ for some $i  \in I\setminus X$ then $B_{\uc,\us}$ is called {\upshape nonstandard}. 
\end{defi}
\begin{rema}\label{rema:ad-hoc}
  The construction of the quantum symmetric pair coideal subalgebras $B_{\uc,\us}$ may seem rather ad hoc, since they are defined by giving explicit quantum analogs of the generators of $U(\kfrak')$. It will be shown in Section \ref{sec:specialize} that $B_{\uc,\us}$ specializes to $U(\kfrak')$ and that $B_{\uc,\us}$ is maximal with this property. For finite dimensional $\gfrak$, Letzter showed that any coideal subalgebra of $\uqg$ with these two properties has to be of the form $B_{\uc,\us}$, see \cite[Theorem 5.8]{a-Letzter99a}, \cite[Theorem 7.5]{MSRI-Letzter}. An analog of Letzter's classification result in the Kac-Moody case would provide a stronger justification for the definition of quantum symmetric pair coideal subalgebras. This problem is left for future work. 
\end{rema}
\begin{rema}\label{rem:DefComp}
  The algebra $B_{\uc,\us}$ for $\uc\in \cC$, $\us \in \cS$ is a quantum analog of $U(\kfrak')$, see Theorem \ref{thm:Bcspecial}. If the minimal realization of $A$ is compatible with $\tau$ then one may replace $U^0_\Theta{}'$ by the algebra $U^0_\theta=\field(q)\langle K_h\,|\,h\in P^\vee,\, \theta(h)=h \rangle$ to obtain a quantum analog of $U(\kfrak)$.
\end{rema}
\subsection{Decompositions and projections for $\uqg$}\label{sec:DecProj}
The triangular decomposition \eqref{eq:triang-decomp} for $\uqg$ implies that the multiplication map gives an isomorphism of vector spaces
\begin{align}
 U^+\ot U^0 \ot S(U^-) \cong \uqg.
\end{align}
This leads to a direct sum decomposition
\begin{align}\label{eq:K-sum}
          \uqg=\mathop{\oplus}_{h\in P^\vee} U^+ K_h S(U^-)
        \end{align}
For any $h\in P^\vee$ let $P_h:\uqg\rightarrow  U^+ K_h S(U^-)$ denote the projection with respect to this decomposition. It follows from the formulas for the coproduct of $\uqg$ that the map $P_h$ is a homomorphism of left $\uqg$-comodules, that is
\begin{align}\label{eq:P-comod-hom}
  \kow\circ P_h(x)=(\id \ot P_h)\kow(x) \qquad  \mbox{for all $x\in \uqg$.}
\end{align}
This relation implies the following grading of right coideal subalgebras of $\uqg$.
\begin{lem}\label{lem:PlamB}
   Let $B$ be a right coideal subalgebra of $\uqg$. Then
   $B=\oplus_{h\in P^\vee} P_h(B)$.
\end{lem}
\begin{proof}
  Let $b\in B$ and $h\in P^\vee$. Relation \eqref{eq:P-comod-hom} implies that $\kow(P_h(b))= b_{(1)}\ot P_h(b_{(2)})\in B\ot\uqg$. Application of $\id \ot \vep$ implies $P_h(b)\in B$. 
\end{proof}
We also use the symbol $P_\lambda$ for $\lambda\in Q$ to denote the projection $P_\lambda:\uqgp \rightarrow U^+ K_\lambda S(U^-)$ obtained as above.

We may also consider the direct sum decomposition 
\begin{align}\label{eq:Upm-sum}
  \uqg=\mathop{\oplus}_{\alpha,\beta\in Q^+} U_\alpha^+ U^0 U^-_{-\beta}.
\end{align}
Let $\pi_{\alpha,\beta}:\uqg\rightarrow U_\alpha^+ U^0 U^-_{-\beta}$ denote the projection with respect to this decomposition.
\subsection{Quantum Serre relations for $B_{\uc,\us}$}\label{sec:q-Serre}
All through this subsection fix $\uc\in \cC$ and $\us \in \cS$ and consider the corresponding quantum symmetric pair coideal subalgebra $B_{\uc,\us}$. Recall the definition of the noncommutative polynomials $F_{ij}$ given for any $i,j\in I$ by \eqref{eq:Fij-def}.
The next lemmas collect properties of $F_{ij}$ evaluated on the elements relevant for
the construction of $B_{\uc,\us}$.
\begin{lem}
 The following relations hold for all $i,j\in I$:
\begin{align}
  F_{ij}(F_iK_i,F_jK_j)&=0\label{eq:Serre1},\\
  F_{ij}(\theta_q(F_iK_i)K_i^{-1},\theta_q(F_jK_j)K_j^{-1})&=0,\label{eq:Serre2}\\
  F_{ij}(F_i,K_j^{-1})&=0\label{eq:Serre2a}. 
\end{align}
\end{lem}
\begin{proof}
  Property \eqref{eq:Serre1} follows from \eqref{eq:q-Serre} because
  \begin{align}\label{eq:q-factors}
    (F_iK_i)^{1-a_{ij}-n} F_jK_j (F_i K_i)^n= F_i^{1-a_{ij}-n} F_j F_i^n K_i^{1-a_{ij}}K_j
  \end{align}
  if $0 \le n \le 1-a_{ij}$. As $\theta_q$ is an algebra automorphism one obtains
  \begin{align*}
    F_{ij}(\theta_q(F_iK_i),\theta_q(F_jK_j))&=0 \qquad \mbox{for all $i,j\in I$.}
  \end{align*}
  This implies \eqref{eq:Serre2} by a calculation analog to \eqref{eq:q-factors}. Finally, to verify \eqref{eq:Serre2a} note that the relation $F_i^{1-a_{ij}-n}K_j^{-1} F_i^n = q_i^{a_{ij}n} F_j^{1-a_{ij}}K_i^{-1}$ implies
  \begin{align*}
    F_{ij}(F_i,K_j^{-1})&=\sum_{n=0}^{1-a_{ij}}(-1)^n q_i^{a_{ij}n}
                        \left[\begin{matrix}1-a_{ij}\\n\end{matrix}\right]_{q_i} F_i^{1-a_{ij}}K_j^{-1}=0
  \end{align*}
  by \cite[0.2.(4)]{b-Jantzen96}. 
\end{proof}
We extend the definition of $B_i$ given in \eqref{eq:Bi-def} for $i\in I\setminus X$ to all elements of
$I$ by defining  $B_i:=F_i$ for $i\in X$.
\begin{lem}\label{lem:Serre2}
 The following relations hold in $\uqg$:
\begin{align}
  F_{ij}(B_i,B_j)&=0&& \mbox{for all $i\in X$, $j\in I$,}\label{eq:Serre3}\\
  \pi_{0,0}(F_{ij}(B_i,B_j))&\in U^0_\Theta{}' &&\mbox{for all $i,j\in
    I$.}\label{eq:p00inTheta}
\end{align}
\end{lem}
\begin{proof}
  For $i\in X, j\in I$ one has $\theta_q(F_iK_i)=F_iK_i$ and hence
  \begin{align*}
    F_{ij}(B_i,B_j) &\stackrel{\phantom{\eqref{eq:Serre2a}}}{=} F_{ij}(F_i, F_j+c_j\theta_q(F_jK_j)K_j^{-1} + s_j K_j^{-1})\\
     &\stackrel{\eqref{eq:Serre2a}}{=} F_{ij}(F_i,F_j+ c_j\theta_q(F_jK_j)K_j^{-1})\\
       &\stackrel{\eqref{eq:q-Serre}}{=} c_j F_{ij}(F_i, \theta_q(F_jK_j)K_j^{-1})\\
      &\stackrel{\phantom{\eqref{eq:Serre2a}}}{=}c_j F_{ij}(\theta_q(F_iK_i)K_i^{-1},\theta_q(F_jK_j)K_j^{-1})\stackrel{\eqref{eq:Serre2}}{=}0.
  \end{align*}
  This proves \eqref{eq:Serre3}. We will now verify relation \eqref{eq:p00inTheta} in several steps. By \eqref{eq:Serre3} we may assume that $i\notin X$.
  
  \noindent{\bf Step 1:} $\pi_{0,0}(F_{ij}(B_i,B_j))=0$ if $i\in I$, $j\in X$.\\
  Indeed, assume that $\pi_{0,0}(F_{ij}(B_i,B_j))\neq 0$ for some $i\in I\setminus X$, $j\in X$. For weight reasons this is only possible if $\tau(i)=i$, $a_{ij}=-1$, and $w_X(\alpha_i)=\alpha_i+\alpha_j$. The latter relation, however, implies that $(\alpha_i+\alpha_j)(h_k)=0$ for all $k\in X\setminus\{j\}$. Hence $\alpha_i(h_k)=\alpha_j(h_k)=0$ for all $k\in X\setminus \{j\}$. This, in turn, implies that $w_X(\alpha_i)=s_j(\alpha_i)=\alpha_i - a_{ji} \alpha_j$ and hence $a_{ji}=-1$. But then $\alpha_i(\rho^\vee_X)=\alpha_i(h_j/2)=a_{ji}/2=-1/2$ which contradicts the fact that $(X,\tau)$ is an admissible pair. 
  
  By Step 1 we may from now on assume that neither $i$ nor $j$ are contained in $X$.
  
  \noindent{\bf Step 2:}  $\pi_{0,0}(F_{ij}(B_i,B_j))\in U^0_\Theta{}'$ if $i,j\notin I_{ns}$.\\
  Again, in this case $\pi_{0,0}(F_{ij}(B_i,B_j))\neq 0$ implies that $F_{ij}(B_i,B_j)$
  has a zero weight summand. By definition of $F_{ij}$ and $B_i,B_j$,
  however, this is only possible if $a_{ij}=0$ and
  $\Theta(\alpha_i)=-\alpha_j$. In this case
  $c_i=c_j$ by definition \eqref{eq:C-def} of $\cC$ and hence
  $B_i=F_i - c_i E_j K_i^{-1}$ and $B_j=F_j -c_j E_i K_j^{-1}$. Now relation
  \eqref{eq:BiBj-com} implies that $F_{ij}(B_i,B_j)\in U^0_\Theta{}'$.
  
  \noindent{\bf Step 3:}  $\pi_{0,0}(F_{ij}(B_i,B_j))=0$ if ($i\notin I_{ns}$ and $j\in I_{ns}$) or ($i\in I_{ns}$ and $j\notin I_{ns}$).\\
  This holds for weight reasons. 
  
  \noindent{\bf Step 4:}  $\pi_{0,0}(F_{ij}(B_i,B_j))=0$ if $i,j\in I_{ns}$ and $-a_{ij}\in 2\N_0$.\\
  Observe first that in this case 
  \begin{align}\label{eq:pi-pi}
    \pi_{0,0}\left( F_{ij}(B_{i},B_{j}) \right)=s_j\pi_{0,0}\left( F_{ij}(B_{i},K_j^{-1}) \right)
  \end{align}
  for weight reasons. Moreover, for $-a_{ij}\in 2\N_0$ the non-commutative polynomial \eqref{eq:Fij-def} can be written as
  \begin{align}\label{eq:Fij-even}
    F_{ij}(x,y)=\sum_{n=0}^{-a_{ij}/2}(-1)^n 
                          \left[\begin{matrix}1-a_{ij}\\n\end{matrix}\right]_{q_i} \Big( x^{1-a_{ij}-n} y x^n - x^n y x^{1-a_{ij}-n}\Big).
  \end{align}
  By \eqref{eq:pi-pi} and \eqref{eq:Fij-even} it suffices to show that
  \begin{align}\label{eq:Goal}
    \pi_{0,0}\left( B_{i}^m K_j^{-1} B_{i}^n - B_{i}^n K_j^{-1} B_{i}^m\right)=0
  \end{align}
  if $n+m$ is odd.
  
  Let $\Mfrak$ denote the free monoid generated by symbols $E,F,K^{-1}$ and let $\iota:\Mfrak\rightarrow\Mfrak$ denote the monoid anti-automorphism defined by $\iota(E)=F$, $\iota(F)=E$, and $\iota(K^{-1})=K^{-1}$. Let, moreover, $\ell:\Mfrak\rightarrow \N_0$ denote the length function and let $\pi:\Mfrak\rightarrow \uqgp$ denote the monoid homomorphisms defined by $\pi(E)=E_i K_i^{-1}$, $\pi(F)=F_i$, and $\pi(K^{-1})=K_i^{-1}$. To verify \eqref{eq:Goal} it suffices to show that
  \begin{align}\label{eq:Goal2}
    \pi_{0,0}\left( \pi(u) K_j^{-1} \pi(v) - \pi(\iota(v))K_j^{-1} \pi(\iota(u))\right)=0
  \end{align} 
for all $u,v\in \Mfrak$ with $\ell(u)=m$ and $\ell(v)=n$. The above relation follows from $\pi_{0,0}(\pi(w))=\pi_{0,0}(\pi(\iota(w)))$ which holds for all $w\in \Mfrak$. This proves Step 4, and hence completes the proof of Equation \eqref{eq:p00inTheta}.
\end{proof}
\begin{rema}\label{rem:S}
  Assume, contrary to the definition of $\cS$ given by \eqref{eq:Sdef}, that $s_j\neq 0$ for some $i,j\in I_{ns }$ with $-a_{ij}$ odd. For $\lambda_{ij}=(1-a_{ij})\alpha_i + \alpha_j$ one then obtains 
  \begin{align*}
    P_{-\lambda_{ij}}\circ \pi_{0,0}\left(F_{i,j}(B_i,K_j^{-1})\right)\neq 0
  \end{align*}
  in general. It is straightforward to verify this for $-a_{ij}=1$ or $-a_{ij}=3$ by direct computation. Hence in this case
   \begin{align*}
    P_{-\lambda_{ij}}\circ \pi_{0,0}\left(F_{i,j}(B_i,B_j)\right)\neq 0.
  \end{align*}
  As will be seen in the proof of Proposition \ref{prop:Z=0}, the above relation would imply that $B_{\uc,\us}\cap U^0{}'\neq U^0_\Theta{}'$. This is the reason why we restrict to parameters $\us$ in the set $\cS$ given by \eqref{eq:Sdef}
\end{rema}
\begin{rema}
  By the above proof, relation \eqref{eq:p00inTheta} can be refined. Indeed, on has $\pi_{0,0}\left(F_{i,j}(B_i,B_j)\right)=0$ unless $\Theta(\alpha_i)=-\alpha_j$ and $a_{ij}=0$ in which case $c_i=c_j$, $s_i=s_j=0$, and
\begin{align*}  
  F_{ij}(B_i,B_j)= c_i \frac{K_i K_j^{-1}-K_j K_i^{-1}}{q_i-q_i^{-1}}
\end{align*}
as calculated in the proof of Lemma \ref{lem:towards-standard}.
\end{rema}
The following technical lemma will be used in the proof of Proposition \ref{prop:Z=0}.
\begin{lem}\label{lem:not-Pt}
  Let $\alpha,\beta\in Q^+$. If
  $\pi_{\alpha,\beta}(F_{ij}(B_i,B_j))\neq 0$ then
  $\lambda_{ij}-\alpha\notin Q^\Theta$ and $\lambda_{ij}-\beta\notin
  Q^\Theta$. 
\end{lem}
\begin{proof}
  By \eqref{eq:Serre3} there is nothing to show if $i\in X$. Hence we may assume that $i\notin X$. Consider first the case that $j\in X$. Then \eqref{eq:Serre2} implies
  \begin{align*}
     0=F_{ij}(\theta_q(F_iK_i)K_i^{-1},\theta_q(F_j K_j)K_j^{-1})=F_{ij}(\theta_q(F_iK_i)K_i^{-1},F_j).
  \end{align*}
  Hence, if $\pi_{\alpha,\beta}(F_{ij}(B_i,B_j))\neq 0$ for some $\alpha,\beta\in Q^+$ then $0\le\beta\le\lambda_{ij}-\alpha_i$ and $0\le \alpha\le -\Theta(\lambda_{ij}-\alpha_i)$. This implies that
  \begin{align*}
    \lambda_{ij} - \beta \ge \alpha_i, \qquad -\Theta(\lambda_{ij})-\alpha \ge -\Theta(\alpha_i).
  \end{align*}
  As $\Theta(\alpha_i)\in -Q^+$ one gets $\lambda_{ij}-\beta \notin Q^\Theta$ and $-\Theta(\lambda_{ij})-\alpha\notin Q^\Theta$. As $\lambda_{ij}+\Theta(\lambda_{ij})\in Q^\Theta$ the second relation implies that $\lambda_{ij}-\alpha\notin Q^\Theta$.

  Finally, we turn to the case that $j\notin X$.  
  By \eqref{eq:q-Serre} and \eqref{eq:Serre2} the relation
  $\pi_{\alpha,\beta}(F_{ij}(B_i,B_j))\neq 0$ implies $0\le
  \beta<\lambda_{ij}$ and $\alpha=\alpha'+\alpha''$ with
  $0\le\alpha'<\tau(\lambda_{ij})$ and $\alpha''\in \sum_{i\in
    X}\N_0\alpha_i$. Hence $\lambda_{ij}-\beta$ and
  $\tau(\lambda_{ij})-\alpha'$ are nonzero elements in
  $\N_0\alpha_i+\N_0\alpha_j$ and $\N_0 \alpha_{\tau(i)}+ \N_0 \alpha_{\tau(j)}$, respectively. Now the relations $\Theta(\alpha_k)\in -Q^+$ for $k\in I\setminus X$ and $\alpha'', \lambda_{ij}-\tau(\lambda_{ij})\in Q^\Theta$ imply the claim.
\end{proof}
\begin{lem}\label{lem:rel1}
   The following relations hold in $B_{\uc,\us}$:
  \begin{align}
    K_\lambda B_i&=q^{-(\alpha_i,\lambda)}B_i K_\lambda &&\mbox{for all
    $i\in I$, $\lambda\in Q^\Theta$, } \label{eq:qKB}\\
    E_j B_i- B_i E_j & = \delta_{ij}
    \frac{K_i-K_i^{-1}}{q_i-q_i^{-1}} & & \mbox{for all $i\in I, j\in
      X$.}\label{eq:EBBE}
  \end{align}
\end{lem}
\begin{proof}
  Recall from \eqref{eq:invTheta} that $(\alpha,\beta)=(\Theta(\alpha),\Theta(\beta))$ holds for all $\alpha,\beta \in Q$. In particular one has $(\alpha_i,\lambda)=(\Theta(\alpha_i),\lambda)$ if $\lambda\in Q^\Theta$. This implies \eqref{eq:qKB}.
  
  Relation \eqref{eq:EBBE} holds for $i\in X$ by the defining relation \ref{sec:quagroup}.(4) of $\uqg$.
  For $i\notin X$ Equation \eqref{eq:EBBE} follows from
  \begin{align*}
    [E_j, \theta_q(F_iK_i)K_i^{-1}]=\theta_q\big([E_j, F_i K_{\alpha_i-\Theta(\alpha_i)}]\big)=0
  \end{align*}  
  which holds as $(\alpha_j,\alpha_i)=(\alpha_j,\Theta(\alpha_i))$ by \eqref{eq:invTheta}.
\end{proof}
For any $J=(j_1,\dots,j_n)\in I^n$ define $\wght(J)=\sum_{i=1}^n\alpha_{j_i}$ and 
\begin{align}\label{eq:EJFJBJ}
  E_J&=E_{j_1}\dots E_{j_n},& F_J&=F_{j_1}\dots F_{j_n}, &B_J&=B_{j_1}\dots B_{j_n}.
\end{align}
Fix $i,j\in I$ and recall that $\lambda_{ij}=(1-a_{ij})\alpha_i + \alpha_j$. We want to further investigate properties of $Y:=F_{ij}(B_i,B_j)$. By relation \eqref{eq:Bi-kow} and Lemma \ref{lem:rel1} one has
\begin{align}\label{eq:kowY}
  \kow(Y)\in Y\ot K_{-\lambda_{ij}}+
  \sum_{\{J\,|\,\wght(J)<\lambda_{ij}\}} \cM_X^+ U^0_\Theta{}' B_J\ot \uqgp.
\end{align}
For later use we observe the following properties of the second term
\begin{align}
  S(U^-)U^0{}'\cap
  \sum_{\{J\,|\,\wght(J)<\lambda_{ij}\}} \cM_X^+ U^0_\Theta{}' B_J &= S(U^-)U^0{}'\cap\cM_XU^0_\Theta{}' =
  S(\cM^-_X)U^0_\Theta{}',\label{eq:intersect1}\\
  U^+U^0{}'\cap
  \sum_{\{J\,|\,\wght(J)<\lambda_{ij}\}} \cM_X^+U^0_\Theta{}' B_J &=U^+ U^0{}'\cap\cM_XU^0_\Theta{}' =
  \cM^+_XU^0_\Theta{}'\label{eq:intersect2}
\end{align}
which follow from the linear independence of the terms $F_J$
for all $J$ with $\wght(J)<\lambda_{ij}$.
The following proposition provides the main tool to write the algebra
$B_{\uc,\us}$ efficiently in terms of generators and relations.
\begin{prop}\label{prop:Z=0}
In $\uqg$ one has $P_{-\lambda_{ij}}(F_{ij}(B_i,B_j))=0$ for all $i,j\in I$.
\end{prop}
\begin{proof}
  To abbreviate notation set
  $Z:=P_{-\lambda_{ij}}(F_{ij}(B_i,B_j))$. Relations
  \eqref{eq:P-comod-hom} and \eqref{eq:kowY} imply 
\begin{align}\label{eq:kowZ}
  \kow(Z)\in Y\ot K_{-\lambda_{ij}}+
  \sum_{\{J\,|\,\wght(J)<\lambda_{ij}\}} \cM^+_XU^0_\Theta{}' B_J\ot
  U^+ K_{-\lambda_{ij}} S(U^-).
\end{align}
Assume now, that $Z \neq 0$. Choose $\alpha\in Q^+$ maximal with
respect to the partial order such that $\pi_{\alpha,\beta}(Z)\neq 0$
for some $\beta\in Q^+$. In this case
\begin{align*}
  0\neq (\id\ot \pi_{\alpha,0})\kow(Z)\in
  S(U^-)K_{-\lambda_{ij}+\alpha}\ot U^+_\alpha K_{-\lambda_{ij}}
\end{align*} 
by \eqref{eq:E-copr}, \eqref{eq:F-copr}. If $\alpha\neq 0$ then
relations \eqref{eq:kowZ} and \eqref{eq:intersect1} imply
$K_{-\lambda_{ij}+\alpha}\in U^0_\Theta{}'$ in contradiction to Lemma
\ref{lem:not-Pt}. Hence $\alpha=0$ and $Z\in S(U^-)K_{-\lambda_{ij}}$.
Now choose $\beta\in Q^+$ maximal such that $\pi_{0,\beta}(Z)\neq
0$. In this case
\begin{align*}
  0\neq (\id\ot \pi_{0,\beta})\kow(Z)\in
  K_{-\lambda_{ij}+\beta}\ot S(U^-_\beta) K_{-\lambda_{ij}}
\end{align*} 
As before, relations \eqref{eq:kowZ} and \eqref{eq:intersect2} together
with Lemma \ref{lem:not-Pt} imply $\beta=0$. Hence
$Z=\pi_{00}(Z)$. But then $0\neq \pi_{0,0}(Z)\in \field(q) K_{-\lambda_{ij}}\nsubseteq
U^0_\Theta{}'$ in contradiction to Equation \eqref{eq:p00inTheta}. Hence $Z=0$.
\end{proof}
\begin{cor}\label{cor:Serre-lower}
  In $B_{\uc,\us}$ one has the relation 
  \begin{align}
    F_{ij}(B_i,B_j)\in \sum_{\{J\in \cJ\,|\,\wght(J)<\lambda_{ij}\}} \cM_X^+ U^0_\Theta{}' B_J \qquad \mbox{ for all $i,j\in I$}.
  \end{align} 
\end{cor}
\begin{proof}
  This follows from Proposition \ref{prop:Z=0} by applying the counit
  to the second tensor factor in \eqref{eq:kowZ} where $Z=P_{-\lambda_{ij}}(F_{ij}(B_i,B_j))=0$.
\end{proof}
\begin{rema}\label{rem:s-indep}
  For $i\in I_{ns}$ one has
\begin{align}\label{eq:kowBis}
  \kow(B_{i}) = B_{i} \ot K_i^{-1} + 1\ot (F_i - c_i E_i K_i^{-1})
\end{align}
 and thus $\kow(B_i)$ does not explicitly depend on the parameter $s_i$ but only on $c_i$. By the proof of the above theorem, this implies that the expression of
 $ F_{ij}(B_i,B_j)$ as an element of $\sum_{\{J\in \cJ\,|\,\wght(J)<\lambda_{ij}\}} \cM_X^+ U^0_\Theta{}' B_J$ does not explicitly depend on $\us$.
\end{rema}
\subsection{References to Letzter's constructions}
  For finite dimensional $\gfrak$, standard and nonstandard quantum symmetric pairs were introduced in full generality in \cite[Variants 1 and 2, after (7.25)]{MSRI-Letzter}. Previously, the parameters $\uc\in \cC$ were implicitly included via a Hopf algebra automorphism $\chi$ which entered the definition of the automorphism $\tilde{\theta}_2$ in \cite[Theorem 3.1]{a-Letzter99a}. The nonstandard quantum symmetric pair coideal subalgebras seem not to be defined explicitly in that paper, however, their existence is already observed in \cite[(5.14) and Remark 5.10]{a-Letzter99a}. In \cite[Section 2]{a-Letzter00} nonstandard analogs were defined for a parameter set larger than $\cS$ but this was corrected at the end of \cite[Variant 2]{MSRI-Letzter}.
  
  The discussion in the present subsection partly follows \cite[Section 7]{MSRI-Letzter}. Proposition \ref{prop:coid} is \cite[Theorem 7.2]{MSRI-Letzter} or \cite[Corollary 4.2]{a-Letzter99a}. Subsection \ref{sec:q-Serre} follows the procedure outlined in \cite[after Theorem 7.2]{MSRI-Letzter} to describe quantum symmetric pair coideal subalgebras in terms of generators and relations. In particular, Lemma \ref{lem:not-Pt} corresponds to \cite[Lemma 7.3]{MSRI-Letzter} and Lemma \ref{lem:rel1} and Corollary \ref{cor:Serre-lower} are subsumed in \cite[Theorem 7.4]{MSRI-Letzter} in the finite case. The proof of Proposition \ref{prop:Z=0} follows the proof of \cite[Theorem 7.4]{MSRI-Letzter}. A discussion similar to Lemma \ref{lem:Serre2} is contained in \cite[before Lemma 7.3]{MSRI-Letzter}. However, this discussion is significantly simpler than the proof of Lemma \ref{lem:Serre2} because Letzter excludes nonstandard quantum symmetric pairs and uses casework which is possible in the finite case due to the classification of admissible pairs in \cite[p.32/33]{a-Araki62}.
\section{Triangular decompositions}\label{sec:triang}
From now on until the end of Section \ref{sec:center} we fix parameters $\uc\in \cC$, $\us\in \cS$ and hence a corresponding quantum symmetric pair coideal subalgebra $B_{\uc,\us}$ of $\uqgp$.
By Corollary \ref{cor:Serre-lower} there exist relations between the generators $B_i$, $i\in I$, of $B_{\uc,\us}$ which are similar to the quantum Serre relations for $\uqgp$.
To make the lower order terms in these relations explicit, and to show that these relations together with those from Lemma \ref{lem:rel1} form a complete set of relations for $B_{\uc,\us}$, one uses triangular decompositions of $\uqgp$ involving $B_{\uc,\us}$. The triangular decompositions in this section will also be used in Section \ref{sec:specialize} to describe the behavior of $B_{\uc,\us}$  under specialization.
\subsection{A one-sided $U^+ U^0{}'$-module basis of $\uqgp$}\label{sec:BJbasis}
Recall from \eqref{eq:EJFJBJ} that for any multi-index $J=(j_1,\dots,j_n)\in I^n$ we defined $\wght(J)=\sum_{i=1}^n \alpha_{j_i}$ and $F_J=F_{j_1}\dots F_{j_n}$ and $B_J=B_{j_1}\dots B_{j_n}$. In this case we also define $|J|=n$. Let $\cJ$ be a fixed subset of $\bigcup_{n\in \N_0}I^n$ such that $\{F_J\,|\,J\in \cJ\}$ is a basis of $U^-$.  By the triangular decomposition \eqref{eq:triang-decomp} of $\uqgp$ the set $\{F_J\,|\,J\in \cJ\}$ forms a basis $\uqgp$ as a left $U^+U^0{'}$-module. By the following proposition this basis can be replaced by the set $\{B_J\,|\,J\in \cJ\}$.

Define a filtration $\cF^\ast$ of $U^-$ by $\cF^n(U^-)=\mathrm{span}\{F_J\,|\,J\in I^m, m\le n\}$ for all $n\in \N_0$. As the quantum Serre relations are homogeneous, the set $\{F_J\,|\,J\in \cJ, |J|\le n\}$ forms a basis of $\cF^n(U^-)$.
\begin{prop}\label{prop:leftBasisU}
  The set $\{B_J\,|\,J\in \cJ\}$ is a basis of the left (or right) $U^+U^0{}'$-module $\uqgp$.
\end{prop}
\begin{proof}
  We prove the result for the left $U^+ U^0{}'$-module $\uqgp$. Let $J\in \cJ$ and $|J|=n$. We first show by induction on $n$ that $F_J$ is contained in the left $U^+ U^0{'}$-module generated by the set $\{B_J\,|\,J\in \cJ\}$. Indeed, $F_J-B_J\in U^+ U^0{}' \cF^{n-1}(U^-)$. Using the quantum Serre relations for $U^-$ one hence obtains that $F_J-B_J$ is contained in the left $U^+ U^0{}'$-submodule of $\uqgp$ generated by the set $\{F_L\,|\,L\in\cJ, |L|\le n-1\}$. The induction hypothesis implies the desired result.
  
  It remains to show that the set $\{B_J\,|\,J\in \cJ\}$ is linearly independent over $U^+ U^0{}'$.  To this end assume that there exists a non-empty finite subset $\cJ'\subset \cJ$ such that
  \begin{align}\label{eq:assume=0}
     \sum_{J\in \cJ'} a_J B_J =0 
  \end{align}
  for some $a_J\in U^+ U^0{}'$. Let $m\in \N$ be maximal such that there exists $J\in \cJ'$ with $|J|=m$. In view of the specific form \eqref{eq:Bi-def} of the generators $B_i$, Equation \eqref{eq:assume=0} implies that 
  \begin{align}\label{eq:F=0}
     \sum_{J\in \cJ', |J|=m} a_J F_J =0. 
  \end{align}
  By the triangular decomposition \eqref{eq:triang-decomp} of $\uqgp$ and the fact that $\{F_J\,|\,J\in \cJ\}$ is linearly independent, Equation \eqref{eq:F=0} implies that $a_J=0$ for all $J\in \cJ'$ with $|J|=m$. Induction on $m$ gives the desired result. 
\end{proof}
Define a subspace of $B_{\uc,\us}$ by 
\begin{align}\label{eq:BcsJ}
  B_{\uc,\us,\cJ} = \oplus_{J\in \cJ} \field(q) B_J.
\end{align}  
Proposition \ref{prop:leftBasisU} can be reformulated by saying that the multiplication map
\begin{align}\label{eq:leftBasisU}
  U^+ \ot U^0{}' \ot  B_{\uc,\us,\cJ} \rightarrow \uqgp
\end{align}
is an isomorphism of vector spaces.

\subsection{A one-sided $\cM_X^+ U^0_\Theta{'}$-module basis of $B_{\uc,\us}$}
By Proposition \ref{prop:leftBasisU} any element in $B_{\uc,\us}$ can be written as a linear combination of elements in $\{B_J\,|\,J\in \cJ\}$ with coefficients in $U^+U^0{}'$. Actually, it is sufficient to allow coefficients from $\cM^+_X U^0_\Theta{}'$.
\begin{prop}\label{prop:leftBasis}
  The set $\{B_J\,|\, J\in \cJ\}$ is a basis of the left (or right) $\cM_X^+ U^0_\Theta{}'$-module $B_{\uc,\us}$.
\end{prop}
\begin{proof}
  We prove the result for the left $\cM_X^+ U^0_\Theta{}'$-module $B_{\uc,\us}$.
  Let $L\in I^n$. One can apply the quantum Serre relations \eqref{eq:q-Serre} repeatedly to write
  \begin{align*}
    F_L = \sum_{J\in \cJ, |J|=n } a_J F_J
  \end{align*}
  for some $a_J \in \field(q)$. Hence, using Corollary \ref{cor:Serre-lower} and the relations in Lemma \ref{lem:rel1}, one obtains
  \begin{align*}
    B_L - \sum_{J\in \cJ, |J|=n } a_J B_J \in \sum_{m<n} \sum_{J\in I^m}  \cM^+_X U^0_\Theta{}' B_J.
  \end{align*}
  One obtains $B_L \in \sum_{J\in \cJ} \cM_X^+ U^0_\Theta{}' B_J$ by induction on $n=\wght(L)$.
  This proves that the set $\{B_J\,|\,J\in \cJ\}$ spans $B_{\uc,\us}$ as a left $\cM^+_XU^0_\Theta{}'$-module.

  On the other hand, the set $\{B_J\,|\,J\in \cJ\}$ is linearly independent over $U^+U^0{'}$ by Proposition \ref{prop:leftBasisU}. Hence it is also linearly independent over $\cM^+_X U^0_\Theta{}'$. 
\end{proof}
Again, one can reformulate the above proposition by saying that the multiplication map
\begin{align*}
  \cM_X^+ \ot U^0_\Theta{}' \ot  B_{\uc,\us,\cJ} \rightarrow B_{\uc,\us}
\end{align*}
is an isomorphism of vector spaces.
\subsection{The quantum Iwasawa decomposition}\label{sec:Iwasawa}
Fix a subset $I^\ast$ of $I\setminus X$ containing exactly one element of each $\tau$-orbit of $I\setminus X$. In particular, $I^\ast$ contains all $i\in I\setminus X$ with $\tau(i)=i$ and precisely one of $j,\tau(j)$ if $\tau(j)\neq j$. Let $U'_\Theta$ denote the subalgebra of $U^0{}'$ generated by all elements in the set $\{K_i, K_i^{-1}\,|\,i\in I^\ast\}$. The multiplication map gives an isomorphism of algebras
\begin{align}\label{eq:U0-decomp}
  U'_\Theta \ot U^0_\Theta{}' \cong U^0{}'.
\end{align} 
Recall from Section \ref{sec:triang-decomp} that $V^+_X$ denotes the subalgebra of $U^+$ generated by the elements in $\ad(\cM_X)(E_i)$ for all $i\in I\setminus X$. By \eqref{Keb1} and \eqref{eq:U0-decomp} the multiplication map gives an isomorphism
\begin{align*}
  V_X^+\ot U'_\Theta \ot \cM^+_X \ot U^0_\Theta{}' \cong U^+ U^0{}'
\end{align*}
of vector spaces. By \eqref{eq:U0-decomp} and Proposition \ref{prop:leftBasis} the above decomposition of $U^+U^0{}'$ implies the following result.
\begin{prop}\label{prop:qIwasawa}
  The multiplication map gives an isomorphism of vector spaces
  \begin{align*}
    V_X^+\ot U'_\Theta\ot B_{\uc,\us} \cong \uqgp.
  \end{align*}
\end{prop}
\begin{rema}\label{rem:Iwa}
  The Iwasawa decomposition of $\gfrak'$ is given by $\gfrak'=\nfrak_X^+\oplus \afrak'\oplus \kfrak'$ where $\afrak'=\{h\in \hfrak'\,|\,\theta(h)=-h\}$ and $\nfrak_X^+$ denotes the $\gfrak_X^+$-module generated by $\{e_i\,|\,i\in I\setminus X\}$ under the adjoint action. In the above decomposition, $V_X^+$ is a $q$-analog of $U(\nfrak_X^+)$ and $B_{\uc,\us}$ is a $q$-analog of $U(\kfrak')$. The algebra $U'_\Theta$, however, is a somewhat coarse analog of $U(\afrak')$. A better $q$-analog of $U(\afrak')$ is given by $A_\Theta=\field(q)\langle K_\beta\,|\,\Theta(\beta)=-\beta\rangle$. However, relation \eqref{eq:U0-decomp} does not hold with $U'_\Theta$ replaced by $A_\Theta$. This problem can be circumvented by adjoining certain roots of $K_i^{\pm 1}$ to $U^0{}'$, as done for instance in \cite{a-Letzter99a}, but this is not necessary for our purposes.
\end{rema}
\subsection{References to Letzter's constructions}
For finite dimensional $\gfrak$, Proposition \ref{prop:leftBasis} appears in the proof of \cite[Theorem 7.4]{MSRI-Letzter}. The quantum Iwasawa decomposition is stated along the lines of Remark \ref{rem:Iwa} in \cite[Theorem 4.5]{a-Letzter99a}. Four variants of Proposition \ref{prop:qIwasawa} which interchange the order of the factors and exchange $V_X^+$ with $V_X^-$ appear in \cite[Theorem 2.2]{a-Letzter04}. In an earlier paper \cite[Theorem 2.4]{a-Letzter97} the quantum Iwasawa decomposition was established in the case $X=\emptyset$.
\section{Generators and relations}\label{sec:GensRels}
\subsection{Generators and relations for $B_{\uc,\us}$}
  The following theorem summarizes the results of Subsection \ref{sec:q-Serre} and is identical to \cite[Theorem 7.4]{MSRI-Letzter} if $\gfrak$ is finite dimensional.
\begin{thm}\label{thm:rels}
  Let $\Btil$ be the algebra freely generated over $\cM_X^+ U^0_\Theta{}'$ by elements $\Btil_i$ for all $i\in I$ and let $\Pi$ denote the canonical algebra homomorphism $\Pi:\Btil\rightarrow B_{\uc,\us}$ given by
  \begin{align}\label{eq:can-map}
   \Btil_i\mapsto B_i, \qquad m\mapsto m\qquad \mbox{for all $i\in I$, $m\in \cM^+ U^0_\Theta{}' $}.
  \end{align} 
  Write $\Btil_J=\Btil_{j_1}\cdots \Btil_{j_n}$ for any $J=(j_1,\dots,j_n)\in I^n$.
  Then there exist elements
  \begin{align*}
    \Ctil_{ij}(\uc)\in \sum_{\wght(J)< \lambda_{ij}} \cM^+_X U^0_\Theta{}' \Btil_J \qquad \mbox{for all $i,j\in I$, $i\neq j$}
  \end{align*}  
  such that the kernel $\ker(\Pi)$ is the ideal of $\Btil$ generated by the following elements:
  \begin{align}
     &K_\beta \Btil_i - q^{-(\beta,\alpha_i)} \Btil_i K_\beta && \mbox{for all $\beta\in Q^\Theta$, $i\in I$,} \label{eq:ker1}\\
     & E_i \Btil_j - \Btil_j E_i - \delta_{ij} \frac{K_i - K_i^{-1}}{q_i-q_i^{-1}} && \mbox{for all $i\in X$, $j\in I $,}\label{eq:ker2}\\
     &F_{ij}(\Btil_i,\Btil_j)-\Ctil_{ij}(\uc) &&\mbox{for all $i,j\in I$, $i\neq j$.} \label{eq:ker3}
  \end{align}
  Moreover, the formal expression of the elements $\Ctil_{ij}(\uc)$ in $\sum_{\wght(J)< \lambda_{ij}} \cM^+_X U^0_\Theta{}' \Btil_J$ is independent of the parameter $\us\in \cS$.
\end{thm}
\begin{proof}
  It follows from Corollary \ref{cor:Serre-lower} that there exist elements 
  \begin{align*}
     \Ctil_{ij}(\uc)\in \sum_{\wght(J)< \lambda_{ij}} \cM^+_X U^0_\Theta{}' \Btil_J
  \end{align*}
 such that $F_{ij}(\Btil_i,\Btil_j)-\Ctil_{ij}(\uc)$ lies in $\ker(\Pi)$ for any $i,j\in I$, $i\neq j$. Let $L$ be the ideal of $\Btil$ generated by the elements in \eqref{eq:ker1}, \eqref{eq:ker2}, and \eqref{eq:ker3} for this choice of $\Ctil_{ij}(\uc)$. It follows from Lemma \ref{lem:rel1} and Corollary \ref{cor:Serre-lower} that $L$ is contained in $\ker(\Pi)$. Hence there is a well defined surjective map $\Btil/L \rightarrow B_{\uc,\us}$ given by \eqref{eq:can-map}. As in the proof of Proposition \ref{prop:leftBasis} one shows that $\Btil/L$ is spanned as a left $\cM_X^+ U^0_\Theta{}'$-module by the elements $\Btil_J$ for $J\in \cJ$. Now Proposition \ref{prop:leftBasis} implies that the surjective map $\Btil/L\rightarrow B_{\uc,\us}$ is also injective.
 The formal independence of $\Ctil_{ij}(\uc)$ of $\us$ was already noted in Remark \ref{rem:s-indep}.
\end{proof} 
In the following we use the notation $C_{ij}(\uc)=\Pi(\Ctil_{ij}(\uc))$.
The results from Subsection \ref{sec:q-Serre} can be efficiently used to find explicit formulas for the elements $C_{ij}(\uc)$. This allows us to obtain a complete presentation of the algebras $B_{\uc,\us}$ in terms of generators and relations without the need to embark on explicit calculations. The method used in Theorem \ref{thm:relsBc} below was developed in \cite{a-Letzter03} for finite dimensional $\gfrak$, but it also works in the symmetrizable Kac-Moody case. To apply it we need the following first order calculation of $\kow(B_i)$. Recall Theorem \ref{thm:theta-props}.(3) and define
  \begin{align*}
    \cZ_i=-v_i \ad(Z^+_{\tau(i)})(K_{\tau(i)}^2)K_{\tau(i)}^{-1} K_i^{-1}
  \end{align*}
for any $i\in I\setminus X$ where we abbreviate $Z^+_{\tau(i)}=Z_{\tau(i)}^+(X)$.  
\begin{lem}\label{lem:fo}
  For any $i\in I\setminus X$ one has
  \begin{align}\label{eq:fo}
    \kow(B_i) =  B_i \ot K_i^{-1} + 1 \ot F_i + c_i \cZ_i \ot E_{\tau(i)}K_i^{-1} +\Upsilon
   \end{align}
   for some  $\Upsilon \in \cM_X U^0_\Theta{}'\ot \sum_{\gamma>\alpha_{\tau(i)}} U_\gamma^+ K_i^{-1}$.
\end{lem}
\begin{proof}
  We apply the formula 
  \begin{align}\label{eq:kowadxu}
    \kow(\ad(x)(u))=x_{(1)} u_{(1)} S(x_{(3)}) \ot \ad(x_{(2)})(u_{(2)})
  \end{align}
  to $x=Z^+_{\tau(i)}$ and $u=E_{\tau(i)}$. By \eqref{eq:E-copr} one obtains
  \begin{align*}
    \kow(\ad(Z^+_{\tau(i)}) (E_{\tau(i)})) = \ad(Z_{\tau(i)}^+)(E_{\tau(i)})\ot 1 + \ad(Z_{\tau(i)}^+)(K_{\tau(i)}^2)K_{\tau(i)}^{-1}\ot E_{\tau(i)} + \Psi
  \end{align*}
  for some $\Psi\in \cM_X^+ K_{\tau(i)}\ot \sum_{\gamma>\alpha_{\tau(i)}} U_\gamma^+$. By definition of $B_i$ and $\cZ_i$ one obtains the claim of the lemma with $\Upsilon = c_i v_i\Psi (K_i^{-1}\ot K_i^{-1})$.
\end{proof}
By Theorem \ref{thm:rels} one has
\begin{align}\label{eq:BZ-com}
  \cZ_i B_j =q^{(\alpha_i-\alpha_{\tau(i)},\alpha_j)}B_j \cZ_i \qquad \mbox{for all } i,j\in I\setminus X.
\end{align}

As before fix distinct $i,j\in I$ and write $Y=F_{ij}(B_i,B_j)$. By Proposition \ref{prop:Z=0} one has $Z=P_{-\lambda_{ij}}(Y)=0$. Using \eqref{eq:kowZ} as in the proof of Corollary \ref{cor:Serre-lower} one obtains
\begin{align}
  C_{ij}(\uc)&\stackrel{\phantom{\eqref{eq:P-comod-hom}}}{=}-(\id\ot \vep)(\kow(Z)-Y\ot K_{-\lambda_{ij}} ) \nonumber\\
             &\stackrel{\eqref{eq:P-comod-hom}}{=}-(\id \ot \vep)(\id \ot (P_{-\lambda_{ij}}\circ\pi_{0,0})) (\kow(Y)- Y\ot K_{-\lambda_{ij}}). \label{eq:Cijd}
\end{align}  
The expression $(\id \ot(P_{-\lambda_{ij}}\circ\pi_{0,0})) (\kow(Y)- Y\ot K_{-\lambda_{ij}})$ can be evaluated in general because many terms in $(\id \ot \pi_{0,0})\kow(Y)$ vanish. More explicitly, we use the fact that summands of $\kow(Y)$ need to contain the same number of $F_i$'s and $E_i$'s in the second tensor factor in order to contribute nontrivially. By Lemma \ref{lem:fo}, however, this can only occur if $i\in \{\tau(i),\tau(j)\}$. One obtains the following result.
\begin{thm}\label{thm:Cij=0if}
  For any $i,j\in I$ such that $i\notin \{j,\tau(i),\tau(j)\}$ one has $C_{ij}(\uc)=F_{ij}(B_i,B_j)=0$ in $B_{\uc,\us}$.
\end{thm}
Recall from Equation \eqref{eq:Serre3} that $C_{ij}(\uc)=0$ if $i\in X$. 
Hence we only need to consider the case that $i\in I\setminus X$.
In the following two subsections we distinguish the two cases $j\notin X$ and $j\in X$.
\subsection{Determining $C_{ij}(\uc)$ in the case $i,j\in I\setminus X$}
The next theorem gives explicit expressions for $C_{ij}(\uc)$ for $a_{ij}\in \{0, -1, -2\}$ if $j\in I\setminus X$. In this case, summands of $\kow(Y)$ involving the term $\Upsilon$ in Equation \eqref{eq:fo} will never survive under $\id \ot \pi_{0,0}$. Recall the definition of the $q$-number $[n]_q$ for $n\in \N$.
\begin{thm}\label{thm:relsBc}
  Assume that $i,j\in I\setminus X$. The elements $C_{ij}(\uc)=\Pi(\Ctil_{ij}(\uc))$ from Theorem \ref{thm:rels} are given by the following formulas.

\noindent {\bf Case 1:} $a_{ij}=0$.
  \begin{align}\label{eq:Cij0}
    C_{ij}(\uc)= \delta_{i,\tau(j)}(q_i-q_i^{-1})^{-1}\big(c_i \cZ_i - c_j \cZ_j \big).
  \end{align}
\noindent {\bf Case 2:} $a_{ij}=-1$.
  \begin{align}
    C_{ij}&(\uc)= \delta_{i,\tau(i)} q_i c_i \cZ_i B_j 
     - \delta_{i,\tau(j)}(q_i+q_i^{-1})\big( q_i c_j \cZ_j + q_i^{-2} c_i \cZ_i \big)B_i.\label{eq:Cij1}
  \end{align}  
\noindent {\bf Case 3:} $a_{ij}=-2$.
  \begin{align}\label{eq:Cij2}
    C_{ij}(\uc)=& \delta_{i,\tau(i)} q_i (q_i+q_i^{-1})^2 c_i \cZ_i (B_i B_j - B_j B_i) \\
     &+ \delta_{i,\tau(j)}[3]_{q_i} q_i(q_i^4 -1 )\big(q_i^{-8} c_i \cZ_i B_i^2-c_j \cZ_j B_i^2 \big). \nonumber 
  \end{align}        
\end{thm}
\begin{proof}
  To abbreviate notation we write $Q_{-\lambda_{ij}}=\id \ot (P_{-\lambda_{ij}}\circ \pi_{0,0})$ for any $i,j\in I$. 
  
  \noindent{\bf Case 1:} By \eqref{eq:Cijd} one has
  \begin{align}\label{eq:Cij0expl}
    C_{ij}(\uc)=-(\id \ot \vep)\circ Q_{-\lambda_{ij}} \left(\kow(B_i B_j-B_j B_i)-(B_iB_j - B_j B_i) \ot K_{-\lambda_{ij}}\right).
  \end{align}
  For $j\in I\setminus X$ one obtains
  \begin{align*}
    Q_{-\lambda_{ij}}(\kow(B_i B_j)- B_i B_j \ot K_{-\lambda_{ij}})
     &= Q_{-\lambda_{ij}}\big( c_j\cZ_j \ot F_i E_{\tau(j)}K_j^{-1}\big)\\
        &= Q_{-\lambda_{ij}}\left(-\delta_{i,\tau(j)} c_j\cZ_j \ot \frac{K_i-K_i^{-1}}{q_i-q_i^{-1}}K_j^{-1}\right)\\
        &= \delta_{i,\tau(j)} (q_i-q_i^{-1})^{-1} c_j \cZ_j  \ot K_{-\lambda_{ij}}.
  \end{align*}
  By \eqref{eq:Cij0expl} this implies the desired formula \eqref{eq:Cij0} for $C_{ij}(\uc)$.
  
   \noindent{\bf Case 2:} For $a_{ij}=-1$ one has $Y=B_i^2 B_j - (q_i+q_i^{-1})B_i B_j B_i + B_j B_i^2$. Performing a calculation analogous to Case 1, one obtains
   \begin{align}\label{eq:twoTerms}
     Q_{-\lambda_{ij}}(\kow(Y)- Y\ot K_{-\lambda_{ij}})=\delta_{i,\tau(i)} C_1\ot K_{-\lambda_{ij}} + \delta_{i,\tau(j)} C_2 \ot K_{-\lambda_{ij}}
   \end{align} 
   where
   \begin{align}
     C_1\ot K_{-\lambda_{ij}}=& Q_{-\lambda_{ij}}\bigg(  c_i \cZ_i B_j\ot F_i E_i K_i^{-1}K_j^{-1}\nonumber \\
          &  - (q_i+q_i^{-1})  B_j  c_i \cZ_i \ot  F_i K_j^{-1} E_i K_i^{-1}
     +  B_j  c_i \cZ_i \ot K_j^{-1} F_i E_i K_i^{-1} \bigg) \nonumber\\
     =& - q_i c_i \cZ_i B_j\ot K_{-\lambda_{ij}}\label{eq:C1}
   \end{align}
   and similarly, for the term $C_2$ where $\tau(i)=j$ one obtains
   \begin{align}
     C_2\ot K_{-\lambda_{ij}}&= Q_{-\lambda_{ij}}\bigg( B_i c_j \cZ_j \ot K_i^{-1} F_i E_i K_j^{-1} + B_i c_j \cZ_j \ot F_i K_i^{-1} E_i K_j^{-1} \nonumber \\
      - &(q_i+q_i^{-1}) c_j \cZ_j B_i \ot F_i E_i K_j^{-1} K_i^{-1} - (q_i+q_i^{-1}) B_i c_i \cZ_i \ot K_i^{-1} F_j E_j K_i^{-1} \nonumber\\
      + &c_i \cZ_i B_i \ot F_j E_j K_i ^{-2}  + B_i c_i \cZ_i \ot F_j K_i^{-1} E_j K_i^{-1}
      \bigg)\nonumber\\
      &= (q_i+q_i^{-1}) \big( q_i c_j \cZ_j B_i + q_i^{-2} c_i \cZ_i B_i\big) \ot K_{-\lambda_{ij}}.\label{eq:C2}
   \end{align}
 Equations \eqref{eq:C1} and \eqref{eq:C2} imply formula \eqref{eq:Cij1}.  

\noindent{\bf Case 3:} For $a_{ij}=-2$ one has $Y=B_i^3 B_j - [3]_{q_i} B_i^2 B_j B_i + [3]_{q_i} B_i B_j B_i^2 - B_j B_i^3$ where $[3]_{q_i}=q_i^{-2} + 1 + q_i^2$. Again relation \eqref{eq:twoTerms} holds with $\lambda_{ij}=3\alpha_i+\alpha_j$. If $\tau(i)=i$ then \eqref{eq:BZ-com} gives
\begin{align*}
   \cZ_i B_i = B_i \cZ_i, \qquad \cZ_i B_j = B_j \cZ_i.
\end{align*}
With these relations one determines
\begin{align}\label{eq:-2C1}
   C_1& =q_i(q_i + q_i^{-1})^2 c_i \cZ_i (B_j B_i - B_i B_j)
\end{align}
by a calculation analogous to Case 2. Similarly, if $\tau(i)=j$ one uses the relations
\begin{align}\label{eq:ZBcom}
  \cZ_i B_i = q_i^{4} B_i \cZ_i, \qquad \cZ_i B_j = q_i^{-4} B_j \cZ_i, \qquad \cZ_j B_i = q_i^{-4} B_i \cZ_j.
\end{align}
to obtain
  \begin{align}\label{eq:-2C2}
    C_2& = [3]_{q_i}q_i (1-q_i^4) B_i^2\big( c_i \cZ_i - q_i^{-8} c_j \cZ_j \big).
  \end{align} 
One obtains relation \eqref{eq:Cij2} by inserting \eqref{eq:-2C1} and \eqref{eq:-2C2} into Equation \eqref{eq:twoTerms}.  
\end{proof}
\begin{rema}
  The method of the above proof can in principle by used to determine the lower order terms $C_{ij}(\uc)$ for $a_{ij}=-3, -4, \dots$. This becomes very tedious. In the case where $\gfrak$ is of finite type $G_2$ and $a_{ij}=-3$ the term $C_{ij}(\uc)$ was calculated in \cite[Theorem 7.1]{a-Letzter03}. For $\uc=(q_1^{-2},q_2^{-2})$ it is also given in \cite[Proposition 3.1 ]{a-KolPel11} in the precise conventions of the present paper. 
\end{rema}
\begin{eg}\label{eg:qOnsager}
  Consider the affine Lie algebra $\hat{\slfrak}_2(\field)$ with generalized Cartan matrix $A=\left(\begin{matrix} 2 & -2 \\ -2 & 2 \end{matrix}\right)$ and $I=\{0,1\}$. Choose the admissible pair $(X,\tau)=(\emptyset, \mathrm{id})$. In this case $\cC=(\field(q)^\times)^2$ and $\cS=\field(q)^2$. The quantum symmetric pair coideal subalgebra $B_{\uc,\us}$ corresponding to $\uc=(c_0,c_1)\in \cC$ and $\us=(s_0,s_1)\in \cS$ is generated by two elements
  \begin{align*}
    B_i=F_i- c_i E_i K_i^{-1} + s_i K_i^{-1} \qquad \mbox{for $i=0,1$.}
  \end{align*}   
One has $\cZ_0=\cZ_1=-1$. By Theorem \ref{thm:relsBc} the algebra $B_{\uc,\us}$ is given by the defining relations
\begin{align}
  B_0^3 B_1 {-} [3]_q B_0^2 B_1 B_0 {+} [3]_q B_0 B_1 B_0^2 {-} B_1 B_0^3 &= q(q{+}q^{-1})^2 c_0 (B_1 B_0 {-} B_0 B_1),\label{eq:Ons1}\\
   B_1^3 B_0 {-} [3]_q B_1^2 B_0 B_1 {+} [3]_q B_1 B_0 B_1^2 {-} B_0 B_1^3 &= q(q{+}q^{-1})^2 c_1 (B_0 B_1 {-} B_1 B_0).\label{eq:Ons2}
\end{align}   
This algebra is the $q$-Onsager algebra discussed in the introduction.
\end{eg}
\subsection{Determining $C_{ij}(\uc)$ in the case $i\in I\setminus X$ and $j\in X$}
In this case Lemma \ref{lem:fo} implies that $C_{ij}(\uc)=0$ unless $i=\tau(i)$. Moreover, the expansion \eqref{eq:fo} needs to be extended to include terms of weight $\alpha_i+\alpha_j$ in the second tensor factor.
\begin{lem} \label{lem:foX}
  Assume $i\in I\setminus X$, $j\in X$, and $\tau(i)=i$. Then there exists $\cW_{ij}\in \cM^+_X$ which is independent of $\uc$, such that
  \begin{align*}
     \kow(B_i) =  B_i \ot K_i^{-1} + 1 \ot F_i + c_i \cZ_i \ot E_{i}K_i^{-1} + c_i\cW_{ij}K_j\ot \ad(E_j)(E_i)K_i^{-1}+\Upsilon
   \end{align*}
   for some  $\Upsilon \in \cM_X U^0_\Theta{}'\ot \sum_{\abv{\gamma>\alpha_{i}}{\gamma \neq \alpha_i+\alpha_j}} U_\gamma^+ K_i^{-1}$.
\end{lem}
\begin{proof}
  For $\beta\in Q$ let $\pi_\beta:\uqg \rightarrow \uqg_\beta$ denote the projection onto the corresponding weight space. Relation \eqref{eq:kowadxu} implies that
  \begin{align*}
    (\id\ot \pi_{\alpha_i+\alpha_j})\kow(\ad(Z_i)(E_i))\in U^+ K_j K_i \ot \ad(E_j)(E_i)
  \end{align*}
  With this observation the claim follows in the same way as Lemma \ref{lem:fo}.
\end{proof}
\begin{thm}\label{thm:relsBc2}
  Assume that $i\in I\setminus X$ and $j\in X$. The elements $C_{ij}(\uc)=\Pi(\Ctil_{ij}(\uc))$ from Theorem \ref{thm:rels} are given by the following formulas.

\noindent {\bf Case 1:} $a_{ij}=0$.
  \begin{align}\label{eq:Cij0X}
    C_{ij}(\uc)= 0.
  \end{align}
\noindent {\bf Case 2:} $a_{ij}=-1$.
  \begin{align}
    C_{ij}&(\uc)= \delta_{i,\tau(i)}c_i \Big[ \frac{1}{q_i-q_i^{-1}}\big(q_i^2 B_j \cZ_i - \cZ_i B_j\big)
     +\frac{q_i+q_i^{-1}}{q_j-q_j^{-1}}\cW_{ij} K_j\Big]. \label{eq:Cij1X}
  \end{align}  
\noindent {\bf Case 3:} $a_{ij}=-2$.
  \begin{align}\label{eq:Cij2X}
    C_{ij}(\uc)=\delta_{i,\tau(i)} \frac{c_i}{q_i-q_i^{-1}}\bigg[& q_i^2 \Big([3]_{q_i} B_i B_j - (q_i^2 + 2)B_j B_i\Big)\cZ_i \\
     &\qquad-\cZ_i\Big((2+q_i^{-2})B_iB_j - [3]_{q_i} B_j B_i \Big)\bigg]\nonumber\\ 
     -&\delta_{i,\tau(i)}\frac{q_i-q_i^{-1}}{q_j-q_j^{-1}}(q_i+q_i^{-1})^2 [3]_{q_i} B_i c_i \cW_{ij} K_j.\nonumber
  \end{align}        
\end{thm}
\begin{proof}
If $a_{ij}=0$ then $Y=B_iB_j-B_jB_i$ and hence $Y=0$ if $j\in X$ by \eqref{eq:Serre3}.
Cases 2 and 3 follow from Lemma \ref{lem:foX} by direct computation in a similar way as the corresponding cases in Theorem \ref{thm:relsBc} follow from Lemma \ref{lem:fo}. One has to keep in mind, however, that in the present case $\cZ_i$ does in general not $q$-commute with $B_j$. As $i=\tau(i)$ the elements $\cZ_i$ and $B_i$ still commute.
\end{proof}
\begin{eg}\label{eg:sp4}
  Consider $\gfrak=\slfrak_4(\field)$ with $X=\{1,3\}$ and $\tau=\mathrm{id}$. In this case $I_{ns}=\emptyset$ and hence only the standard quantum symmetric pair coideal subalgebra $B_{\uc}$ exists. It is generated by $\cM_X=\field(q)\langle E_1, F_1, K_1^{\pm 1}, E_3, F_3, K_3^{\pm 3}\rangle$ and the element 
  \begin{align*}
    B_2=F_2- c_2 T_1 T_3(E_2)K_2^{-1}= F_2 - c_2 \ad(E_3 E_1)(E_2)K_2^{-1}. 
  \end{align*}
To determine $C_{21}(\uc)$ one calculates
  \begin{align*}
    \kow(B_2)=&B_2\ot K_2^{-1} + 1\ot F_2 - c_2(1-q^{-2})^2 E_1 E_3 \ot E_2 K_2^{-1}\\
         &- c_2(1{-}q^{-2}) E_3 K_1 \ot \ad(E_1) (E_2)K_2^{-1} - c_2(1{-}q^{-2}) E_1 K_3 \ot \ad(E_3)(E_2)K_2^{-1}\\
         &- c_2 K_1 K_3 \ot \ad(E_1 E_3)(E_2)K_2^{-1}
  \end{align*}  
which implies 
\begin{align*}
  \cZ_2=-(1-q^{-2})^2 E_1 E_3, \qquad \cW_{21}=-(1-q^{-2})E_3.
\end{align*} 
By \eqref{eq:Cij1X} one obtains
\begin{align*}
  C_{21}(\uc)= -q^{-1}(q-q^{-1})^2 c_2 F_1 E_1 E_3 -q^{-2} c_2 K_1^{-1} E_3 - c_2 K_1 E_3.
\end{align*} 
\end{eg}
\subsection{References to Letzter's constructions}
  For finite dimensional $\gfrak$, the above method to determine the coefficients $C_{ij}(\uc)$ was devised in \cite[Section 7]{a-Letzter03} and the result corresponding to Theorem \ref{thm:relsBc} is stated in \cite[Theorem 7.1]{a-Letzter03}. By reference to the finite list of cases in \cite[pp. 32/33]{a-Araki62}, Letzter's formulas become simpler than those given here, even for $a_{ij}=-1,-2$.  The necessity to include terms of weight $\alpha_i+\alpha_j$ in the second tensor factor of $\kow(B_i)$ in the calculation of $C_{ij}(\uc)$ for $j\in X$ seems to have been overlooked in \cite{a-Letzter03}. As Example \ref{eg:sp4} above shows, the formulas in \cite[Theorem 7.1(iv)]{a-Letzter03} are only valid for $j\notin X$.   
\section{The center of quantum symmetric pair coideal subalgebras}\label{sec:center}
The center $Z(B_{\uc,\us})$ of the algebra  $B_{\uc,\us}$ can be described in full generality. There are two distinct cases. Recall that the Cartan matrix $A$ is always assumed to be indecomposable. If $\gfrak=\gfrak(A)$ is infinite dimensional then the center $Z(B_{\uc,\us})$ is trivial, as will be shown in Subsection \ref{sec:Z(B)notfinite}. If $\gfrak$ is finite dimensional then the center of a slight extension of $B_{\uc,\us}$ was determined in \cite{a-KL08}. For completeness, and as the  conventions of \cite{a-KL08} differ from those of the present paper, we summarize the structure of the center for the finite case in the brief Subsection \ref{sec:Z(B)finite}.
In both cases, the description of the integrable part of $\uqgp$ due to \cite{a-JoLet1}, \cite{a-JoLet2} provides an essential ingredient in the proof.
\subsection{Integrability of $B_{\uc,\us}$-invariant elements}
  For any left $\uqgp$-module $M$ define its integrable part $I(M)$ by
  \begin{align*}
    I(M)=\{m\in M\,|\, \forall i\in I \, \exists k\in \N\, \mbox{ such that } E_i^km=0=F_i^km\}.
  \end{align*}
  Observe that $I(M)$ is a $\uqgp$-submodule of $M$. We will be in particular interested in the integrable part of $\uqgp$ considered as a module over itself via the left adjoint action. In this case one has by
  \cite[6.4]{a-JoLet1}, \cite[7.1.6]{b-Joseph} the relation
  \begin{align}\label{eq:I(U)}
    I(\uqgp)=\bigoplus_{\lambda\in -R^+} \ad(\uqgp)(K_{-\lambda})
  \end{align}
  where $R^+=Q\cap 2 P^+$ with $P^+=\{\lambda\in P\,|\, \lambda(h_i)\ge 0 \,\forall i\in I\}$. Recall that a $\uqgp$-module $M$ is called a weight module if there exists a direct sum decomposition 
  \begin{align*}
    M=\bigoplus_{\mu\in P} M_\mu
  \end{align*}
  such that for all $m\in M_\mu$ one has 
  \begin{align*}
    K_i m&= q^{(\alpha_i,\mu)}m, & E_i m &\in M_{\mu+\alpha_i}, & F_i m &\in M_{\mu-\alpha_i}.
  \end{align*}
  Finally, for any $\uqgp$-module $M$ define the subspace of $B_{\uc,\us}$-invariant elements by
  \begin{align*}
    M^{B_{\uc,\us}}=\{m\in M\,|\, b m = \vep(b)m \quad \forall b\in B_{\uc,\us}\}. 
  \end{align*}
  To obtain the following result we adapt the proof of \cite[Lemma 4.4]{a-Letzter97} to the present setting.
  \begin{prop}\label{prop:inv-int}
    Let $M$ be a $\uqgp$-weight module. 
    \begin{enumerate}
      \item If $M^{B_{\uc,\us}}\neq \{0\}$ then there exists $v\in M$ such that $F_iv=0$ for all $i\in I$.
      \item $M^{B_{\uc,\us}}\subseteq I(M)$.
    \end{enumerate}   
  \end{prop}
  \begin{proof}
    Assume $m\in M^{B_{\uc,\us}}$ and write $m=\sum_{\mu\in P} m_\mu$ with $m_\mu\in M_\mu$. Choose $\mu\in P$ minimal such that $m_\mu\neq 0$. Then $B_i m=\vep(B_i) m$ implies that $F_i m_\mu=0$ for all $i\in I$. This proves (1).
    
     Next we show that for all $\mu\in P$ there exists $k\in \N$ such that $F_i^k m_\mu=0$. Indeed, otherwise choose $\lambda\in P$ minimal such that $F_i^k m_\lambda\neq 0$ for all $k\in \N$. By minimality, for all $\mu<\lambda$ there exists $k_\mu\in \N$ such that $F_i^{k_\mu} m_\mu=0$. Hence $B_i^n m_\mu \in \bigoplus_{\nu > \mu -k_\mu \alpha_i} M_\nu$ if $\mu<\lambda$. For large $n$ this implies that $B_i^n m $ has $M_{\lambda-n\alpha_i}$-component $F_i^n m_\lambda\neq 0$. This is a contradiction to $B_i^n m = \vep(B_i)^n m$.
    
In the same way one shows that for all $\mu\in P$ there exists $k\in \N$ such that $E_i^k m_\mu=0$. To this end one has to replace the generators $B_i$ in the above argument by elements $C_i= E_i + \overline{C}_i$ with $\overline{C}_i\in U^-U^0{}'$. Such elements exist in $B_{\uc,\us}$ as can be seen via the adjoint action of $\cM_X^-$ on $B_iK_i$ for $i\in I\setminus X$. This completes the proof of (2).  
  \end{proof}
\subsection{The center of $B_{\uc,\us}$ for $A$ of infinite type}\label{sec:Z(B)notfinite}
  We now want to determine the center $Z(B_{\uc,\us})$ of the algebra $B_{\uc,\us}$. Observe first that
  \begin{align*}
    Z(B_{\uc,\us}) = \{z\in B_{\uc,\us}\,|\,\ad(b)(z) = \vep(b)z \,\forall b\in B_{\uc,\us}\},
  \end{align*}
  see also \cite[1.3.3]{b-Joseph}, \cite[Theorem 1.2]{a-Letzter-memoirs}, or \cite[Lemma 4.11]{a-KolbStok09}. By Proposition \ref{prop:inv-int}.(2) one obtains
   \begin{align}\label{eq:ZinI}
    Z(B_{\uc,\us})\subset I(\uqgp)^{B_{\uc,\us}}.
  \end{align}
In view of Equation \eqref{eq:I(U)}, to determine $Z(B_{\uc,\us})$ one hence has to investigate the subspaces  $\left(\ad(\uqgp)(K_{-\lambda})\right)^{B_{\uc,\us}}$ of $\uqgp$ for $\lambda\in R^+$.
\begin{lem}\label{lem:adUK=0}
  Assume that $\gfrak=\gfrak(A)$ is infinite dimensional. If $\lambda\in R^+$ satisfies $\lambda(h_i)\neq 0$ for some $i\in I$ then 
  \begin{align*}
    \left(\ad(\uqgp)(K_{-\lambda})\right)^{B_{\uc,\us}}=\{0\}.
  \end{align*} 
\end{lem}  
\begin{proof}
  Let $V(\lambda/2)$ denote the integrable left $\uqgp$-module of highest weight $\lambda/2$. Its dual space $V(\lambda/2)^\ast$ is a lowest weight $\uqgp$-module of lowest weight $-\lambda/2$ with left action given by
  \begin{align*}
    (u f)(v) = f(S(u)v) \qquad \mbox{ for all } f\in V(\lambda)^\ast, v\in V(\lambda), u\in \uqgp.
  \end{align*}  
  By \cite[Proof of Lemma 7.1.15 (iii)]{b-Joseph} one has an isomorphism of left $\uqgp$-modules
  \begin{align}\label{eq:VV}
    \ad(\uqgp)(K_{-\lambda}) \cong V(\lambda/2) \ot V(\lambda/2)^\ast.
  \end{align}
  Moreover, by \cite[7.1.15 (ii)]{b-Joseph} the simple highest weight $\uqgp$-module $V(\lambda/2)$ is infinite dimensional and therefore does not contain a nonzero element annihilated by all $F_i$ for $i\in I$.  Hence $V(\lambda)\ot V(\lambda/2)^\ast$ does not contain a nonzero element annihilated by all $F_i$ for $i\in I$. By Proposition \ref{prop:inv-int}.(1) one obtains  
  \begin{align*}
    \left(V(\lambda/2) \ot V(\lambda/2)^\ast\right)^{B_{\uc,\us}}=\{0\}
  \end{align*}
which by \eqref{eq:VV} completes the proof of the lemma. 
\end{proof}
With the above lemma one can show that the center of $B_{\uc,\us}$ is trivial unless $A$ is of finite type. 
\begin{thm}\label{thm:Z(B)infinite}
   Assume that $\gfrak$ is infinite dimensional. Then $Z(B_{\uc,\us})=\field(q)1$.
\end{thm}
\begin{proof}
  Assume that $z\in Z(B_{\uc,\us})$. By relations \eqref{eq:I(U)} and  \eqref{eq:ZinI} one can write
  $z=\sum_{\lambda\in R^+} z_\lambda$ with 
  \begin{align*}
    z_\lambda\in \left(\ad(\uqgp)(K_{-\lambda})\right)^{B_{\uc,\us}}.
  \end{align*}
  By Lemma \ref{lem:adUK=0} the relation $z_\lambda\neq 0$ implies that $\lambda(h_i)=0$ for all $i\in I$. But in this case $\ad(\uqgp)(K_{-\lambda})=\field(q)K_{-\lambda}$ and hence $z$ is a linear combination of the elements $K_{-\lambda}$ for $\lambda\in R^+$ with $\lambda(h_i)=0$ for all $i\in I$. Applying the coproduct to $z$ one may assume that $z=K_{-\lambda}$ for some $\lambda\in R^+$ with $\lambda(h_i)=0$ for all $i\in I$. For each such $\lambda$, however, one has $\Theta(\lambda)=-\tau(\lambda)$ and hence $K_\lambda\in U^0_\Theta{}'$ implies that $\lambda=0$.  
\end{proof}
\subsection{The center of $B_{\uc,\us}$ for $A$ of finite type}\label{sec:Z(B)finite}
If $\gfrak$ is of finite dimensional then it is convenient to augment $U^0{}'$ to the group algebra of the weight lattice and to work over the field $\field(q^{1/2})$. The resulting extension $\uqgh$ of $\uqgp$ hence has generators $E_i$, $F_i$ and $K_\lambda$ for all $i\in I$ and $\lambda\in P$. It is sometimes called the simply connected quantized enveloping algebra of $\gfrak$. Relation \eqref{eq:I(U)} simplifies to
\begin{align*}
  I(\uqgh)=\bigoplus_{\lambda\in P^+} \ad(\uqg)(K_{-2\lambda})
\end{align*} 
Similarly, one extends $B_{\uc,\us}$ to the algebra $\cBcs$ generated by 
\begin{align*}
  \check{U}^0_\Theta=\field(q^{1/2})\langle K_\lambda\,|\,\lambda \in P, \Theta(\lambda)=\lambda\rangle
\end{align*}
and by $\cM_X$ and the elements $B_i$ defined by \eqref{eq:Bi-def} for all $i\in I\setminus X$ as before.
Define a subset $P^+_\Theta$ of the set of dominant integral weights by
\begin{align*}
  P^+_\Theta = \{\lambda\in P^+ \,|\, \Theta(\lambda) = \lambda + w^0\lambda -w_X\lambda\}
\end{align*}
where $w^0$ denotes the longest element of the Weyl group $W$. The following theorem summarizes the main results of \cite{a-KL08} in the conventions of the present paper.
\begin{thm}
  Assume that $\gfrak$ is finite dimensional.
  \begin{enumerate}
    \item $\displaystyle Z(\cBcs) = \bigoplus_{\lambda\in P^+_\Theta} Z(\cBcs) \cap \ad(\uqg)(K_{-2\lambda})$.
    \item $\displaystyle \dim \big(Z(\cBcs) \cap \ad(\uqg)(K_{-2\lambda}) \big) = 
            \begin{cases}
              1 & \mbox{if $\lambda \in P^+_\Theta$},\\
              0 & \mbox{else}.
            \end{cases}$
   \item The center $Z(\cBcs)$ is a polynomial ring in $\rank(\kfrak)$ variables.         
  \end{enumerate}
\end{thm}
The projection $P_{-2\lambda}$ from Section \ref{sec:DecProj} satisfies $P_{-2\lambda}(I(\uqgh))=\ad(\uqg)(K_{-2\lambda})$. By Lemma \ref{lem:PlamB} this proves (1), see also \cite[footnote p. 318]{a-KL08}. The proof of Property (2) is more involved and takes up most of \cite{a-KL08}. Property (3) follows from (1) and (2) and an analysis of the set $P^+_\Theta$, see \cite[Section 9]{a-KL08}.
\section{Equivalence of coideal subalgebras}\label{sec:equivalence}
In this section we investigate equivalence of quantum symmetric pair coideal subalgebras in the sense of the following definition.
\begin{defi}\label{def:equivalence}
  Let $C$ and $D$ be right coideal subalgebras of a Hopf algebra $H$. We say that $C$ is equivalent to $D$ if there exists a Hopf algebra automorphism $\varphi$ of $H$ such that $\varphi(C)=D$. In this case we write  $C\sim D$.
\end{defi}
To obtain more automorphisms, during this section, we work with $\uqgp$ defined over the algebraic closure $\oKq$. Accordingly, we extend the set of parameters and define
\begin{align*}
  \ocC&=\{\uc\in (\overline{\field(q)}^\times)^{I\setminus X}\,|\, \mbox{$c_i=c_{\tau(i)}$ if
    $\tau(i)\neq i$ and $(\alpha_i,\Theta(\alpha_i))=0$} \} \\
  \ocS&=\{\us\in \overline{\field(q)}^{I\setminus X}\,|\,s_i\neq 0 \Rightarrow \, (i\in  I_{ns} \mbox{ and } a_{ij}\in -2\N_0 \forall j\in I_{ns}\setminus \{i\})\}. 
\end{align*}
Recall from Subsection \ref{sec:Iwasawa} that we have fixed a subset $I^*\subset I\setminus X$ containing exactly one element of each $\tau$-orbit in $I\setminus X$.  As we will see, any $B_{\uc,\us}$ is equivalent to a $B_{\ud,\us'}$ for some $\us'\in \ocS$ and ${\ud}$ in the subset 
\begin{align}\label{eq:D-def}
  \ocD=\{\ud\in (\oKq^\times)^{I\setminus X}\,|\, \mbox{$d_i=1$ if
    $\big(\tau(i)=i$ or $(\alpha_i,\Theta(\alpha_i))=0$ or
    $i\notin I^\ast\big)$} \}
\end{align}
of $(\oKq^\times)^{I\setminus X}$.
\subsection{Equivalence under the action of $\Htil$}
Let $B$ and $B'$ be coideal subalgebras of $\uqgp$. Refining the notation introduced in Definition \ref{def:equivalence} we write $B\simHt B'$ if there exists $x\in \Htil$ such that $\Ad(x)(B)=B'$.
Moreover, define an equivalence relation on $\ocS$ by
\begin{align*}
  \us \simS \us' \, \Longleftrightarrow \, s_i= \pm s_i' \quad\mbox{for all } i\in I_{ns}.
\end{align*}
if $\us=(s_i)_{i\in I\setminus X}$ and $\us'=(s_i')_{i\in I\setminus X}$.
For the following result the fact that $\uqgp$ is defined over $\oKq$ is necessary because one needs to take square roots. 
\begin{prop}\label{prop:Htil}
  (1) Let $\uc\in\ocC$ and $\us \in \ocS$. Then there exist $\ud\in \ocD$ and $\us'\in \ocS$ such that
  $B_{\uc,\us}\simHt B_{\ud,\us'}$.\\
  (2) Let $\ud, \ud'\in \ocD$ and $\us,\us'\in \ocS$. Then $B_{\ud,\us}\simHt B_{\ud',\us'}$ if and only if
  $\ud=\ud'$ and $\us\simS \us'$. 
\end{prop}
\begin{proof}
  (1) For all $i\in I\setminus X$ such that $\tau(i)=i$ choose $d_i\in
  \oKq$ such that $d_i^2=c_i$. Define $x\in \Htil$ by $x(\alpha_i)=1$ if $\tau(i)\neq i$ or $i\in X$ and
  $x(\alpha_i)=d_i^{-1}$ else. Then 
  \begin{align*}
    \Ad(x)(B_i)&=\Ad(x)(F_i + c_i\,\theta_q(F_i K_i)K_i^{-1}+ s_i K_i^{-1})\\
               &=d_i(F_i + \theta_q(F_i K_i)K_i^{-1}+ \frac{s_i}{d_i}K_i^{-1})
  \end{align*}
  for all $i$ with $\tau(i)=i$, and $\Ad(x)$ leaves $\cM_X$,
  $U^0_\Theta{}'$, and $B_j$ with $\tau(j)\neq j$ 
  invariant. Hence $B_{\uc,\us}\simHt B_{\uc',\us'}$ where
  $\uc'\in\ocC $ and $\us'\in \ocS$ are given by
  \begin{align*}
    c_i'=\begin{cases}
           1 & \mbox{if $\tau(i)=i$,}\\
           c_i& \mbox{else,}
         \end{cases} \qquad \mbox{ and } \qquad s_i'=\frac{s_i}{d_i} \quad \forall \, i\in I_{ns}.
  \end{align*}
  Assume now that $\tau(j)\neq j$ and $(\alpha_j,\Theta(\alpha_j))=0$. By Lemma \ref{lem:cond2} we have
  $(\alpha_j,\alpha_{\tau(j)})=0$ and $\Theta(\alpha_j)=-\alpha_{\tau(j)}$ and therefore $\alpha_j(h_l)=0=\alpha_{j}(\rho^\vee_X)$ for all $l\in X$. Hence $B_j=F_j - c_j\,E_{\tau(j)}K_j^{-1}$ and $B_{\tau(j)}=F_{\tau(j)} - c_{\tau(j)}\,E_{j} K_{\tau(j)}^{-1}$. 
  As $\uc\in \ocC$ one obtains $c_j=c_{\tau(j)}$. Define $x\in \Htil$ by $x(\alpha_j)=c_j^{-1}$ and $x(\alpha_i)=1$ if $i\neq j$. Then $\Ad(x)(B_j)=c_j(F_j - E_{\tau(j)} K_j^{-1})$, $\Ad(x)(B_{\tau(j)})=F_{\tau(j)}- E_j K_{\tau(j)}^{-1}$ and $\Ad(x)(B_i)=B_i$ if $i\notin\{j,\tau(j)\}$.

  (2) Assume that $\us, \us'\in \cS$ satisfy $\us\sim \us'$. Define $x\in \Htil$ by
  \begin{align*}
    x(\alpha_i)=\begin{cases}
                   1 & \mbox{if $i\notin I_{ns}$ or ($i\in I_{ns}$ and $s_i=s_i'$),}\\
                  -1 & \mbox{else.}
                \end{cases}
  \end{align*}
  One sees immediately that $\Ad(x)(B_{\ud,\us})=B_{\ud,\us'}$ holds for all $\ud \in \ocD$.
  
  For the converse implication we make use of the following fact which is a consequence of Proposition \ref{prop:leftBasisU}:
  \begin{itemize}
    \item[($\ast$)]Let $\ud=(d_i)_{i\in I\setminus X} \in \ocD$, $\us=(s_i)_{i\in I_{ns}}\in \cS$, and $c_i\in \oKq^\times$, $t_i\in \oKq$ for some $i\in I^\ast$. If $F_i + c_i \theta_q(F_i K_i)K_i^{-1} + t_i K_i^{-1}\in B_{\ud,\us}$ then  $c_i=d_i$ and $t_i=s_i$.
  \end{itemize}
  Now write $\ud=(d_i)_{i\in I\setminus X}$, $\us=(s_i)_{i\in I\setminus X}$ and $\ud'=(d_i')_{i\in I\setminus X}$, $\us'=(s_i')_{i\in I\setminus X}$ and assume that $B_{\ud,\us}=\Ad(x)(B_{\ud',\us'})$ for some $x\in \Htil$. We denote the generators of $B_{\ud,\us}$ and $B_{\ud',\us'}$ by $B_i$ and $B_i'$, respectively.
Assume first that $i\in I_{ns}$. Then $d_i=d_i'=1$ and $\theta_q(F_iK_i)=-E_i$  and hence
\begin{align*}
  \Ad(x)(B_i')= x(-\alpha_i) \left( F_i - x(2\alpha_i) E_i K_i^{-1} + x(\alpha_i) s_i'\right).
\end{align*}
Property ($\ast$) implies that $x(\alpha_i)^2=1$ and $x(\alpha_i) s_i' = s_i$. Hence $s_i=\pm s_i'$.

 Next assume that $i\in I^\ast$ with $\tau(i)\neq i$. By definition of $\ocD$ we have in particular $d_{\tau(i)}=1=d'_{\tau(i)}$. It suffices to show that $B_i=B_i'$. To this end one calculates
  \begin{align*}
    \Ad(x)(B_i')&=x(-\alpha_i) F_i + x(-\Theta(\alpha_i)) d_i' \theta_q(F_i K_i) K_i^{-1}\\
                &= x(-\alpha_i) [ F_i + x(\alpha_i + w_X \alpha_{\tau(i)})  d_i' \theta_q(F_i K_i) K_i^{-1} ].
  \end{align*}
By ($\ast$) the above relation implies that 
\begin{align}\label{eq:dd'1}
  d_i = x(\alpha_i + w_X \alpha_{\tau(i)}) d_i'.
\end{align}  
Analogously one obtains with $d_{\tau(i)}=1=d'_{\tau(i)}$ the relation
  \begin{align*}
    \Ad(x)(B_{\tau(i)}')&= x(-\alpha_{\tau(i)}) [ F_{\tau(i)} + x(\alpha_{\tau(i)} + w_X \alpha_{i}) \theta_q(F_i K_i) K_i^{-1} ]
  \end{align*}
and hence $x(\alpha_{\tau(i)} + w_X \alpha_{i})=1$. The relation $\alpha_i-\alpha_{\tau(i)}\in Q^\Theta$ implies that  $\alpha_i + w_X \alpha_{\tau(i)}=\alpha_{\tau(i)} + w_X \alpha_{i}$. The desired relation $d_i=d_i'$ now follows from \eqref{eq:dd'1}.  
\end{proof}
\begin{rema}\label{rema:D=1}
  The definition of $\ocD$ suggests the investigation of the set
  \begin{align*}
    I_{\cD}= \{i\in I^\ast\,|\, \tau(i)\neq i, (\alpha_i,\Theta(\alpha_i))\neq 0\}.
  \end{align*}
  For simple $\gfrak$ of finite type it was observed in \cite[Section 7, Variation 1]{MSRI-Letzter} that $I_\cD$ is empty unless the restricted root system corresponding to the involution $\theta$ is of nonreduced type $BC$. In the latter case $I_\cD$ contains at most one element.
For $\gfrak=\gfrak(A)$ of affine type with indecomposable $A$ the set $I_\cD$ contains at most two elements. This can be seen by direct inspection of the affine Dynkin diagrams in \cite[pp. 54, 55]{b-Kac1}.
  As an example with $|I_\cD|=2$ consider $\gfrak=\widehat{\slfrak}_4(\C)$ with $X=\emptyset$ and $\tau=(01)(23)$. 
\end{rema}
\subsection{Equivalence under the action of $\Aut_{\mathrm{Hopf}}(\uqgp)$}
It is possible that $B_{\ud}\sim B_{\ud'}$ for $\ud,\ud'\in \ocD$ with $\ud\neq \ud'$, even for finite dimensional, simple $\gfrak$. Let $\tau'\in \Aut(A)$ be a diagram automorphism which commutes with $\tau$ and which leaves $X$ invariant. By Proposition \ref{prop:Htil}.(1) one has $\tau'(B_{\ud}) \simHt B_{\ud'}$ for some $\ud'\in \ocD$. As the following example shows it is possible that $\ud'\neq \ud$.
\begin{eg}\label{eg:sl3iso}
  For $\gfrak=\slfrak_3(\C)$, $X=\emptyset$, and $\tau=(12)$ with $I^*=\{1\}$ the standard quantum symmetric pair coideal subalgebra $B_{\ud}$ is generated by the elements
  \begin{align*}
    F_1 - d E_2 K_1^{-1},\quad  F_2 - E_1 K_2^{-1}
  \end{align*}
  if $\ud=(d,1)$. Hence $\tau(B_{\ud})$ is generated by the elements $F_1 - E_2 K_1^{-1}$ and $F_2 - d E_1 K_2^{-1}$. Now apply $\Ad(x)$ where $x(\alpha_1)=d^{-1}$ and $x(\alpha_2)=1$. The subalgebra $\Ad(x)\circ \tau(B_{\ud})$ is generated by the elements
   \begin{align*}
    F_1 - d^{-1} E_2 K_1^{-1},\quad  F_2 - E_1 K_2^{-1}
  \end{align*}
and hence coincides with $B_{\ud'}$ where $\ud'=(d^{-1},1)$. 
\end{eg}
The above example can immediately be generalized to the following statement.
\begin{prop}
  Assume that $\gfrak$ is finite dimensional and simple. Let $\ud, \ud'\in \ocD$ and denote their only nontrivial entry (if any) by $d_i$ and $d_i'$, respectively. Then $B_{\ud} \sim B_{\ud'}$ if and only if $d_i=d_i'$ or $d_i^{-1}=d_i'$.
\end{prop}
  In the symmetrizable Kac-Moody case there may be more diagram automorphisms which commute with $\tau$ and leave $X$ invariant. Hence one obtains additional Hopf algebra automorphisms which identify $B_{\ud}$ and $B_{\ud'}$ for different $\ud,\ud' \in \ocD$.
\begin{eg}
   As in Remark \ref{rema:D=1} consider $\gfrak=\widehat{\slfrak}_4(\C)$ with $X=\emptyset$ and $\tau=(01)(23)$. By definition, the corresponding quantum symmetric pair coideal subalgebra $B_{\ud}$ for $\ud=(d_0,1,d_2,1)$ is generated by elements
\begin{align*}
  B_0=F_0 {-} d_0 E_1 K_0^{-1}, \,  B_1=F_1 {-} E_0 K_1^{-1}, \,
  B_2=F_2 {-} d_2 E_3 K_2^{-1}, \,  B_3=F_3 {-} E_2 K_3^{-1}. 
\end{align*}
  Consider the diagram automorphism $\tau'=(02)(13)$.
  One has $\tau'(B_{\ud})=B_{\ud'}$ where $\ud'=(d_2,1,d_0,1)$.
\end{eg}
To describe equivalence classes of quantum symmetric pair coideal subalgebras in general, define
\begin{align*}
    \Aut(A,X)^\tau=\{\sigma\in \Aut(A,X)\,|\,\sigma\circ\tau = \tau\circ \sigma\}.
\end{align*}    
Via Proposition \ref{prop:Htil} the action of $\Aut(A,X)^\tau$ on the set of quantum symmetric pair coideal subalgebras induces an action of $\Aut(A,X)^\tau$ on the set $\ocD \times \ocS/\simS$. By Theorem \ref{thm:AutUqg} any automorphism of $\uqgp$ which maps a quantum symmetric pair coideal subalgebra corresponding to $(X,\tau)$ to another such coideal subalgebra, is of the form $\Ad(x)\circ \sigma$ form some $x\in \Htil$ and $\sigma \in \Aut(A,X)^\tau$. This implies the following parametrization of equivalence classes of quantum symmetric pair coideal subalgebras corresponding to the admissible pair $(X,\tau)$.
\begin{thm}\label{thm:equiv}
  There is a one-to-one correspondence between the equivalence classes of quantum symmetric pair coideal subalgebras of $\uqgp$ corresponding to the admissible pair $(X,\tau)$ and the $\Aut(A,X)^\tau$-orbits in $\ocD \times \ocS/\simS$. 
\end{thm}
\subsection{References to Letzter's constructions}\label{sec:LetzterEquiv}
For finite dimensional $\gfrak$, the parameter sets $\ocD$ and $\ocS$ appear implicitly in \cite[Variants 1 and 2]{MSRI-Letzter} and explicitly in \cite[Section 2]{a-Letzter03}. It was asked at the end of \cite[Variant 1]{MSRI-Letzter} if $B_{\ud}$ and $B_{\ud'}$ for $\ud, \ud'\in \ocD$ can be isomorphic as algebras. Example \ref{eg:sl3iso} shows that they can even be isomorphic via a Hopf algebra automorphism of $\uqg$. It remains an interesting question, however, if they can be isomorphic as algebras if $\ud$ and $\ud'$ belong to different $\Aut(A,X)^\tau$-orbits. The main interest in equivalence of quantum symmetric pair coideal subalgebras stems for the classification result \cite[Theorem 5.8]{a-Letzter99a}, \cite[Theorem 7.5]{MSRI-Letzter} which states that all right coideal subalgebras of $\uqg$ which specialize to $U(\kfrak)$ and satisfy a maximality condition (explained in \ref{sec:MaxCond}) are equivalent to a quantum symmetric pair coideal subalgebra. As indicated in Remark \ref{rema:ad-hoc}, it would be interesting to extend this result to the setting of symmetrizable Kac-Moody algebras. Then Theorem \ref{thm:equiv} would provide a parametrization up to Hopf algebra automorphism of all coideal subalgebras of $\uqgp$ which can reasonably be considered as quantum analogs of $U(\kfrak')$.
\section{Specialization}\label{sec:specialize}
In the present section we consider the limit of quantum symmetric pair coideal subalgebras as $q$ tends to $1$. This is done using non-restricted specialization as outlined in \cite[1.5]{a-DCoKa90}. We follow the presentation in \cite{b-HongKang02}. All through this section we abbreviate $\qfield=\field(q)$.
\subsection{Specialization of $\theta_q(X,\tau)$}\label{sec:spec}
We recall specialization of the $\qfield$-algebra $\uqgp$ following \cite[3.3, 3.4]{b-HongKang02}. Let $\bA=\field[q]_{(q-1)}$ denote the localization of the polynomial ring $\field[q]$ with respect to the prime ideal generated by $q-1$. For any $i\in I$ define 
\begin{align*}
   (K_i; 0)_q=\frac{K_i-1}{q-1}.
\end{align*}
The $\bA$-form $\cU'_\bA$ of $\uqgp$ is defined to be the $\bA$-subalgebra of $\uqgp$ generated by the elements $E_i, F_i, K_i^{\pm 1}$, and $(K_i;0)_q$ for all $i\in I$. Consider $\field$ as an $\bA$-module via evaluation at $1$. The algebra 
\begin{align*}
  \cU_1' = \field\ot_\bA \cU'_\bA
\end{align*} 
is called the specialization of $\uqgp$ at $q=1$. For any $x\in \cU'_\bA$ we denote its image in $\cU'_1$ by $\overline{x}$. The following fact is well known and a minor variation on \cite[Theorem 3.4.9]{b-HongKang02}.
\begin{thm}\label{thm:specialization}
  There exists an isomorphism of algebras $\cU'_1\rightarrow U(\gfrak')$ such that $\overline{E}_i\mapsto e_i$, $\overline{F_i}\mapsto f_i$, and $\overline{(K_i;0)_q}\mapsto \epsilon_ih_i$.
\end{thm}
Let now $\phi:\uqgp\rightarrow \uqgp$ be a $\qfield$-linear map. We say that $\phi$ is specializable if $\phi(\cU'_\bA)\subseteq \cU'_\bA$. In this case $\phi$ induces a $\field$-linear map $\id\ot \phi|_{\cU'_\bA}:\cU'_1\rightarrow \cU'_1$. Via Theorem \ref{thm:specialization} one thus obtains a $\field$-linear map $\overline{\phi}: U(\gfrak') \rightarrow U(\gfrak')$. We say that $\phi$ specializes to $\overline{\phi}$. Observe that if $\qfield$-linear maps $f, g: \uqgp\rightarrow \uqgp$ are both specializable then so is $f\circ g$ and $\overline{f\circ g}=\overline{f} \circ \overline{g}$.
We now apply specialization to the quantum involution $\theta_q(X,\tau)$.
\begin{prop}\label{thm:special-prop}
   For any admissible pair $(X,\tau)$ the quantum involution
   $\theta_q(X,\tau)$ specializes to $\theta(X,\tau)$. 
\end{prop}
\begin{proof}
  Clearly both $\tau$ and $\omega$ are specializable and specialize to the corresponding automorphisms of $U(\gfrak')$. By the explicit formulas given in \cite[37.1.3]{b-Lusztig94} the Lusztig automorphism $T_i$ is also specializable for any $i\in I$. It specializes to the automorphism $\overline{T_i}: U(\gfrak')\rightarrow U(\gfrak')$ given by 
  \begin{align*}
    e_i\mapsto -f_i,& \quad f_i \mapsto -e_i, \quad h\mapsto r_i(h),\\
    e_j\mapsto \frac{\ad(e_i)^{-a_{ij}}}{(-a_{ij})!}(e_j),&\quad  f_j\mapsto\frac{\ad( f_i)^{-a_{ij}}}{(-a_{ij})!}(f_j), \quad \mbox{ for $j\neq i$.}
  \end{align*}
One checks by an $\slfrak_2$-argument that $\overline{T_i}=\Ad(m_i)|_{U(\gfrak')}$. Hence $T_X$ specializes to $\Ad(m_X)|_{U(\gfrak')}$. Now the theorem follows from the fact observed above that specialization is compatible with composition of maps.  
\end{proof}
We give an immediate application of the above proposition. We call a pair of multi-indices $(\uc,\us)\in (\qfield^\times)^{I\setminus X} \times \qfield^{I\setminus X}$ specializable if $c_i,s_i\in \bA$ and $c_i(1)=1$ for all $i\in I\setminus X$.
\begin{cor}\label{cor:specialize-Bi}
  Let $(\uc,\us)\in (\qfield^\times)^{I\setminus X}\times \qfield^{I\setminus X}$ be specializable. The generators $B_i$ of $B_{\uc,\us}$ belong to $\cU'_\bA$ and satisfy $\overline{B_i}=f_i+\theta(f_i)+\overline{s}_i$ for all $i\in I\setminus X$.
\end{cor} 
\begin{proof}
  Let $i\in I\setminus X$. As $F_iK_i\in \cU'_\bA$ the above proposition gives $\theta_q(F_iK_i)K_i^{-1}\in \cU'_\bA$ and
  $\overline{\theta_q(F_iK_i)K_i^{-1}}=\theta(f_i)$. Together with the assumptions on $c_i$ and $s_i$ this implies that $B_i\in (B_{\uc,\us})_\bA$ and $\overline{B_i}=f_i+\theta(f_i)+\overline{s}_i$.
\end{proof}
\subsection{Properties of $\bA$-modules}
The ring $\bA$ is a principal ideal domain. We begin by recalling two facts about principal ideal domains which will be repeatedly used in the sequel. The first statement of the following proposition is a consequence of the fundamental theorem of finitely generated modules over principle ideal domains. The second statement can be found in \cite[Corollary 6.3]{b-Eisenbud95}.
\begin{prop}\label{prop:PID}
  Let $R$ be  principal ideal domain.
  \begin{enumerate}
    \item Any finitely generated, torsion-free module over $R$ is free.
    \item An $R$-module is torsion-free if and only if it is flat.
  \end{enumerate}
\end{prop}
As a first application one obtains the following result.
\begin{lem}\label{lem:0spec}
  Let $x\in \cU'_\bA$. Then $\overline{x}=0$ if and only if $x\in (q-1)\cU'_\bA$. 
\end{lem}
\begin{proof}
 Consider the short exact sequence
  \begin{align*}
    0\rightarrow (q-1)\bA\rightarrow \bA \rightarrow \underbrace{\bA/(q-1)\bA}_{\cong \field} \rightarrow 0
  \end{align*}
 of $\bA$-modules.   By the above proposition $\cU'_\bA$ is a flat $\bA$-module. Hence, tensoring by $\cU'_\bA$, one obtains an exact sequence
  \begin{align*}
    0\rightarrow \underbrace{(q-1)\bA \ot_\bA \cU'_\bA}_{\cong (q-1)\cU'_\bA} \rightarrow \cU'_\bA \rightarrow \field\ot_\bA\cU'_\bA \rightarrow 0.
  \end{align*}
  This proves the lemma.
\end{proof}
Let $W$ be a $\qfield$-vector space. All $\bA$-submodules of $W$ are torsion-free. Hence, for any $\bA$-submodule $M$ of $W$ the map $M\rightarrow M\ot_\bA \qfield$ is injective.
As a consequence of the above proposition one obtains the following lemma which will be used in the Subsection \ref{sec:A-triang} to verify triangular decompositions over $\bA$ analog to those in Section \ref{sec:triang}.
\begin{lem}\label{lem:obvious-isos}
  Let $W$ be a $\qfield$-vector space and let $M$ and $M'$ be $\bA$-submodules of $W$. Let $M_\qfield$ and $M'_\qfield$ denote the $\qfield$-vector subspace of $W$ generated by $M$ and $M'$, respectively. Then the following hold:
  \begin{enumerate}
    \item The map $\iota_M: M\ot_\bA \qfield\rightarrow M_\qfield$ is an isomorphism.
    \item The map $\iota_{M,M'}: M \ot_\bA M'\rightarrow M_\qfield \ot_\qfield M'_\qfield$ is injective.
  \end{enumerate}
\end{lem}
\begin{proof}
  (1) Assume that $m=\sum_{i=1}^N m_i\ot a_i \in \ker(\iota_M)$. Let $\tilde{M}$ denote the $\bA$-submodule of $M$ generated by the set $\{m_i\,|\,i=1,\dots, N\}$. Being finitely generated, $\tilde{M}$ is a free $\bA$-module by Proposition \ref{prop:PID}.(1). We can choose a basis $\{b_i\,|\,i=1,\dots,N'\}$ and write
$m=\sum_{i=1}^{N'} b_i\ot a_i'$ for some $a_i'\in \qfield$. There exists $n\in \N$ such that $(q-1)^na_i'\in \bA$ for all $i=1,\dots,N'$. Hence $m\in \ker(\iota_M)$ implies that $0=\iota_M(m(q-1)^n)=\sum_{i=1}^{N'} b_i a_i'(q-1)^n$. By linear independence of the set $\{b_i\,|\,i=1,\dots,N'\}$ one obtains $a_i'(q-1)^n=0$ and hence $m=0$. This proves that $\iota_M$ is injective, and surjectivity holds by construction.

 (2) Part (1) implies that
\begin{align*}
  M_\qfield \ot_\qfield M'_\qfield \cong M \ot_\bA \ot \qfield \ot_\qfield \qfield \ot_\bA M'
  \cong M \ot_\bA \qfield \ot_\bA M' \cong (M\ot_\bA M')\ot_\bA \qfield.
\end{align*}
The $\bA$-module $M$ is torsion free and hence flat by Proposition \ref{prop:PID}. Therefore the map $M\ot_\bA M'\rightarrow M\ot_\bA(M'\ot_\bA \qfield)\cong (M\ot_\bA M')\ot_\bA \qfield$ is injective. This proves that $\iota_{M,M'}$ is injective.
\end{proof}
\subsection{Specialization and triangular decompositions}\label{sec:A-triang}
For any subspace $W\subset \uqgp$ we define $W_\bA= W\cap \cU'_\bA$ and $\overline{W}=\field\ot_\bA W_\bA\subset \cU'_1 $. By \cite[Proposition 3.3.3]{b-HongKang02} the multiplication map yields an isomorphism
\begin{align}\label{eq:A-triang}
  U^+_\bA \ot_\bA U^0_\bA{}'\ot _\bA U^-_\bA \cong \cU'_\bA
\end{align}
analogous to the triangular decomposition \eqref{eq:triang-decomp}. By the following theorem all triangular decompositions from Section \ref{sec:triang} also hold true over $\bA$. Set $\bA^\times=\bA\cap \qfield^\times$ and recall the definition of the subspace $B_{\uc,\us,\cJ}$ given in \eqref{eq:BcsJ}.
\begin{thm}\label{thm:A-triang}
  Let $(\uc,\us)\in \cC\times \cS$ be specializable. The multiplication maps give the following isomorphisms of $\bA$-modules.
  \begin{enumerate}
    \item $(U^0_\Theta{}')_\bA \ot_\bA (U'_\Theta)_\bA \cong U^0_\bA{'},$ 
    \item $(V^+_X)_\bA \ot_\bA (\cM^+_X)_\bA \cong U^+_\bA,$
    \item $U^+_\bA \ot _\bA U^0_\bA{}' \ot _\bA (B_{\uc,\us,\cJ})_\bA \cong \cU'_\bA,$
    \item $(\cM^+_X)_\bA \ot _\bA (U^0_\Theta{}')_\bA \ot _\bA  (B_{\uc,\us,\cJ})_\bA \cong (B_{\uc,\us})_\bA,$
    \item $(V_X^+)_\bA\ot _\bA (U'_\Theta)_\bA \ot_\bA (B_{\uc,\us})_\bA \cong \cU'_\bA$.
  \end{enumerate}
\end{thm}
\begin{proof}
  Injectivity follows in all five cases from Lemma \ref{lem:obvious-isos} and from the triangular decompositions in Section \ref{sec:triang}. It remains to prove surjectivity.
  
 (1) ${U^0_\bA}{}'$ is the $\bA$-subalgebra of $U^0{}'$ generated by the elements
   \begin{align}\label{eq:Ki}
     K_i, \quad K_i^{-1}, \quad (K_i;0)_q
   \end{align}
   for all $i\in I$. If $i\in X$ or $i\in I^*$ then the elements \eqref{eq:Ki} belong to $(U^0_\Theta{}')_\bA$ and $(U'_\Theta)_\bA$, respectively. Hence in this case, the elements \eqref{eq:Ki} are in the image of the multiplication map. If $i\in I\setminus (I^\ast\cup X)$ then $K_iK_{\tau(i)}^{-1}, K_i^{-1} K_{\tau(i)}\in ({U^0_\Theta}')_\bA$ and $K_{\tau(i)}^{\pm 1}\in (U'_\Theta)_\bA$. Hence $K_i$ and $K_i^{-1}$ lie in the image of the multiplication map. Finally,
the relation 
   \begin{align*}
     (K_i;0)_q=(K_iK_{\tau(i)}^{-1};0)_q - K_i\,(K_{\tau(i)}^{-1};0)_q
   \end{align*}
shows that in this case $(K_i;0)_q$ also lies in the image of the multiplication map.   

 (2) $U_\bA^+$ is the $\bA$-subalgebra of $\uqgp$ generated by all $E_i$ for $i\in I$. If $i\in X$ then $E_i\in (\cM^+_X)_\bA$. If $i\notin X$ then $E_i\in (V^+_X)_\bA$. Hence, all $E_i$ lie in the image of the multiplication map.
 
 (3) Recall the filtration $\cF^\ast$ of $U^-$ defined at the beginning of Subsection \ref{sec:BJbasis}. Via the triangular decomposition \eqref{eq:triang-decomp} the filtration $\cF^\ast$ extends to a filtration of $\uqgp$ by
 \begin{align*}
   \cF^n(\uqgp)=U^+\ot U^0{'}\ot \cF^n(U^-). 
 \end{align*}
We show by induction on $n$ that $\cF^n(\uqgp)_\bA$ lies in the image of the multiplication map. For $n=0$ this holds true because \eqref{eq:A-triang} implies that
  \begin{align*}
    \cF^0(\uqgp)_\bA = U^+_\bA \ot_\bA U^0_\bA{'}\ot_\bA \bA.
  \end{align*} 
For $y\in \cF^n(\uqgp)_\bA$, again by \eqref{eq:A-triang}, there exist $u_J\in U^+_\bA\ot_\bA U^0_\bA{'}$ such that
  \begin{align*}
    y-\sum_{J\in I^n\cap\cJ} u_J F_J \in \cF^{n-1}(\uqgp)_\bA.
  \end{align*}  
Hence  
  \begin{align*}
    y-\sum_{J\in I^n\cap\cJ} u_J B_J \in \cF^{n-1}(\uqgp)_\bA.
  \end{align*}  
By induction hypothesis $y$ hence lies in the image of the multiplication map (3).  

 (4) Let $b\in (B_{\uc,\us})_\bA$. By Proposition \ref{prop:leftBasis} one has $b=\sum_{J\in \cJ}a_J B_J$ for some $a_J\in \cM_X^+ U^0_\Theta{}'$. Let $m\in \N_0$ be maximal such that $a_J\neq 0$ for some $J\in \cJ$ with $|J|=m$. By the triangular decomposition \eqref{eq:A-triang} one can write
  \begin{align*}
    b=\sum_{J\in \cJ, |J|\le m} b_J F_J 
  \end{align*}
for some  $b_J\in (U^+ U^0{'})_\bA$. As $B_i-F_i\in (U^+ U^0{'})_\bA$ one obtains $a_J=b_J$ if $|J|=m$.
Hence  $a_J\in (\cM^+_XU^0_\Theta{}')_\bA$ for all $J\in \cJ$ with $|J|=m$. By induction on $m$ one obtains $a_J\in (\cM^+_XU^0_\Theta{}')_\bA$ for all $J\in \cJ$. 
 
 (5) Surjectivity follows by inserting (4) into the left hand side of (5) and using (3) together with (1) and (2).
\end{proof}
\subsection{Specialization of $B_{\uc,\us}$}
We now want to show that for specializable $(\uc,\us)\in \cC\times \cS$ the quantum symmetric pair coideal subalgebra $B_{\uc,\us}$ specializes to $U(\kfrak')$. To this end, let $\gfrak_X^+$ and $\hfrak'_\theta$ denote the Lie subalgebras of $\gfrak'$ generated by the sets $\{e_i\,|\, i\in X\}$ and $\{h_i+\theta(h_i)\,|\,i\in I\}$, respectively.  One verifies that
\begin{align}\label{eq:mh-special}
  \overline{\cM_X^+}&=U(\gfrak^+_X) & \overline{U^0_\Theta{}'}&=U(\hfrak'_\theta{}).
\end{align}  
With this observation, Theorem \ref{thm:A-triang}.(4) gives the following result.
\begin{thm}\label{thm:Bcspecial}
  Let $(\uc,\us)\in \cC\times \cS$ be specializable. Then $\overline{B_{\uc,\us}}=U(\kfrak')$.
\end{thm}
\begin{proof}
  As $(\uc,\us)$ is specializable, Corollary \ref{cor:specialize-Bi} implies that $f_i+\theta(f_i) + \overline{s}_i\in \overline{B_{\uc,\us}}$ for all $i\in I\setminus X$. As $F_i\in B_{\uc,\us}$ one has $f_i\in \overline{B_{\uc\,\us}}$ for all $i\in X$. Moreover, $\gfrak^+_X \subset \overline{B_{\uc,\us}}$ and $\hfrak'_\theta{}\subset \overline{B_{\uc,\us}}$ by \eqref{eq:mh-special}. By Corollary \ref{cor:U(g)generators} this proves that $U(\kfrak')\subseteq \overline{B_{\uc,\us}}$. Conversely, Theorem \ref{thm:A-triang}.(4) and \eqref{eq:mh-special} imply that $\overline{B_{\uc,\us}}\subseteq U(\kfrak')$.
\end{proof}
\subsection{A maximality condition}\label{sec:MaxCond}
Observe that $U^+_\bA$ is generated by all ordered monomials in the generators $ E_i$, $i\in I$, as an $\bA$-module. There are only finitely many such ordered monomials of any given weight. Hence, for any finite dimensional subspace $V\subset U^+$ one obtains that $V_\bA$ is a submodule of a finitely generated $\bA$-module. As $\bA$ is Noetherian this implies that $V_\bA$ is finitely  generated. Similarly one shows that if $V$ is a finite dimensional subspace of $U^0{}'$ then $V_\bA$ is finitely generated. Combining these observations with the triangular decomposition \ref{eq:A-triang} one obtains the following result.
\begin{lem}\label{lem:finfree}
  Let $V\subset \uqgp$ be a finite dimensional subspace. Then $V_\bA$ is a finitely generated, free $\bA$-submodule of $\cU'_\bA$.
\end{lem}
The above lemma and Theorem \ref{thm:A-triang}.(5) imply the following result.
\begin{lem}\label{lem:free-quot}
  Let $(\uc,\us)\in \cC\times \cS$ be specializable and $u\in \uqgp$. Then 
  \begin{align*}
     \big( \qfield u + B_{\uc,\us}\big)_\bA \big/ (B_{\uc,\us})_\bA
  \end{align*}
  is a finitely generated, free $\bA$-module.
\end{lem}
\begin{proof}
  If $u\in B_{\uc,\us}$ then there is nothing to show. Otherwise, choose finite dimensional subspaces
  $V\subset V_X^+$, $T\subset U'_\Theta$, and $D \subset B_{\uc,\us}$ such that $u\in V\ot T \ot D$ with respect to the quantum Iwasawa decomposition. We may assume that the unit $1$ of $\uqgp$ is contained in both $V$ and $T$. Set $V_+=V\cap \ker{\vep}$ and $T_+=T\cap \ker{\vep}$. 
Then the subspace $\qfield u + B_{\uc,\us}$ is contained in the direct sum 
\begin{align*}
  (V_+\ot T \ot D) \bigoplus (\qfield \ot T_+ \ot D) \bigoplus B_{\uc,\us}.
\end{align*}
Observe that $V_{\bA}=(V_+)_\bA \oplus \bA$ and $T_{\bA}=(T_+)_\bA \oplus \bA$. 
By Theorem \ref{thm:A-triang}.(5) this implies that $(\qfield u + B_{\uc,\us})_\bA$ is contained in 
  \begin{align*}
  ((V_+)_\bA\ot_\bA T_\bA \ot_\bA D_\bA) \bigoplus (\bA \ot_\bA (T_+)_\bA \ot_\bA D_\bA) \bigoplus (B_{\uc,\us})_\bA.
\end{align*}
Hence one obtains an injective $\bA$-module homomorphism
\begin{align*}
  \big( \qfield u + B_{\uc,\us}\big)_\bA \big/ (B_{\uc,\us})_\bA \rightarrow ((V_+)_\bA\ot_\bA T_\bA \ot_\bA D_\bA) \bigoplus (\bA \ot_\bA (T_+)_\bA \ot_\bA D_\bA).
\end{align*}
By Lemma \ref{lem:finfree} the $\bA$-module on the right hand side is finitely generated and free. Hence, by Proposition \ref{prop:PID}, so is $\displaystyle \big(\qfield u + B_{\uc,\us}\big)_\bA \big/ (B_{\uc,\us})_\bA$ as $\bA$ is Noetherian.  
\end{proof}
\begin{thm}\label{thm:maxcond}  
   Let $W$ be a vector subspace of $\uqgp$ which contains $B_{\uc,\us}$ for some specializable $(\uc,\us)\in \cC\times \cS$. If $\overline{W}=U(\kfrak')$ then $W=B_{\uc,\us}$.
\end{thm}
\begin{proof}
   Assume that there exists $u\in W \setminus B_{\uc,\us}$. By Lemma \ref{lem:free-quot} the nonzero $\bA$-module $\displaystyle   N_\bA=\big( \qfield u + B_{\uc,\us}\big)_\bA \big/ (B_{\uc,\us})_\bA$ is free. Let $v\in \big( \qfield u + B_{\uc,\us}\big)_\bA$ be a representative of a basis element of $N_\bA$. The assumption $\overline{W}=U(\kfrak')$ together with Theorem \ref{thm:Bcspecial} imply that $\overline{v}=\overline{b}$ for some $b\in (B_{\uc,\us})_\bA$.  By Lemma \ref{lem:0spec} one obtains
 $v-b\in(q-1) \cU'_\bA\cap\big( \qfield u + B_{\uc,\us}\big)=(q-1)\big( \qfield u + B_{\uc,\us}\big)_\bA$. This contradicts the assumption that $v$ represents a basis element of $N_\bA$.     
\end{proof}
\subsection{References to Letzter's constructions}
  Specialization is a major theme in both \cite{a-Letzter99a} and \cite{MSRI-Letzter}. For finite dimensional $\gfrak$, Theorems \ref{thm:Bcspecial} and \ref{thm:maxcond} are stated as \cite[Theorem 4.8]{a-Letzter99a} and \cite[(7.24), (7.25)]{MSRI-Letzter}. The additional assumption that $C$ is an algebra made in \cite[Theorem 4.8]{a-Letzter99a} and \cite[(7.25)]{MSRI-Letzter} seems not to be necessary for the proof. The fact that the triangular decompositions from Section \ref{sec:triang} also hold for the $\bA$-forms, Theorem \ref{thm:A-triang}, is not made explicit in Letzter's papers. Instead she uses a process of `rescaling' and  `subtracting', see \cite[proof of Theorem 4.8]{a-Letzter99a} and \cite[before (7.23)]{MSRI-Letzter}. The fact that this process terminates seems to boil down to the properties of the principal ideal domain $\bA$ stated in Proposition \ref{prop:PID}. A proof of Lemma \ref{lem:0spec} in a similar spirit can be found in \cite[Theorem 2.5]{a-Letzter99a}. As already indicated in Subsection \ref{sec:LetzterEquiv}, Letzter's analysis proceeds one step beyond the present paper with the classification theorem \cite[Theorem 5.8]{a-Letzter99a}, \cite[Theorem 7.5]{MSRI-Letzter}. It is to be expected that a rigorous proof of this result in the Kac-Moody case will involve the triangular decompositions of $\bA$-forms given in Theorem \ref{thm:A-triang}. 
\section{Twisted quantum loop algebras of the second kind}\label{sec:TwistedQLoop}
  In this section and the next we will apply the theory developed so far to construct two classes of quantum symmetric pair coideal subalgebras which in special cases have been considered previously in the literature. The present section is devoted to quantum versions of twisted loop algebras of the second kind which are defined in general in Subsection \ref{sec:symLoop}. It is then shown, in Subsection \ref{sec:symLoopFRT}, how the twisted $q$-Yangians introduced by Molev, Ragoucy, and Sorba in \cite{a-MolRagSor03} appear as examples of such twisted quantum loop algebras. 
\subsection{Symmetric loop algebras and their quantization}\label{sec:symLoop}
Let $\gfrak$ be a finite-dimensional simple Lie algebra over $\field$ and let $\theta:  \gfrak\rightarrow \gfrak$ be an involutive automorphism. Assume that $\rank(\gfrak)=n$ and that $I=\{1,2,\dots,n\}$.
As before write $A=(a_{ij})_{i,j\in I}$ to denote the Cartan matrix of $\gfrak$. By Theorem \ref{classThm} the involution $\theta$ is determined up to conjugation by an admissible pair $(X, \tau)$.

Let $\hat{\gfrak}=\gfrak\ot \field[t,t^{-1}]\oplus \field c\oplus \field d$ be the corresponding untwisted affine Kac-Moody algebra as defined in \cite[7.2]{b-Kac1}. Set $\hat{I}=\{0,1,\dots,n\}$ and recall that the Cartan matrix of $\hat{\gfrak}$ is given by $\hat{A}=(a_{ij})_{i,j\in \hat{I}}$. Here
$a_{00}=2$ and $a_{i0}=a_{0i}=-\alpha_{max}(h_i)$ for all $i\in I$ where $\alpha_{max}$ denotes the highest root in the finite root system corresponding to $\gfrak$. Let $\hat{\hfrak}$ and $\hat{\bfrak}$ denote the standard Cartan and Borel subalgebra of $\hat{\gfrak}$, respectively. Consider the involutive automorphism $\hat{\theta}:\hat{\gfrak}\rightarrow \hat{\gfrak}$ given by 
\begin{align}\label{eq:thetah}
  \hat{\theta}(x\ot t^n)=\theta(x)\ot t^{-n},\qquad\hat{\theta}(c)=-c, \qquad \hat{\theta}(d)=-d.
\end{align}
Let $\hat{\tau}:\hat{I}\rightarrow \hat{I}$ be the map given by $\hat{\tau}|_{I}=\tau$ and $\hat{\tau}(0)=0$. It follows from the explicit description of $\hat{\theta}$ that the induced involution $\hat{\Theta}:\hat{\hfrak}^\ast \rightarrow \hat{\hfrak}^\ast$ is given by $\hat{\Theta}=-w_X \circ \hat{\tau}$. Hence $(\hat{\hfrak},\hat{\bfrak})$ is a split pair for $\hat{\theta}$ in the sense of Definition \ref{splitpair} and $\hat{\theta}$ is an involutive automorphism of the second kind. The preceding discussion implies the following lemma.
\begin{lem}
  Let $(X,\tau)$ be an admissible pair for $\gfrak$ and $\theta=\theta(X,\tau)$.
  \begin{enumerate}
    \item There exists  $\hat{\tau}\in \Aut(\hat{I},X)$ such that $\hat{\tau}(0)=0$ and $\hat{\tau}(i)=\tau(i)$ for all $i\in I$.
    \item The pair $(X,\hat{\tau})$ is admissible for $\hat{\gfrak}$. 
    \item The automorphism $\hat{\theta}:\hat{\gfrak} \rightarrow \hat{\gfrak}$ given by \eqref{eq:thetah} is the involutive automorphism of the second kind corresponding to the admissible pair $(X,\hat{\tau})$.
  \end{enumerate}  
\end{lem}
By the above lemma, all the results of this paper can be applied to the involutive automorphism $\hat{\theta}$ corresponding to the admissible pair $(X,\hat{\tau})$. In particular, for any $\uc\in \cC$ and $\us \in \cS$ one obtains the corresponding quantum symmetric pair coideal subalgebra $B_{\uc,\us}$ of $\uqgp$.

The Lie algebra $\hat{\gfrak}'=\gfrak\ot \field[t,t^{-1}]\oplus \field c$ has a one-dimensional center spanned by $c$. The Cartan matrix of $\hat{\gfrak}$ is of rank $n$. There exist uniquely determined integers $b_j\in \N$, $j=0,1,\dots,n$, with $b_0=1$, given in \cite[p. 54]{b-Kac1}, such that \begin{align*}
  \sum_{j=0}^n b_j a_{ij}=0 \qquad \mbox{for all $i\in \hat{I}$}.
\end{align*}  
The above relation implies that $\alpha_i(\sum_{j=0}^n b_j\epsilon_j h_j)=0$ for all $i\in \hat{I}$ and therefore $\sum_{i=0}^n b_j \epsilon_j h_j$ spans the center of $\gfrak'$. The quotient
\begin{align*}
  L(\gfrak)=\hat{\gfrak}'/\field c\cong \gfrak\ot \field[t,t^{-1}]
\end{align*}
is the loop algebra corresponding to $\gfrak$. The involution $\hat{\theta}$ of $\hat{\gfrak}'$ induces an involutive automorphism $\hat{\theta}_L : L(\gfrak)\rightarrow L(\gfrak)$ and the invariant Lie subalgebra $\kfrak'_L=\{x \in L(\gfrak)\,|\,\hat{\theta}_L(x)=x\}$ is isomorphic to $\kfrak'$ because $\hat{\theta}(c)=-c$.

Consider the quantum analog of the central element $c$ given by $K_c=\prod_{j=0}^n K_j^{b_j}$. The Hopf algebra $\uqLg$ defined by
\begin{align*}
  \uqLg=\uqgp/(K_c-1)\uqgp
\end{align*}
is a $q$-analog of the universal enveloping algebra of the loop algebra $L(\gfrak)$. Let $\pi:\uqgp \rightarrow \uqLg$ denote the canonical projection. We call the image $\pi(B_{\uc,\us})$ a \textit{twisted quantum loop algebra (of the second kind)}. It is a right coideal subalgebra of $\uqLg$.  By the following lemma the twisted quantum loop algebra $\pi(B_{\uc,\us})$ is isomorphic to $B_{\uc,\us}$ as an algebra.
\begin{lem}
  For any $\uc \in \cC$ and $\us \in \cS$ one has $(K_c-1)\uqgp \cap B_{\uc,\us}=\emptyset$.
\end{lem} 
\begin{proof}
  By the quantum Iwasawa decomposition (Proposition \ref{prop:qIwasawa}) and the fact that $K_c-1$ is central in $\uqgp$, it suffices to show that $(K_c-1)U_\Theta'\cap U^0_\Theta{}'=\emptyset$. Any element of $U_\Theta'$ can be written as a Laurent polynomial in $K_0$ with coefficients in $\field(q)\langle K_i^{\pm 1}\,|\, i\in I^\ast \setminus \{0\}\rangle$. The product of any such a Laurent polynomial with $(K_c-1)$ will always depend on $K_0$ and therefore does not belong to $U^0_\Theta{}'$.
\end{proof}
\begin{rema}
  A remark on terminology is in order. Besides \eqref{eq:thetah} there is another way to extend the involution $\theta$ of $\gfrak$ to $\hat{\gfrak}$. Indeed, consider the involutive automorphism $\thetah_1:\hat{\gfrak}\rightarrow \hat{\gfrak}$ given by 
  \begin{align}\label{eq:thetah1}
    \thetah_1(x\ot t^n)=\theta(x)\ot (-t^n),\qquad\thetah_1(c)=c, \qquad \thetah_1(d)=d.
  \end{align}
The automorphism $\thetah_1$ is of the first kind. Formula \eqref{eq:thetah1} also defines an involution of the loop algebra $L(\gfrak)$. The existing literature on twisted loop algebras is concerned with the Lie subalgebra of $L(\gfrak)$ fixed under $\thetah_1$, mostly in the case when $\theta$ is a diagram automorphism. It seems natural to say that the subalgebra of $L(\gfrak)$ fixed by $\thetah_1$ is a twisted loop algebra of the first kind. The terminology for twisted loop algebras then reflects the terminology for involutions of Kac-Moody algebras.    

In \cite{a-MolRagSor03} Molev, Ragoucy, and Sorba use the name \textit{twisted $q$-Yangians} for quantum analogs of $U(\kfrak'_L)$. As pointed out in the introduction of \cite{a-MolRagSor03}, however, Ol'shanski\u{i}'s original definition of twisted Yangians involves involutions of the first kind \cite{a-Olsh92}. For this reason we prefer to call the quantum analogs of $U(\kfrak'_L)$ twisted quantum loop algebras of the second kind. 
\end{rema}
The structure theory of $\uqLg$ is very similar to the structure theory of $\uqgp$. In particular, one can formulate specialization for $\uqLg$. The results of Section \ref{sec:specialize} literally translate to $\uqLg$ and $\pi(B_{\uc,\us})$ and one obtains the following result in the same way as Theorems \ref{thm:Bcspecial} and \ref{thm:maxcond}.
\begin{thm}\label{thm:maxcond2}
  Let $\uc\in \cC$ and $\us \in \cS$ be specializable.
    \begin{enumerate}
      \item The subalgebra $\pi(B_{\uc,\us})$ of $\uqLg$ specializes to $U(\kfrak'_L)$.
      \item Let $W$ be a vector subspace of $\uqLg$ which contains $\pi(B_{\uc,\us})$. If $W$ specializes to $U(\kfrak'_L)$ then $W=\pi(B_{\uc,\us})$.  
    \end{enumerate}
\end{thm}
\subsection{FRT-realizations of twisted quantum loop algebras of the second kind}\label{sec:symLoopFRT}
In this subsection we restrict to the case $\gfrak=\slfrak_{N}(\field)$ for $n=N-1\in \N$. We assume that the Cartan matrix $A=(a_{ij})_{1\le i,j\le n}$ is given in the standard form 
\begin{align*}
  a_{ij}=\begin{cases}
           2 & \mbox{if $i=j$,}\\
           -1& \mbox{if $|i-j|=1$,}\\
           0& \mbox{else.}
         \end{cases}
\end{align*}
Let $\sigma:I\rightarrow I$ denote the only non-trivial element in $\Aut(A)$.
There exist three types of admissible pairs for $\gfrak$:
\begin{align*}
  \mbox{AI: }& (X,\tau)=(\emptyset, \id_I),\\
  \mbox{AII: }& (X,\tau)=((1,3,\dots,2m-1), \id_I) \quad\mbox{if $n=2m-1$,}\\
  \mbox{AIII/IV: }& (X,\tau)=((r+1,r+2,\dots,n-r), \sigma) \quad \mbox{for any $1\le r\le (n+1)/2$.} 
\end{align*}
The corresponding fixed Lie subalgebras of $\gfrak$ coincide with $\sofrak_N(\field)$, $\spfrak_{2m}(\field)$, and $\slfrak_N(\field)\cap(\glfrak_r(\field) \oplus \glfrak_{N-r}(\field))$, respectively. In \cite{a-MolRagSor03} Molev, Ragoucy, and Sorba construct $q$-analogs of $U(\kfrak'_L)$ in the cases AI and AII. Here we recall their construction and relate it to the coideal subalgebras $\pi(B_{\uc,\us})$ of $\uqLg$ constructed in the last subsection. The construction in \cite{a-MolRagSor03} uses the FRT realization of quantized enveloping algebras \cite{a-FadResTak2} as opposed to the Drinfeld-Jimbo realization used in the present paper. The translation between the different realizations of quantized enveloping algebras is given in detail by Frenkel and Mukhin in \cite{a-FrenkelMukhin02}.  Moreover, \cite{a-MolRagSor03} work with $\glfrak_N(\field)$ instead of $\slfrak_N(\field)$. Let $\mbox{Id}_N$ denote the identity matrix in $\glfrak_N(\field)$. The involution $\theta=\theta(X,\tau):\slfrak_N(\field)\rightarrow \slfrak_N(\field)$ corresponding to each of the three admissible pairs described above can be extended to an involutive automorphism of $\glfrak_N(\field)$ by setting $\theta(\lambda \mbox{Id}_N)=-\lambda \mbox{Id}_N$ for any $\lambda\in \field$. Just as in the previous subsection one obtains an involutive automorphism of the loop algebra $L(\glfrak_N(\field))=\glfrak_N(\field)\ot \field[t,t^{-1}]$ by
\begin{align*}
  \hat{\theta}: L(\glfrak_N(\field)) \rightarrow L(\glfrak_N(\field)), \qquad \hat{\theta}(x\ot t^k)=\theta(x)\ot t^{-k}.
\end{align*}
Let $L(\glfrak_N(\field))^{\hat{\theta}}$ denote the Lie subalgebra of $L(\glfrak_N(\field))$ fixed under $\hat{\theta}$. For $\lambda=1$ the Lie subalgebra of $\glfrak_N(\field)$ fixed under $\theta$ coincides with the fixed Lie subalgebra in $\slfrak_N(\field)$. However, $L(\glfrak_N(\field))^{\hat{\theta}}$ contains an infinite dimensional central subspace spanned by the elements $C_k=\mbox{Id}_N\ot t^k - \mbox{Id}_N\ot t^{-k}$ for all $k\in \N$ which are not contained in $L(\slfrak_N(\field))$.

Following \cite[2.3]{a-FrenkelMukhin02} and \cite[Section 3]{a-MolRagSor03} we now recall the definition of the Hopf algebra $\uqRglN$ which is a $q$-analog of the universal enveloping algebra of $L(\glfrak_N(\field))$. For $1\le i,j\le N$ let $E_{ij}$ denote the $N \times N$ matrix with only nonzero entry $1$ in the $(i,j)$-th position and set 
\begin{align*}
   R(z,w)=&(z-w) \sum_{i\neq j} E_{ii} \ot E_{jj} + (q^{-1}z - q w)\sum_{i} E_{ii} \ot E_{ii} \\
          & +(q^{-1}-q)z \sum_{i>j} E_{ij} \ot E_{ji} + (q^{-1} - q) w \sum_{i<j} E_{ij} \ot E_{ji}.
\end{align*}
By definition, the algebra $\uqRglN$ is generated by the coefficients $L^{\pm}_{ij}[\pm k]$ of the formal series 
\begin{align*}
  L_{ij}^{\pm}(z) &= \sum_{k=0}^\infty L_{ij}^{\pm}[\pm k] z^{\pm k}& 1\le i,j\le N
\end{align*}
subject to the relations
\begin{align}
  L_{ji}^+[0]&=L_{ij}^-[0]=0 & &\mbox{ for } 1\le i<j\le N,\label{eq:L0}\\
  L_{ii}^-[0] L_{ii}^+[0] &= L_{ii}^+[0] L_{ii}^-[0] =1 & &\mbox{ for } 1\le i\le N,\nonumber\\
  R(z,w) L_1^{\pm}(z) L_2^{\pm}(w) &= L_2^{\pm}(w) L_1^{\pm}(z) R(z,w), && \label{eq:RLL1}\\
  R(z,w) L_1^{+}(z) L_2^{-}(w) &= L_2^{-}(w) L_1^{+}(z) R(z,w).\label{eq:RLL2}
\end{align}
Here \eqref{eq:RLL1} and \eqref{eq:RLL2} are equations in $\uqRglN[[z,z^{-1},w,w^{-1}]]\ot \End(\field(q)^N)^{\ot 2}$ and
\begin{align*}
  L_1^{\pm}(z)&=\sum_{i,j=1}^N L_{ij}^{\pm}(z)\ot E_{ij}\ot \mbox{Id}_N, & L_2^{\pm}(z)&=\sum_{i,j=1}^N L_{ij}^{\pm}(z)\ot \mbox{Id}_N \ot E_{ij}.
\end{align*}
The relations \eqref{eq:RLL1} and \eqref{eq:RLL2} can be written more explicitly as
\begin{align}
  &(q^{-\delta_{ij}} z- q^{\delta_{ij}} w) L^\pm_{ia}(z) L^{\pm}_{jb}(w) + (q^{-1} - q)(z \delta_{i>j} + w \delta_{i<j})L^\pm_{ja}(z) L^{\pm}_{ib}(w) \label{eq:RLL1expl}\\
  =\, &(q^{-\delta_{ab}} z- q^{\delta_{ab}} w) L^\pm_{jb}(w) L^{\pm}_{ia}(z) + (q^{-1} - q)(z \delta_{a<b} + w \delta_{a>b})L^\pm_{ja}(w) L^{\pm}_{ib}(z)\nonumber
\end{align}
and
\begin{align*}  
   &(q^{-\delta_{ij}} z- q^{\delta_{ij}} w) L^+_{ia}(z) L^-_{jb}(w) + (q^{-1} - q)(z \delta_{i>j} + w \delta_{i<j})L^+_{ja}(z) L^-_{ib}(w)\\
  &\,=\,(q^{-\delta_{ab}} z- q^{\delta_{ab}} w) L^-_{jb}(w) L^+_{ia}(z) + (q^{-1} - q)(z \delta_{a<b} + w \delta_{a>b})L^-_{ja}(w) L^+_{ib}(z), 
\end{align*}
see also \cite[(2.42)-(2.44)]{a-MolevGow10}. In particular, if one collects coefficients of $z^0w^1$ then one obtains the relations
\begin{align}\label{eq:RTTfin1}  
   &q^{\delta_{ij}} L^{\pm}_{ia}[0] L^{\pm}_{jb}[0] - q^{\delta_{ab}} L^{\pm}_{jb}[0] L^{\pm}_{ia}[0] = (q-q^{-1})(\delta_{b<a} - \delta_{i<j}) L^{\pm}_{ja}[0] L^{\pm}_{ib}[0] 
\end{align}
and
\begin{align}\label{eq:RTTfin2}  
   q^{\delta_{ij}} L^+_{ia}[0]& L^-_{jb}[0] - q^{\delta_{ab}} L^-_{jb}[0] L^+_{ia}[0]\\ &= (q-q^{-1})   \big(\delta_{a>b} L^-_{ja}[0] L^+_{ib}[0] - \delta_{i<j} L^+_{ja}[0] L^-_{ib}[0]\big).\nonumber 
\end{align}
The algebra $\uqRglN$ is a Hopf algebra with coproduct given by
\begin{align}\label{eq:uqglhat-delta}
  \kow(L_{ij}^{\pm}(z))=\sum_{k=1}^N L_{ik}^\pm(z) \ot L_{kj}^\pm(z)
\end{align}
and antipode $S(L_{ij}^{\pm}(z))=L^{\pm}(z)^{-1}$. Consult \cite[Section 3]{a-FadResTak2}, \cite[2.3]{a-FrenkelMukhin02}, \cite[Section 3]{a-MolRagSor03}, or \cite[2.3]{a-MolevGow10} for more details.

In \cite{a-MolRagSor03} the authors define $\Yto$ to be the subalgebra of $\uqRglN$ generated by the coefficients $s_{ij}[-k]$ of the entries of the matrix 
\begin{align*}
  S(z)=(s_{ij}(z))=L^-(z) L^+(z^{-1})^t.
\end{align*}
More precisely, one considers the power series in $z^{-1}$ given by 
\begin{align*}
  s_{ij}(z)=\sum_{a=1}^N L_{ia}^-(z) L_{ja}^+(z^{-1})
\end{align*}
for $1\le i,j\le N$ and defines $s_{ij}[-k]$ by 
\begin{align*}
  s_{ij}(z)=\sum_{k=0}^\infty s_{ij}[-k]z^{-k}.
\end{align*}
Similarly, in the case $N=2m$, consider
\begin{align*}
  S(z)=L^-(z) G L^+(z^{-1})^t
\end{align*}
where $G$ denotes the $N\times N$-matrix given by
\begin{align*}
  G=q \sum_{k=1}^m E_{2k-1,2k} - \sum_{k=1}^m E_{2k,2k-1}.
\end{align*}
Molev, Ragoucy, and Sorba define $\Ytsp$ to be the subalgebra of $\uqRglN$ generated by the coefficients $s_{ij}[-k]$ of the entries of the matrix $S(z)$ and the elements $s_{i,i+1}[0]^{-1}$ for $i=1,3,\dots,2m-1$.
It follows from \eqref{eq:uqglhat-delta} that $\Yto$ and $\Ytsp$ are left coideal subalgebras of $\uqRglN$.

One can perform specialization for $\uqRglN$ in the same way as for $U_q(\hat{\gfrak})$ and $\uqLg$, see \cite[Section 3]{a-MolRagSor03}. The following statement is made in \cite[Corollaries 3.5, 3.12]{a-MolRagSor03}.
\begin{lem}\label{lem:Y-spec} The subalgebras $\Yto$ and $\Ytsp$ of $\uqRglN$ specialize to $U(L(\glfrak_N(\field))^{\hat{\theta}})$ for the involution $\theta$ corresponding to the admissible pairs of type $AI$ and $AII$, respectively. 
\end{lem}
To relate the algebras $\Yto$ and $\Ytsp$ to twisted quantum loop algebras (of the second kind) as defined in the previous subsection we use the fact that the Hopf algebra $U_q(L(\slfrak_N(\field)))$ can be embedded into $\uqRglN$.
\begin{prop}[{\cite[Lemma 3.8]{a-FrenkelMukhin02}}]\label{prop:FMinclusion}
  The following formulas define an embedding of Hopf algebras $\hI:U_q(L(\slfrak_N(\field)))\rightarrow \uqRglN$: 
  \begin{align*}
    \hI(E_0)&=(-q)^N (q^{-1}-q)^{-1} L^-_{1N}[-1] L^+_{NN}[0],\\
    \hI(F_0)&=(-q)^{-N} (q - q^{-1})^{-1} L^-_{NN}[0] L^+_{N1}[1],\\
    \hI(E_i)&= (q^{-1}-q)^{-1} L^-_{i+1,i}[0] L^+_{ii}[0],\\
    \hI(F_i)&= (q-q^{-1})^{-1} L^-_{ii}[0] L^+_{i,i+1}[0],\\
    \hI(K_i)&= L^+_{ii}[0] L^-_{i+1,i+1}[0] \qquad \qquad (i=1,\dots,N-1).
  \end{align*}
\end{prop}
In the following we will suppress the symbol $\hI$ and consider $U_q(L(\slfrak_N(\field)))$ as a Hopf subalgebra of $\uqRglN$.

For the admissible pairs of type $AI$ and $AII$ one has $\cS=\{\mathbf{0}\}$ and hence the corresponding twisted quantum loop algebras are parametrized by elements $\uc\in \cC$.
We are now in a position to establish the desired relation between the algebras $\Yto$ and $\Ytsp$ defined above and the twisted quantum loop algebras $\pi(B_{\uc})$ in the cases $AI$ and $AII$, respectively. As $\Yto$ and $\Ytsp$ are left coideal subalgebras we use the antipode $S$ to turn the right coideal subalgebra $\pi(B_{\uc})$ into a left coideal subalgebra. The following theorem is the main result of this section.
\begin{thm}\label{thm:MRS=QSP}
  (1) Let $(X,\tau)$ be of type $AI$, set 
  \begin{align*}
    \uc=(c_0,c_1,c_3,\dots,c_{N-1})=(q^{-2(N-1)},q^2,q^2,\dots,q^2),
  \end{align*}  
    and let $\pi(B_{\uc})\subset U_q(L(\slfrak_N(\field)))$ be the corresponding twisted quantum loop algebra. Then $S(\pi(B_{\uc})) = \Yto \cap U_q(L(\slfrak_N(\field)))$.
  
  (2) Let $(X,\tau)$ be of type $AII$ with $N=2m$, set 
    \begin{align*}
      \uc=(c_0,c_2,c_4,\dots,c_{N-2})=(q^{-2N+7},q^7,q^7,\dots,q^7)
    \end{align*}   
    and let $\pi(B_{\uc})\subset U_q(L(\slfrak_N(\field)))$ be the corresponding twisted quantum loop algebra.
    Then $S(\pi(B_{\uc})) = \Ytsp \cap U_q(L(\slfrak_N(\field)))$.
\end{thm}
\begin{proof}
  (1) It follows from the definition of the $\bA$-form $\uqRglN_\bA$ of $\uqRglN$ given by \cite[(3.8), (3.9)]{a-MolRagSor03} that $U_q(L(\slfrak_N(\field)))_\bA \subset \uqRglN_\bA$. Hence Lemma \ref{lem:Y-spec} implies that $\Yto\cap U_q(L(\slfrak_N(\field)))$ specializes to a subalgebra of 
  \begin{align*}
    U(L(\slfrak_N(\field))) \cap U(L(\glfrak_N(\field))^{\hat{\theta}})= U(\kfrak'_L)
  \end{align*}
where as in Subsection \ref{sec:symLoop} we write $\kfrak'_L=\{x\in L(\slfrak_N(\field))\,|\,\hat\theta(x)=x\}$. In the following we will show that
\begin{align}\label{eq:Ycontained}
  S(\pi(B_{\uc}))\subseteq \Yto \cap U_q(L(\slfrak_N(\field)))
\end{align} 
Specialization is compatible with the antipode $S$. Hence $S(\pi(B_{\uc}))$ specializes to $U(\kfrak'_L)$. By \eqref{eq:Ycontained} this implies that $\Yto \cap U_q(L(\slfrak_N(\field)))$ also specializes to $U(\kfrak'_L)$. Hence $W=S^{-1}(\Yto \cap U_q(L(\slfrak_N(\field))))$ is a subspace of $U_q(L(\slfrak_N(\field)))$ which contains $\pi(B_{\uc})$ and specializes to $U(\kfrak'_L)$. Theorem \ref{thm:maxcond2} implies that $W=\pi(B_{\uc})$.

It remains to verify the inclusion \eqref{eq:Ycontained}. To this end one uses Proposition \ref{prop:FMinclusion} and calculates
\begin{align}
  s_{1N}[-1]&\stackrel{\phantom{\eqref{eq:L0}}}{=} \sum_{a=1}^N L^-_{1a}[0] L^+_{Na}[1] + L^-_{1a}[-1] L^+_{Na}[0] \nonumber\\
     &\stackrel{\eqref{eq:L0}}{=} L^-_{11}[0] L^+_{N1}[1] + L^-_{1N}[-1] L^+_{NN}[0] \nonumber\\
    &\stackrel{\phantom{\eqref{eq:L0}}}{=} (q-q^{-1}) (-q)^N (K_0 F_0 - q^{-2N} E_0)\nonumber\\
   &\stackrel{\phantom{\eqref{eq:L0}}}{=}  -(q-q^{-1}) (-q)^{N-2} S(F_0 - q^{-2(N-1)} E_0 K_0^{-1}) \label{eq:B0contained}
\end{align}
Similarly, for $i=1,\dots, N{-}1$, one calculates
\begin{align}
  s_{i+1,i}[0]&\stackrel{\phantom{\eqref{eq:L0}}}{=} \sum_{a=1}^N L^-_{i+1,a}[0] L^+_{ia}[0]  \nonumber\\
     &\stackrel{\eqref{eq:L0}}{=} L^-_{i+1,i}[0] L^+_{ii}[0] + L^-_{i+1,i+1}[0] L^+_{i,i+1}[0] \nonumber\\
    &\stackrel{\phantom{\eqref{eq:L0}}}{=} -(q-q^{-1}) (E_i- K_i F_i)\nonumber\\
   &\stackrel{\phantom{\eqref{eq:L0}}}{=}  -(q-q^{-1}) q^{-2} S(F_i - q^{2} E_i K_i^{-1}). \label{eq:Bicontained}
\end{align}
Relations \eqref{eq:B0contained} and \eqref{eq:Bicontained} show that $S(\pi(B_i))\in \Yto$ for all generators $B_i$ of $B_{\uc}$. Hence \eqref{eq:Ycontained} is verified which completes the proof of part (1) of the theorem.   

\medskip

(2) By the same argument as in the proof of part (1) it suffices to show that
$S(\pi(B_{\uc}))\subseteq \Ytsp \cap U_q(L(\slfrak_N(\field)))$. We first show that $\cM_X$ is contained in $\Ytsp$. Indeed, for $i=1,3,\dots,2m-1$ one has
\begin{align*}
  s_{i,i+1}[0]&=q L_{ii}^-[0] L_{i+1,i+1}^+[0] = q K_i^{-1}\\
  s_{ii}[0]&= q L_{ii}^-[0] L_{i,i+1}^+[0]= q(q-q^{-1}) F_i,\\
  s_{i+1,i+1}[0]&= q L_{i+1,i}^-[0] L_{i+1,i+1}^+[0]= -q(q-q^{-1}) E_i K_i^{-1}.
\end{align*}
As the inverse of $s_{i,i+1}[0]$ is by definition also contained in $\Ytsp$ one obtains $\cM_X\subset \Ytsp$.
Without loss of generality we now restrict to the case $m=2$. Recall that
\begin{align*}
  B_2= F_2 - c_2 \ad(E_3 E_1)(E_2)K_2^{-1}, \qquad B_0= F_0 - c_0 \ad(E_1 E_3)(E_0)K_0^{-1},
\end{align*}
and define
\begin{align*}
  B_2' &= \ad(F_3)(B_2 K_2) K_2^{-1}, &  B_0' &= \ad(F_1)(B_0 K_0) K_0^{-1}. 
\end{align*}
The algebra $\pi(B_{\uc})$ is generated by $\cM_X$ and $B_0', B_2'$. Hence, it suffices to show that $S(B'_i)\in \Ytspfour$ for $i=0,2$. One calculates
\begin{align}
  S(B_2') &= q^2 [F_2,F_3]_q K_2 - c_2 q^{-2} [E_2,E_1]_{q^{-1}} K_1^{-1}, \label{eq:SB2p}\\
  S(B_0') &= q^2 [F_0,F_1]_q K_0 - c_0 q^{-2} [E_0,E_3]_{q^{-1}} K_3^{-1}  \label{eq:SB0p}
\end{align}
where $[a,b]_v=ab - v\, ba$. Proposition \ref{prop:FMinclusion} and relations \eqref{eq:RTTfin1}, \eqref{eq:RTTfin2} imply that 
\begin{align*}
  [F_2,F_3]_q&=-q(q-q^{-1})^{-1}  L_{22}^-[0] L_{24}^+[0],\\ 
  [E_2,E_1]_{q^{-1}}&=q^{-1}(q-q^{-1})^{-1}  L_{31}^-[0] L_{11}^+[0].
\end{align*}
Entering these relations into \eqref{eq:SB2p} one obtains
\begin{align*}
  S(B_2') = - \frac{q^4}{q-q^{-1}} \big( L_{33}^-[0] L_{24}^+[0]  + c_2 q^{-7}L_{31}^-[0] L_{22}^+[0]  \big). 
\end{align*}
The assumption $c_2=q^7$ gives $S(B_2')= -\frac{q^3}{ q-q^{-1}} s_{32}[0]$ and hence $S(B_2')\in \Ytspfour$.

We next determine $S(B_0')$. Collecting coefficients of $z^{-1}w$ in \eqref{eq:RLL1expl} with two minus signs as superscripts one obtains the relations
\begin{align*}
  L_{14}^-[-1] L_{43}^-[0] - L_{43}^-[0] L_{14}^-[-1] &= (q-q^{-1}) L_{13}^-[-1] L_{44}^-[0]\\
  L_{14}^-[-1] L_{33}^-[0] - L_{33}^-[0] L_{14}^-[-1] &=0.
\end{align*}
Using the above relations, the relation $L_{44}^+[0] L_{43}^-[0] = q^{-1} L_{43}^-[0] L_{44}^+[0]$, Proposition \ref{prop:FMinclusion}, and the fact that $N=4$ is even, one calculates
\begin{align}\label{eq:E0E3}
  [E_0,E_3]_{q^{-1}}= q^{N-1}(q-q^{-1})^{-1} L_{13}^-[-1] L_{33}^+[0].
\end{align}
Similarly, collecting coefficients of $z^0w^2$ in \eqref{eq:RLL1expl} with two plus signs as superscripts one obtains the relations
\begin{align*}
  L_{12}^+[0] L_{41}^+[1] - L_{41}^+[1] L_{12}^+[0] &= (q-q^{-1}) L_{42}^+[1] L_{11}^+[0]\\
  L_{11}^+[0] L_{41}^+[1] - q L_{41}^+[1] L_{11}^+[0]  &=0.
\end{align*}
With these relations and Proposition \ref{prop:FMinclusion} one gets
\begin{align}\label{eq:F0F1}
  [F_0,F_1]_q = - q^{-N+1} (q-q^{-1})^{-1} L_{44}^-[0] L_{42}^+[1].
\end{align}
Entering \eqref{eq:E0E3} and \eqref{eq:F0F1} into \eqref{eq:SB0p} one obtains
\begin{align*}
  S(B_0') = - \frac{q^{-N+4}}{q-q^{-1}} \big( L_{11}^-[0] L_{42}^+[1] + c_0 q^{2N-7} L_{13}^-[-1] L_{44}^+[0] \big).
\end{align*}
The assumption $c_0=q^{-2N+7}$ gives $S(B_2')= -\frac{q^{-N+3}}{ q-q^{-1}} s_{14}[-1]$ and hence $S(B_0')\in \Ytspfour$. Hence all generators of $\pi(B_{\uc})$ are contained in $\Ytsp$ which completes the proof of the theorem.
\end{proof}
\begin{rema}
  An FRT-construction of a family of twisted quantum loop algebras of type $AIII/IV$ was given in \cite{a-CGM14} depending on one parameter apart from $q$. It is to be expected that the analog of the above theorem holds for this family and the corresponding twisted quantum loop algebras $\pi(B_{\uc,\us})$ considered in Subsection \ref{sec:symLoop}. Moreover, for type $AIII/IV$ Theorem \ref{thm:equiv} provides a family of nonequivalent quantum symmetric pair coideal subalgebras $\pi(B_{\uc,\us})$ depending on two parameters apart from $q$. It would be interesting to discover the second parameter in the constructions of \cite{a-CGM14}.
\end{rema}
\begin{rema}
  By \cite[Section 4.1]{a-MolRagSor03} both $\Yto$ and $\Ytsp$ contain central subalgebras which are polynomial rings in infinitely many variables. To treat both cases simultaneously we write $\Yt$ to denote either of $\Yto$ and $\Ytsp$, we denote the central polynomial subalgebra in each case by $\field(q)[C_N]$, and, as before, we let $\pi(B_{\uc})$ be the corresponding twisted quantum loop algebra. By Theorem \ref{thm:Z(B)infinite} the algebra $S(\pi(B_{\uc}))$ has trivial center. It would be interesting to know if the multiplication map
  \begin{align*}
    S(\pi(B_{\uc})) \ot \field(q)[C_N] \rightarrow \Yt
  \end{align*}  
is bijective. This seems natural as 
  \begin{align*}
    L(\glfrak_N(\field))^{\hat{\theta}} \cong L(\slfrak_N(\field))^{\hat{\theta}} \oplus C_N
  \end{align*}   
where $C_N$ denotes the infinite dimensional central subspace of $L(\glfrak_N(\field))^{\hat{\theta}}$ spanned by the elements $C_k=\mathrm{Id}_N \ot t^k - \mathrm{Id}_N \ot t^{-k}$ for all $k\in \N$. 
\end{rema}
\subsection{References to Letzter's constructions}
For finite dimensional $\gfrak$, the strategy of the proof of Theorem \ref{thm:MRS=QSP} was employed in \cite[Section 6]{a-Letzter99a} to show that the coideal subalgebras constructed by Noumi and Sugitani in \cite{a-Noumi96}, \cite{a-NS95} coincide with the quantum symmetric pair coideal subalgebras constructed in \cite{a-Letzter99a}. 
\section{Quantized GIM Lie algebras}\label{sec:QGIM}
As a second application of the theory developed in this paper we construct quantum group analogs of generalized intersection matrix (GIM) Lie algebras. These Lie algebras are generalizations of Kac-Moody Lie algebras for Cartan matrices which also allow positive off-diagonal entries. They were originally introduced by Slodowy in his study of the deformation theory of singularities \cite{habil-Slodowy}, \cite{a-Slodowy86}. In Section \ref{sec:GIMdef} we define GIM Lie algebras. In Section \ref{sec:qGIM} their quantum group analogs are constructed in full generality.
Special examples of quantized GIM Lie algebras were previously constructed by Y.~Tan in \cite{a-Tan05}. The motivation of that paper was the apparent similarity of their defining relations to those of the quantum toroidal Lie algebras considered by Ginzburg, Kapranov, and Vasserot in \cite[Section 3]{a-GKV95}. 
\subsection{Generalized intersection matrix algebras}\label{sec:GIMdef}
\begin{defi}
  A generalized intersection matrix {\upshape (GIM)} is an integral matrix $A=(a_{ij})_{i,j\in I}$ for some finite set $I$ such that
  \begin{align*}
    a_{ii}&=2,&
    a_{ij}>0 &\Leftrightarrow a_{ji}>0, &
    a_{ij}<0 &\Leftrightarrow a_{ji}<0
  \end{align*}
  for all $i,j\in I$.
  A generalized intersection matrix $A$ is called symmetrizable if there exists a diagonal matrix $D=\mbox{diag}(\epsilon_i\,|\,i\in I)$ with coprime entries $\epsilon_i\in \N$ such that $DA$ is symmetric.  
\end{defi}
Associated to any generalized intersection matrix $A$ is a GIM Lie algebra $\cL(A)$ which is given explicitly in terms of generators and relations in \cite{habil-Slodowy}, \cite{a-Slodowy86}. Here, however, $\cL(A)$ will be defined using an observation by Berman \cite[Proposition 2.1]{a-Berman89} which identifies $\cL(A)$ with the Lie subalgebra of a Kac-Moody algebra fixed under an involution. To this end one associates to $A$ a generalized Cartan matrix $C(A)$ as follows.
Let $\oI$ be an identical copy of $I$ with elements denoted by $\oi$ for all $i\in I$ and define $\hat{I}=I\cup \oI$. Then $C(A)=(c_{ij})_{i,j\in \hat{I}}$ where
$c_{ii}=c_{\oi\,\oi}=2$ for all $i\in I$ and
\begin{align*}
  c_{ij}&=c_{\oi\,\oj}=a_{ij} \quad \mbox{if $a_{ij} \le 0$,}&
  c_{ij}=c_{\oi\,\oj}&=0 \qquad\,\, \mbox{if $a_{ij} > 0$ and $i\neq j$,}\\
  c_{i\oj}&=0 \quad\, \mbox{if $a_{ij} \le 0$ or $i = j$,} &
  c_{i \oj}=c_{\oi\,j}&=-a_{ij} \,\,\,\,\, \mbox{if $a_{ij} > 0$ and $i\neq j$.}   
\end{align*} 
The definition of $C(A)$ is most easily understood in terms of Dynkin diagrams as explained in \cite[Section 2]{a-Berman89}. If $A$ is symmetrizable with diagonal matrix $D$, then so is $C(A)$ with diagonal matrix $D\oplus D$. The generalized intersection matrix $A$ is called unoriented if $C(A)$ is indecomposable. By construction, there exists $\sigma\in \Aut(\hat{I})$ given by
$\sigma(i)=\oi$ and $\sigma(\oi)=i$ for all $i\in I$. As $\sigma^2=\id_{\hat{I}}$, the map $\theta=\theta(\emptyset,\sigma)=\sigma\circ\omega$ defines an involutive Lie algebra automorphism of $\gfrak(C(A))$.
\begin{defi}
  Let $A$ be an unoriented GIM and $\gfrak=\gfrak(C(A))$. The Lie subalgebra $\cL(A)$ of $\gfrak(C(A))$ fixed under the involution $\theta$ is called the GIM Lie algebra corresponding to $A$.
\end{defi}
\subsection{Quantized GIM algebras}\label{sec:qGIM}
Quantum symmetric pairs provide an immediate quantum analog of the above notion of a GIM Lie algebra. Retain the notations of the previous subsection. As the entries of $C(A)$ satisfy $c_{i\oi}=c_{\oi i}=0$ for all $i\in I$, the parameter set $\cC$ defined by \eqref{eq:C-def} for $\hat{I}$ and $(X,\tau)=(\emptyset, \sigma)$ becomes
\begin{align*}
  \cC=\{\uc\in (\field(q)^\times)^{\hat{I}}\,|\,c_i= c_{\oi} \quad\mbox{for all $i\in I$}\}=(\field(q)^\times)^I.
\end{align*}
Similarly, the set $\cS$ defined by \eqref{eq:Sdef} is trivial in this case. Hence there exist only standard quantum symmetric pair coideal subalgebras corresponding to the involution $\theta=\theta(\emptyset,\sigma)$ of $\gfrak(C(A))$.
\begin{defi}
  Let $A$ be an unoriented, symmetrizable GIM, and $\theta=\theta(\emptyset,\sigma)$ the corresponding involution of $\gfrak(C(A))$ and $\uc \in \cC=\field(q)^I$. The corresponding quantum symmetric pair coideal subalgebra $B_{\uc}$ of $U_q(\gfrak(C(A))')$ is called a quantum GIM algebra of $A$ and is denoted by $U_q(\cL(A))_{\uc}$.  
\end{defi}
By definition, $U_q(\cL(A))_{\uc}$ is hence the subalgebra of $U_q(\gfrak(C(A))')$ generated by the elements
\begin{align*}
  B_i&=F_i - c_i E_{\oi} K_i^{-1}, &  B_{\oi}&=F_{\oi} - c_{i} E_{i} K_{\oi}^{-1}, & K_i K_{\oi}^{-1}, & &K_{\oi} K_i^{-1}
\end{align*}
for all $i\in I$. By Theorem \ref{thm:rels} and Theorem \ref{thm:relsBc} it is straightforward to obtain a presentation of $U_q(\cL(A))_{\uc}$ in terms of generators and relations. Theorem \ref{thm:Cij=0if} implies that if $c_{ij}\neq 0$ for $i,j\in \hat{I}$ then $\sigma(i)\notin\{i,j\}$ and therefore $C_{ij}(\uc)=0$. Hence the quantum Serre relations for $U_q(\cL(A))_{\uc}$ can have nonzero lower order terms only if $c_{ij}=0$. 

The noncommutative polynomial $F_{ij}(x,y)$ was defined in \eqref{eq:Fij-def} only for a specified generalized Cartan matrix. To express the quantum Serre relations for $U_q(\cL(A))_{\uc}$ in terms of $A$, introduce noncommutative polynomials ${}_{-k}F_{i}$ in two variables $x,y$ defined for any $i\in I$ and any $k\in \N_0$
by
\begin{align*}
  {}_{-k}F_{i}(x,y) = \sum_{n=0}^{1+k}(-1)^n
  \left[\begin{matrix}1+k\\n\end{matrix}\right]_{q_i}
  x^{1+k-n}yx^n.   
\end{align*}
In the case that $A$ itself is a generalized Cartan matrix one hence has ${}_{a_{ij}}F_i(x,y)=F_{ij}(x,y)$.
For $\uc\in \cC$ let $\cG_q(A)_{\uc}$ denote the unital, associative $\field(q)$-algebra with generators 
  \begin{align*}
    G_i,\, \ovG_{i},\, L_i,\, \oL_{i} \qquad \mbox{for all $i\in I$}
  \end{align*}
and the following defining relations:
\begin{enumerate}
  \item $L_i \oL_i = 1 = \oL_i L_i$ \quad for $i\in I$,
  \item $L_i G_j = q_i^{-a_{ij}} G_j L_i$, \quad $L_i \ovG_j = q_i^{a_{ij}} \ovG_j L_i$ \quad for all $i,j\in I$,
  \item If $a_{ij}=0$ then $[G_i,G_j]=[\ovG_i,\ovG_j]=0$ and $\displaystyle [G_i,\ovG_j]=\delta_{ij}c_i \frac{L_i- \oL_i}{q_i-q_i^{-1}}$
  \item If $a_{ij}<0$ then ${}_{a_{ij}}F_i(G_i,G_j)=0={}_{a_{ij}}F_i(\ovG_i,\ovG_j)$ and $[G_i,\ovG_j]=0$.
  \item If $a_{ij}>0$ then 
    ${}_{-a_{ij}}F_i(G_i,\ovG_j)=0={}_{-a_{ij}}F_i(\ovG_i,G_j)$ and \\ $[G_i,G_j]=0=[\ovG_i,\ovG_j]$.
\end{enumerate}
As indicated above, the following theorem is a direct consequence of Theorems \ref{thm:rels} and \ref{thm:relsBc}.
\begin{thm}
  Let $A$ be an unoriented, symmetrizable GIM and $\uc \in \cC$. There is a uniquely determined algebra isomorphism $\varphi:\cG_q(A)_{\uc}\rightarrow U_q(\cL(A))_{\uc}$ such that
  \begin{align*}
    \varphi(L_i)&=K_i K_{\oi}^{-1}, & \varphi(G_i)&=B_i, & \varphi(\ovG_i)&= B_{\oi} & \mbox{ for all $i\in I$.}
  \end{align*}
\end{thm}
\begin{rema}
  For two-fold affinizations of Cartan matrices of type ADE a quantized GIM algebra was introduced in \cite{a-Tan05}. In this case the algebra defined in \cite[Definition 2.1]{a-Tan05} coincides with the algebra $\cG_q(A)_{\uc}$ for $\uc=(1,1,\dots,1)$ up to extension of $U^0_\Theta{}'$ via a compatible minimal realization of $C(A)$ as in Remark \ref{rem:DefComp}. In \cite[Theorem 3.1]{a-LvTan12} the above theorem is verified in this special case by explicit computation. Moreover, in \cite{a-Tan05} an action of a braid group on quantized GIM algebras associated to two-fold ADE affinizations is constructed. This braid group action can be interpreted within a general framework of braid group actions on quantum symmetric pair coideal subalgebras outlined in the finite case in \cite{a-KolPel11}. 
\end{rema}

\appendix

\section{Classification of involutions of the second kind}\label{app:classThm}
  In this appendix the statement of Theorem \ref{classThm} is reduced to results explicitly stated in \cite{a-KW92}.
\subsection{Cartan and Borel subalgebras}
Let $\gfrak$ denote an indecomposable, symmetrizable Kac-Moody algebra as in Section \ref{sec:KM} with standard Cartan subalgebra $\hfrak$ and standard Borel subalgebras $\bfrak^+=\hfrak\oplus \nfrak^+$ and $\bfrak^-=\hfrak\oplus \nfrak^-$. More generally, any maximal $\ad$-diagonalizable Lie subalgebra of $\gfrak$ is called a split Cartan subalgebra of $\gfrak$, and any maximal completely solvable Lie subalgebra of $\gfrak$ is called a Borel subalgebra of $\gfrak$ \cite[1.25, 1.26]{a-KW92}. As in the finite case, all Cartan subalgebras of $\gfrak$ are conjugate under the action of the Kac-Moody group $G$. Borel subalgebras, however, are generally either conjugate to $\bfrak^+$ or to $\bfrak^-$. More precisely, one has the following result. 
\begin{prop}[{\cite[Theorem 3]{a-PK83}, \cite[1.26]{a-KW92}}] \label{prop:split-conj}
  Let $\bfrak$ be a Borel subalgebra of $\gfrak$ and let $\tfrak$ be a Cartan subalgebra of $\gfrak$ such that $\tfrak\subset \bfrak$. There exists $g\in G$ such that $(\Ad(g)(\bfrak),\Ad(g)(\tfrak))=(\bfrak^+,\hfrak)$ or $(\Ad(g)(\bfrak),\Ad(g)(\tfrak))=(\bfrak^-,\hfrak)$.
\end{prop}
\subsection{Split pairs}\label{sec:split-pairs}
As in Section \ref{sec:inv2nd}, let $\theta:\gfrak\rightarrow \gfrak$ be an involutive automorphism of the second kind. By \cite[Lemma 5.7, Corollary 5.8 (ii) $\Rightarrow$ (i)]{a-KW92} every Borel subalgebra of $\gfrak$ contains a $\theta$-stable Cartan subalgebra $\tfrak$. As in \eqref{eq:Theta} we denote the induced map on $\tfrak^\ast$ by $\Theta$. 
\begin{defi}[{\cite[5.15]{a-KW92}}] \label{splitpair} 
  Let $\tfrak$ be a $\theta$-stable Cartan subalgebra
  of $\gfrak$ and let $\bfrak$ be a Borel subalgebra of $\gfrak$
  containing $\tfrak$. Let $\Delta$ and $\Delta^+$ denote the set of roots of
  $\tfrak$ in $\gfrak$ and $\bfrak$, respectively, and let $\Pi$
  denote the set of simple roots of $\tfrak$ in $\bfrak$.

  The pair $(\tfrak,\bfrak)$ is called a split pair for $\theta$ if
  there exists a subset $X$ of $\Pi$ satisfying the following conditions:
  \begin{enumerate}
    \item $\Delta^+\cap \Theta(\Delta^+)=\Delta^+\cap \Z X$.
    \item $\theta|_{\gfrak_\alpha}=\mathrm{id}_{\gfrak_\alpha}$ for all
    $\alpha\in \Delta^+\cap \Z X$.
  \end{enumerate}
\end{defi}
Assume that $(\tfrak,\bfrak)$ is a split pair for the involutive automorphism $\theta$ of $\gfrak$, and that $X\subset \Pi$ is as in the above definition. Let, moreover, $\gfrak_X$ denote the Lie subalgebra of $\gfrak$ generated by $\gfrak_{\pm \alpha}$ for all $\alpha \in X$. By \cite[5.16]{a-KW92} the Lie algebra $\gfrak_X$ is finite-dimensional and semisimple. Hence, by Property (2) of the above definition one has $\theta|_{\gfrak_X}=\id_{\gfrak_X}$.

The involution $\theta$ of $\gfrak$ induces an involution of $G$. Let $G^\theta$ denote the fix point subgroup and observe that $\Ad(g)\circ \theta=\theta\circ \Ad(g)$ for all $g\in G^\theta$.
The following result is the main ingredient in the proof of Theorem \ref{classThm}.
\begin{thm}[{\cite[5.19, 5.32]{a-KW92}}] \label{thm:split-exist}
Let $\theta:\gfrak \rightarrow \gfrak$ be an involutive automorphism of the second kind.
\begin{enumerate}
  \item  There exists a split pair $(\tfrak,\bfrak)$ for $\theta$.
  \item  Assume that  $(\tfrak_1,\bfrak_1)$ and
    $(\tfrak_2,\bfrak_2)$ are split pairs for $\theta$ such that $\Ad(g)(\bfrak_1)=\bfrak_2$ for some $g\in G$. Then there exists $g'\in G^\theta$ such that $\Ad(g')(\tfrak_1)=\tfrak_2$ and $\Ad(g')(\bfrak_1)=\bfrak_2$.
\end{enumerate}
\end{thm}
\subsection{Square roots}
In the following we will encounter involutive automorphisms of $\gfrak$
which we would like to commute with elements of $\Ad(\Htil)$ or
$\Aut(\gfrak,\gfrak')$. The following lemma will provide a useful
tool for this purpose. We will say that a group $H$ allows square
roots if for any $h\in H$ there exists an element $g\in H$ such that $g^2=h$.
\begin{lem}\label{CommuteLem}
  Let $\theta$ be an involutive automorphism of $\gfrak$ and $H\subset
  \Aut(\gfrak)$ a commutative subgroup which allows square roots and
  satisfies $\theta\circ H\circ \theta =H$. Then for any $h\in H$ we
  can write
  \begin{align*}
    h\circ \theta = h_\theta \circ \theta^h
   \end{align*}
  where $\theta^h\in \Aut(\gfrak)$ is $H$-conjugate to $\theta$ and $h_\theta\in H$
  commutes with $\theta_h$.
\end{lem}
\begin{proof}
  For any $k\in H$ define $\theta(k)=\theta\circ k\circ \theta\in
  H$. Given $g,h\in H$ with $g^2=h$ define $\theta^h:=\theta(g^{-1})\circ
  \theta \circ \theta(g)$ and $h_\theta:= g\circ\theta(g)=\theta(g)\circ g$. Then
  \begin{align*}
    h_\theta \circ \theta^h=g\circ \theta(g)\circ \theta(g^{-1})\circ
    \theta\circ \theta(g)= g \circ \theta\circ \theta(g)=h\circ
    \theta
  \end{align*}
  and on the other hand 
  \begin{align*}
    \theta^h\circ h_\theta= \theta(g^{-1})\circ \theta\circ
    \theta(g)\circ \theta(g)\circ g=\theta(g^{-1})\circ h\circ
    \theta(g)\circ \theta=h\circ \theta
  \end{align*}
  which proves the lemma.
\end{proof}
\subsection{Proof of Theorem \ref{classThm}}\label{sec:inv2}
Let $\theta:\gfrak\rightarrow \gfrak$ be an involutive automorphism of
the second kind. By Theorem \ref{thm:split-exist}.(1) there exists a split
pair $(\tfrak,\bfrak)$ for $\theta$. Conjugating $\theta$ by an
element of $\Ad(G)$ and using Proposition \ref{prop:split-conj} we may
assume that $(\tfrak,\bfrak)=(\hfrak,\bfrak^+)$ or
$(\tfrak,\bfrak)=(\hfrak,\bfrak^-)$. Note that $(\hfrak,\bfrak^+)$ is
a split pair for $\theta$ if and only if $(\hfrak,\bfrak^-)$ is a
split pair for $\theta$. Hence, we may assume that $(\tfrak,\bfrak)=(\hfrak,\bfrak^+)$.
\begin{rema}
  Note that the subset $X\subseteq \Pi$ obtained in this way is
  uniquely determined by the original involution $\theta$. Indeed, given
  $g\in G$ with $\Ad(g)(\hfrak)=\hfrak$ and $\Ad(g)(\bfrak)=\bfrak$, then
  $g\in N_G(\hfrak)$ represents an element $w$ in the Weyl group $W$
  which satisfies $w(\Delta^+)=\Delta^+$ and hence $w(\Pi)=\Pi$. But this
  implies that $w$ is the unit element in $W$ \cite[3.11 b)]{b-Kac1}.
\end{rema}
As $(\tfrak,\bfrak)=(\hfrak,\bfrak^+)$ we may identify $\Pi$ with the index set $I$ and consider $X$ as a subset of $I$. As observed before Theorem \ref{thm:split-exist} the subset $X$ thus obtained is of
finite type. Hence we may use notation and results from Subsection \ref{sec:Weyl}. 
The automorphism $\Ad(m_X)$ commutes with $\theta$ because $\Ad(e_i)$
and $\Ad(f_i)$ do so for all $i\in X$. Hence the map  $-w_X\circ
\Theta$ on $\hfrak^\ast$, which is induced by $\omega \circ \Ad(m_X)\circ\theta$, satisfies
\begin{align}\label{square=id}
  (- w_X \circ \Theta)^2=\id.
\end{align}
By Proposition \ref{Rousseau-Prop} the automorphism $\omega \circ \Ad(m_X) \circ \theta \in
\Aut(\gfrak)$ leaves $(\hfrak,\bfrak^+)$ and $\gfrak_X$ invariant. Hence
there exists $\tau\in \Aut(A,X)$ such that
$\tau\circ\omega\circ Ad(m_X)\circ \theta$ induces the trivial action
on $\Delta$. Hence $-\tau\circ w_X=\Theta$ and \eqref{square=id} implies
\begin{align}\label{eq:tau2=1}
  \tau^2=\id. 
\end{align}
As $\tau\circ\omega\circ \Ad(m_X)\circ \theta$ induces the trivial action
on $\Delta$ we can find $s\in \Htil$ such that 
\begin{align}\label{withs}
  \Ad(s)\circ \tau \circ \omega \circ \Ad(m_X)\circ \theta
  |_{\gfrak'}=\id _{\gfrak'}.
\end{align}
By Proposition \ref{Rousseau-Prop}.(1) we have
\begin{align}\label{eq:tauommX}
  \tau\circ\omega\circ \Ad(m_X)|_{\gfrak_X}=\id_{\gfrak_X}.
\end{align}
This relation, together with  $\theta|_{\gfrak_X}=\id_{\gfrak_X}$ from Subsection \ref{sec:split-pairs}, implies that the restriction $s|_{Q_X}\equiv 1$ is the constant function. Hence $\Ad(s)$
commutes with $\Ad(m_X)$. Moreover, as $\Ad(\Htil)$ is a commutative
subgroup of $\Aut(\gfrak)$ which is invariant under conjugation by
$\theta$, we may apply Lemma \ref{CommuteLem} and assume that $\Ad(s)$
and $\theta$ commute. By (\ref{withs}) this assumption implies that $\Ad(s)$ and
$\tau\circ\omega\circ \Ad(m_X)$ commute and therefore $\Ad(s)$ also
commutes with $\tau\circ \omega$. Hence we obtain
\begin{align}\label{scond}
  s(\alpha_i)=s(-\alpha_{\tau(\alpha_i)})=s(\alpha_{\tau(\alpha_i)})^{-1}
\end{align}
for all $i\in I$. This implies in particular that $s(\alpha_i)^2=1$ if
$\tau(i)=i$. Now we use $\theta^2=\id_\gfrak$ as well as the
commutativity from above and from Proposition \ref{Rousseau-Prop}.(3)
to obtain the relation 
\begin{align}\label{square-property}
  (\Ad(s))^2\circ (\Ad(m_X))^2|_{\gfrak'}= (\Ad(s)\circ \tau \circ \omega \circ
  \Ad(m_X))^2|_{\gfrak'}= \id_{\gfrak'}
\end{align}
Using Proposition \ref{Rousseau-Prop}.(2) one obtains from
\eqref{scond} and \eqref{square-property} the relation
\begin{align}\label{inZ}
  \alpha_i(\rho^\vee_X) \in \Z \qquad \mbox{if $\tau(i)=i$.}
\end{align}
Relations \eqref{eq:tau2=1}, \eqref{eq:tauommX}, and \eqref{inZ} imply that $(X,\tau)$ is an admissible pair. Moreover, $(\Ad(s)\circ \tau \circ \omega \circ
\Ad(m_X))^2|_\hfrak=\id_\hfrak$ and hence \eqref{square-property} implies 
\begin{align}\label{square=id2}
  \big( \Ad(s)\circ \tau \circ \omega \circ \Ad(m_X)\big)^2 =\id_\gfrak
\end{align}
Relation \eqref{withs} implies that there exists $\phi\in
\Aut(\gfrak;\gfrak')$ such that
\begin{align*}
  \theta=\phi\circ \Ad(s)\circ \tau \circ \omega \circ \Ad(m_X).
\end{align*}
As $\Aut(\gfrak;\gfrak')\subset \Aut(\gfrak)$ is a commutative subgroup
which allows square-roots (compare \cite[4.20]{a-KW92} for an explicit
description of $\Aut(\gfrak;\gfrak')$) we may apply Lemma
\ref{CommuteLem} and assume that $\phi$ commutes with $\theta$. But
then relation \eqref{square=id2} implies $\phi^2=\id$ which by
\cite[4.20]{a-KW92} is impossible if $\mbox{char}(\field)=0$ unless $\phi=\id_\gfrak$. The following proposition collects the result of the discussion up to this point.
\begin{prop}\label{half-way}
  Let $\theta:\gfrak\rightarrow \gfrak$ be an involutive automorphism
  of the second kind. Then there exists an admissible pair $(X,\tau)$ such that $\theta$ is $\Aut(\gfrak)$-conjugate to
  \begin{align*}
    \theta(X,\tau,s):=\Ad(s)\circ\tau\circ \omega\circ \Ad(m_X)
  \end{align*}
  for some $s\in \Htil$ such that $\Ad(s)$ commutes with $\theta(X,\tau,s)$.
\end{prop}
  By the discussion preceding Proposition \ref{half-way} we know that
  $s|_{Q_X} \equiv 1$ and that the relations
  $s(\alpha_j)=s(\alpha_{\tau(j)})^{-1}$ and
  $s(\alpha_j)^2=(-1)^{\alpha_j( 2\rho_X^\vee)}$ 
  hold for all $j\in I\setminus X$. Note that these relations imply
  $s(\alpha_j)=\pm s(X,\tau)(\alpha_j)$ for all $j\in I$.
  Define $u\in \Htil$ by
  \begin{align*}
    u(\alpha_j):=\begin{cases}
        1& \mbox{if $s(\alpha_j)=s(X,\tau)(\alpha_j)$,}\\
        i& \mbox{if $s(\alpha_j)=-s(X,\tau)(\alpha_j)$.} 
      \end{cases}
  \end{align*}
  Then $\theta(X,\tau)=\Ad(u)\circ \theta(X,\tau,s) \circ (\Ad(u))^{-1}$
  which proves that $\theta$ is conjugate to an automorphism of the
  form $\theta(X,\tau)$ as given by \eqref{tauDef}. 
  
  It is straightforward to check that if two admissible pairs belong
  to the same $\Aut(A)$-orbit, then the corresponding involutive
  automorphisms are $\Aut(\gfrak)$ conjugate. Hence it remains to
  show that 
  \begin{align}\label{conj}
    \theta(X,\tau)=\phi\circ  \theta(X',\tau')\circ \phi^{-1}
  \end{align}
  for some $\phi\in \Aut(\gfrak')$ implies that $(X,\tau)$ and
  $(X',\tau')$ belong to the same $\Aut(A)$-orbit. Indeed, note 
  that in this case both $(\hfrak, \bfrak)$ and $(\phi(\hfrak),
  \phi(\bfrak))$ are split pairs for $\theta:=\theta(X,\tau)$. Moreover, replacing $\phi$ by $\phi\circ \omega$ if necessary, on may assume that $\phi(\bfrak)$ and $\bfrak$ are $\Ad(G)$-conjugate. By Theorem \ref{thm:split-exist}.(2) there exists $g\in G^\theta$ such that
  $\phi(\hfrak)=\Ad(g)(\hfrak)$ and
  $\phi(\bfrak)=\Ad(g)(\bfrak)$. Hence, replacing $\phi$ by
  $\Ad(g^{-1})\circ \phi$, we may assume that 
  \begin{align}\label{phihb}
    \phi(\hfrak)=\hfrak \quad \mbox{ and } \quad\phi(\bfrak)=\bfrak. 
  \end{align}
  Let $\Theta, \Theta'$, and $\Phi$ denote the automorphisms of
  $\hfrak^\ast$ induced by $\theta, \theta':=\theta(X',\tau')$, and
  $\phi$, respectively. Relation \eqref{phihb} implies $\Phi\in
  \Aut(A)$ and by relation \eqref{conj} one has
  $\Theta=\Phi\circ \Theta'\circ \Phi^{-1}$. Thus $\Phi(X')=X$ and
  \begin{align*}
  \tau'=-w_{X'}\circ \Theta'=-\Phi^{-1}\circ w_X\circ \Phi
  \circ \Phi^{-1}\circ\Theta \circ \Phi=\Phi^{-1}\circ \tau\circ \Phi
  \end{align*}
  which shows that $(X,\tau)$ and $(X',\tau')$ belong to the same
  $\Aut(A)$-orbit. This completes the proof of Theorem \ref{classThm}.

\providecommand{\bysame}{\leavevmode\hbox to3em{\hrulefill}\thinspace}
\providecommand{\MR}{\relax\ifhmode\unskip\space\fi MR }
\providecommand{\MRhref}[2]{%
  \href{http://www.ams.org/mathscinet-getitem?mr=#1}{#2}
}
\providecommand{\href}[2]{#2}

\end{document}